\documentclass[12pt]{amsart}
\usepackage[margin=3cm]{geometry}
\usepackage{graphicx} 
\usepackage{adjustbox}
\usepackage{amsmath,amssymb,amsxtra,tikz,stackengine,bm,mathrsfs,bbm,tikz-cd,comment,hyperref,combelow,amsopn,float}
\usepackage{enumitem}
\usepackage{ytableau}
\usepackage{quiver}
\usepackage{marvosym}
\usetikzlibrary{calc}%
\usetikzlibrary{shapes}%
\usetikzlibrary{patterns}%
\usetikzlibrary{positioning}%
\usetikzlibrary{arrows.meta}
\usetikzlibrary{knots}
\usetikzlibrary{decorations.markings}
\usetikzlibrary{hobby}
\usepackage{relsize}

\numberwithin{equation}{section}

\newtheorem{theorem}{Theorem}[section]
\newtheorem{proposition}[theorem]{Proposition}
\newtheorem{corollary}[theorem]{Corollary}
\newtheorem{lemma}[theorem]{Lemma}

\theoremstyle{definition}
\newtheorem{remark}[theorem]{Remark}

\newtheorem{definition}[theorem]{Definition}
\newtheorem{example}[theorem]{Example}
\newtheorem{conjecture}[theorem]{Conjecture}

\newcommand{\C}{\mathbb{C}}

\newcommand{\Z}{\mathbb{Z}}

\newcommand{\seeds}{\mathsf{mut}}
\definecolor{babyblueeyes}{rgb}{0.63, 0.79, 0.95}

\newcommand{\Fl}{\mathcal{F}\ell}
\newcommand{\CF}{\mathcal{F}}
\newcommand{\CG}{\mathcal{G}}

\newcommand{\F}{\mathcal{F}}

\newcommand{\diag}{\mathrm{diag}}

\newcommand{\relpos}[1]{\xrightarrow{#1}}
\newcommand{\wrelpos}[1]{\buildrel #1 \over \dashrightarrow}
\newcommand{\dem}{\star}

\newcommand{\std}{\mathrm{std}}
\newcommand{\ant}{\mathrm{ant}}
\def\br{\beta}
\def\Br{\mathrm{Br}}
\def\BS{\mathrm{BS}}
\def\minor{\Delta}
\def\gl{\mathfrak{gl}}
\def\GL{\mathrm{GL}}
\def\last{\mathrm{last}}
\def\mut{\mathrm{Mut}}
\def\seed{\Sigma}

\def\colour{\mathrm{clr}}
\def\xun{x^{(1)}}
\def\x2{x^{(2)}}
\def\phipul{\Phi_{r_1}^{\ast}}
\def\u1{u^{(1)}}
\def\w0{\Delta}
\def\Br{\mathrm{Br}}

\def\bsbra{\varphi_1}
\def\bsbrb{\varphi_2}
\def\longest{\Delta}
\newcommand{\Wop}{W^{\mathrm{op}}}

\newcommand{\zz}{\mathbf{z}}
\newcommand{\Ext}{\mathrm{Ext}}

\newcommand{\yy}{\mathbf{y}}
\newcommand{\pp}{\mathbf{p}}
\newcommand{\CU}{\mathcal{U}}

\newcommand{\Gr}{\mathrm{Gr}}
\newcommand{\cox}{\mathbf{c}}
\newcommand{\id}{\mathbf{1}}
\newcommand{\borel}{\mathsf{B}}
\newcommand{\unipotent}{\mathsf{U}}

\def\ltrt{\vec{\mathbb{T}}}

\newcommand{\HHH}{\mathrm{HHH}}

\def\QO{Q_{\br^1}\sqcup Q_{\br^2}}
\def\QP{\widehat{Q}_{\br}}

\newcommand{\lift}[1]{\underline{#1}}
\newcommand{\subs}[1]{\langle #1 \rangle}

\newcommand{\newword}[1]{\emph{\textbf{#1}}}

\newcommand{\TS}[1]{{\color{purple}{{\bf Tonie}: #1}}}

\newcommand{\EG}[1]{{\color{red}{{\bf Eugene}: #1}}}

\definecolor{aquamarine}{rgb}{0.5, 1.0, 0.83}
\definecolor{aqua}{rgb}{0.0, 1.0, 1.0}

\title{Splicing braid varieties}
\author[E. Gorsky]{Eugene Gorsky}
\address{Department of Mathematics, University of California Davis\\ One Shields Avenue, Davis CA 95616}
\email{egorskiy@ucdavis.edu}
\author[S. Kim]{Soyeon Kim}
\address{Department of Mathematics, University of California Davis\\ One Shields Avenue, Davis CA 95616}
\email{syxkim@ucdavis.edu}
\author[T. Scroggin]{Tonie Scroggin}
\address{Department of Mathematics, University of California Davis\\ One Shields Avenue, Davis CA 95616}
\email{tmscroggin@ucdavis.edu}
\author[J. Simental]{Jos\'e Simental}
\address{Instituto de Matem\'aticas, Universidad Nacional Aut\'onoma de M\'exico. Ciudad Universitaria, CDMX,  M\'exico}
\email{simental@im.unam.mx}
\date{}

\begin{document}

\begin{abstract}
For a positive braid $\beta \in \mathrm{Br}^{+}_{k}$, we consider the braid variety $X(\beta)$. We define a family of open sets $\mathcal{U}_{r, w}$ in $X(\beta)$, where $w \in S_k$ is a permutation and $r$ is a positive integer no greater than the length of $\beta$. For fixed $r$, the sets $\mathcal{U}_{r, w}$ form an open cover of $X(\beta)$. We conjecture that $\mathcal{U}_{r,w}$ is given by the nonvanishing of some cluster variables in a single cluster for the cluster structure on $\C[X(\beta)]$ constructed in \cite{CGGLSS, GLSB, GLSBS22} and that $\mathcal{U}_{r,w}$ admits a cluster structure given by freezing these variables. Moreover, we show that $\mathcal{U}_{r, w}$ is always isomorphic to the product of two braid varieties, and we conjecture that this isomorphism is quasi-cluster. In some important special cases, we are able to prove our conjectures.
\end{abstract}
\maketitle



\section{Introduction}


In this paper, we study splicing maps between braid varieties, generalizing the constructions of \cite{GKSS,GS}. 

\subsection{Splicing braid varieties}
 
We consider the positive braid monoid $\Br^{+}_{k}$. For an element $\br\in\Br^{+}_{k}$, the \emph{braid variety} $X(\br)$ is a smooth affine algebraic variety defined in terms of configurations of complete flags in the $k$-dimensional space $\C^k$, see Section \ref{sec:braid and dbs} for a precise definition. In recent years, braid varieties have gained considerable attention due to their connections to character varieties, Legendrian link invariants, cluster algebras, and link homology. Additionally, many classical varieties appearing in geometric representation theory, such as open Richardson varieties or double Bruhat cells, are isomorphic to braid varieties.

 From now on, we will assume that $\br$ contains a reduced expression for the longest element $w_0\in S_k$ as a (not necessarily consecutive) subword (by \cite[Lemma 3.4]{CGGLSS}, this is without loss of generality).  
 We consider a decomposition  $\br = \br^1\br^2$ 
 of $\br$ as a product of positive braids. 
A \newword{splicing map} is an open embedding
\[
X\left(\widetilde{\br}^1\right) \times X\left(\widetilde{\br}^2\right) \to X(\br),
\]
where $\widetilde{\br}^1$, $\widetilde{\br}^2$ are braids explicitly determined by $\br^1, \br^2$ and some auxiliary data from the splicing map. Motivated by their applications to the study of the Khovanov-Rozansky homology of torus links, splicing maps are studied by the first and third authors in \cite{GS, Scroggin} in the special case of the top positroid variety in the Grassmannian $\Gr(k,n)$. This construction was later generalized by the authors in \cite{GKSS} to the case of \emph{skew shaped positroids}, i.e., those positroids that are open in the Grassmannian Richardson variety that contains them. In a different guise, another example of a splicing map is given by \cite[Definition 3.1]{eberhardt-stroppel}, where these maps are used to construct a convolution product on the compactly supported cohomology of Grassmannian Richardson varieties. Our first construction is a common generalization of these.

Assume that $\br = \sigma_{i_1}\cdots \sigma_{i_r}$, where $\sigma_1, \dots, \sigma_{k-1}$ are the usual simple generators of the positive braid monoid $\Br^{+}_{k}$. For an element $u \in S_k$, we denote by $\lift{u} \in \Br_{k}^{+}$ 
its braid lift of minimal length. For each $r_1 = 1, \dots, r$, we let $\beta^1 = \beta^1(r_1) = \sigma_{i_1}\cdots \sigma_{i_{r_1}}$ and $\br^2 = \br^2(r_1) = \sigma_{i_{r_1+1}}\cdots \sigma_{i_r}$, so that $\br = \br^1\br^2$. For each element $w \in S_k$, we define a principal open set $\CU_{r_1,w} \subseteq X(\br)$ (see \eqref{eq:def-U}) and prove the following. 


\begin{theorem}\label{thm:main-braid-varieties}
For each $r_1 = 1, \dots, r$ and $w \in S_k$, we have an isomorphism of algebraic varieties
\begin{equation}\label{eq:main-braid-varieties}
\Psi_{r_1, w}: X\left(\lift{(w^{-1}w_0)}\br^1\right) \times X\left(\br^2\lift{w}\right) \xrightarrow{\cong} \CU_{r_1,w}.
\end{equation}
where $w_0$ is the longest element of $S_k$.
\end{theorem}

\begin{remark}
    For simplicity, we choose to state and prove all the results in this paper only in type $A$. However, it is not hard to see from the proof (see Section \ref{sec:splicing-braid} below) that Theorem \ref{thm:main-braid-varieties} holds in arbitrary type. 
\end{remark}

More precisely, the braid variety $X(\br)$ is defined via chains of flags with specified relative positions, and the open set $\CU_{r_1, w}$ is given by the condition that the $r_1$-th flag is transverse to the coordinate flag $\CF(w_0w)$. Schematically, the inverse $\Phi_{r_1, w}$ of the map $\Psi_{r_1, w}$ is given as follows: 

\begin{equation}\label{eq:diagram-intro}
\begin{tikzpicture}
\node at (0,0) {$\CF^{\std}$};
\draw[->] (0.5,0) to (1.5,0);
\node at (1,0.15) {\scriptsize $s_{i_1}$};
\node at (2,0) {$\CF^{1}$};
\draw[->] (2.5,0) to (3.5,0);
\node at (3,0.15) {\scriptsize $s_{i_2}$};
\node at (4,0) {$\cdots$};
\draw[->] (4.5,0) to (5.5,0);
\node at (4.75,0.15) {\scriptsize $s_{i_{r_1}}$};
\node at (6,0) {$\CF^{r_1}$};
\draw[->] (6.5,0) to (7.5,0);
\node at (6.95,0.15) {\scriptsize $s_{i_{r_1+1}}$};
\node at (8,0) {$\cdots$};
\draw[->] (8.5,0) to (9.5,0);
\node at (9,0.15) {\scriptsize $s_{i_{r}}$};
\node at (10,0) {$\CF^{\ant}$};
\draw[double equal sign distance] (10, -0.1) to (10, -0.9);
\node at (10,-1.3) {$\CF^{\ant}$};
\draw[->] (9.5, -1.3) to (8.5, -1.3);
\node at (9,-1.15) {\scriptsize $s_{a_1}$};
\node at (8,-1.3) {$\cdots$};
\draw[->] (7.5, -1.3) to (6.8, -1.3);
\node at (7.1,-1.15) {\scriptsize $s_{a_{\ell(w)}}$};
\node at (6,-1.3) {$\CF(w_0w)$};
\draw[->] (5.2, -1.3) to (4.5, -1.3);
\node at (4,-1.3) {$\cdots$};
\draw[->] (3.5, -1.3) to (2.8, -1.3);
\node at (2,-1.3) {$\widetilde{\CF}^{\ell(w_0)-1}$};
\draw[->] (1.3, -1.3) to (0.5, -1.3);
\node at (0.9,-1.15) {\scriptsize $s_{a_{\ell(w_0)}}$};
\node at (0, -1.3) {$\CF^{\std}$};
\draw[double equal sign distance] (0, -0.1) to (0, -0.9);

\draw[dashed, red] (10.5, 0.5) to (5,0.5) to (5, -1.7) to (10.5, -1.7);
\node at (10.7, -0.6) {\color{red}$\Phi^2$};

\draw[dashed, blue] (-0.5, 0.7) to (7, 0.7) to (7, -2) to (-0.5, -2);
\node at (-0.7, -0.6) {\color{blue}$\Phi^1$};
\end{tikzpicture}
\end{equation}

In \eqref{eq:diagram-intro}, the flags on the bottom row are all coordinate flags, and their successive relative positions spell a reduced word for $w_0$. If $\CF^{r_1}$ is transverse to $\CF(w_0w)$, then the blue part of the diagram belongs, up to an overall shift, to $X\left(\left(\lift{w^{-1}w_0}\right)\br^1\right)$. Similarly, the red part of \eqref{eq:diagram-intro} belongs to $X\left(\br^2\lift{w}\right)$. 


Note that it may be that $\CU_{r_1,w} = \emptyset$, in which case the product on the left of \eqref{eq:main-braid-varieties} is also empty. This phenomenon can be easily understood using the notion of the Demazure product (see Section \ref{sec: background}). The open set $\CU_{r_1,w}$ is not empty precisely when both Demazure products $\delta\left(\lift{(w^{-1}w_0})\br^1\right)$ and $\delta\left(\br^2\lift{w}\right)$ are equal to $w_0$. 

One advantage of the maps \eqref{eq:main-braid-varieties} is that, for fixed $r_1 = 1, \dots, r$ we have:
\[
\bigcup_{w \in S_k}\CU_{r_1,w} = X(\br),
\]
i.e., the sets $\CU_{r_1,w}$ form a cover of $X(\br)$ by open subsets which are themselves isomorphic to products of braid varieties for simpler braids. 

A disadvantage is that the properties of the map \eqref{eq:main-braid-varieties} remain mysterious at the moment, especially those regarding the relationship between \eqref{eq:main-braid-varieties} and the cluster structure on $X(\br)$ constructed in \cite{CGGLSS, GLSB, GLSBS22}. The following surprising inequality is an easy consequence of Theorem \ref{thm:main-braid-varieties} (see Remark \ref{rem: frozen inequality}).

\begin{corollary}
\label{cor: frozen inequality intro}
Let $f, f_1,f_2$ respectively denote the numbers of frozen variables for $X(\beta),X\left(\lift{(w^{-1}w_0)}\br^1\right)$ and $X\left(\br^2\lift{w}\right)$. Then we have an inequality
\begin{equation}
\label{eq: frozen inequality intro}
f_1+f_2\ge f.
\end{equation}
\end{corollary}
More precisely, we formulate the following conjecture relating the cluster structures.

\begin{conjecture}\label{conj:braid-varieties}
For each $r_1 = 1, \dots, r$ and $w \in S_k$ such that $\CU_{r_1,w} \not= \emptyset$, there exists a seed $\Sigma = (Q, \mathbf{x})$ in $\C[X(\br)]$ with cluster variables $x_{a_1}, \dots, x_{a_s} \in \mathbf{x}$ such that:
\begin{itemize}
\item[(a)] The open set $\CU_{r_1,w}$ is the common non-vanishing locus of $x_{a_1}, \dots, x_{a_s}$.
\item[(b)] The variety $\CU_{r_1,w}$ is the cluster variety associated to the seed obtained from $\Sigma$ upon freezing the cluster variables $x_{a_1}, \dots, x_{a_s}$.
\item[(c)] The map \eqref{eq:main-braid-varieties} is a cluster quasi-isomorphism. 
\end{itemize}
\end{conjecture}

In particular, Conjecture \ref{conj:braid-varieties}(b) ensures that $\CU_{r_1,w}$ is a cluster chart in $X(\beta)$ in the sense of Muller \cite{Muller}. The open set $\CU_{r_1,w}$ is defined in $X(\beta)$ by the non-vanishing of  minors $\Delta_{ww_0[i],[i]}(M_{r_1}), i=1,\ldots,k$ where $M_{r_1}$ is a certain matrix related to the flag $\CF^{r_1}$ in \eqref{eq:diagram-intro}. Conjecture \ref{conj:braid-varieties}(b) then implies that these minors are cluster monomials in the seed $\Sigma$, and their irreducible factors (up to monomials in frozen variables) are precisely $x_{a_1}, \dots, x_{a_s}$. The number $s$ of cluster variables that one needs to freeze in Conjecture \ref{conj:braid-varieties} equals 
$$
s=f_1+f_2-f
$$
which is nonnegative by Corollary \ref{cor: frozen inequality intro}. See Sections \ref{sec:conj-properties} and \ref{sec:richardson} for more details and examples. Also, see Lemma \ref{lem:conditional-conjecture} for more details and dependencies between (a)-(c). 



\subsection{Splicing open Richardson varieties} Open Richardson varieties are smooth, affine subvarieties of the flag variety, given by specifying relative positions with respect to the standard flag and the antistandard flag, see Section \ref{sec:richardson} for details. For $v \leq w \in S_k$, the open Richardson variety is
\[
R(v,w) := \{\CF \in \Fl(k) \mid \CF^{\std} \xrightarrow{w} \CF \xrightarrow{v^{-1}w_0} \CF^{\ant}\},
\]
see Section \ref{sec: background} for details on relative position and unexplained notation. It is known, see e.g. \cite{CGGLSS, GLSBS22, GLSB} and Section \ref{sec:richardson} below, that open Richardson varieties are special cases of braid varieties. Specializing Theorem \ref{thm:main-braid-varieties} to this setting, we obtain the following result.

\begin{theorem}\label{thm:richardson-intro}
For $u \leq v \leq w \in S_k$, define the set
\[
\CU_{u,v,w} := \left\{\CF \in R(u,w) \mid \CF \xrightarrow{w_0} \CF(vw_0)\right\}.
\]
Then, $\CU_{u,v,w}$ is principal open in $R(u,w)$ and
\[
\CU_{u,v,w} \cong R(u,v) \times R(v,w). 
\]
\end{theorem}

Note that Theorem \ref{thm:richardson-intro} implies that, if $f_{v,w}$ denotes the number of frozen variables in $R(v,w)$, then we have the inequality $f_{u,v} + f_{v,w} \geq f_{u, w}$. The problem of finding a combinatorial rule to compute the quantity $f_{v,w}$ is an interesting one. In Section \ref{sec:richardson}, we apply Theorem \ref{thm:richardson-intro} to this problem. In particular, we find that $\max\{f_{v,w} \mid v \leq w \in S_k\} - k+1$ grows at least linearly in $k$, see Remark \ref{rmk:growth-frozens}.

We can identify $R(u,w)$ with a locally closed subset of the affine space $\C^{\ell(w)}$, by identifying an explicit matrix form of all flags $\CF$ such that $\CF^{\std} \xrightarrow{w} \CF$. It is a hard problem to determine which minors are cluster monomials. If true, Conjecture \ref{conj:braid-varieties} applied to the Richardson setting would imply that, for every $u \leq v \leq w$, the minors $\minor_{v[i], [i]}$ are cluster monomials in $R(u,w)$ for all $i = 1, \dots, k$. Note that it was recently shown in \cite[Theorem C]{MMSBV} that for open positroid varieties in the Grassmannian $\Gr(r,k)$ \emph{all} non-zero Pl\"ucker coordinates are cluster monomials. Viewed in the Richardson setting, the Pl\"ucker coordinates correspond to minors of the form $\Delta_{I, [r]}$, where $I \subseteq [k]$ is an $r$-element set. 

One can also iterate Theorem \ref{thm:richardson-intro} to obtain the following:
\begin{corollary}
Suppose that $u = v_0 < v_1 < v_2 < ... < v_{\ell} = w$ is a maximal chain from $u$ to $w$ in the Bruhat poset. Then we have an open embedding
\begin{equation}
\label{eq: richadson chain torus}
(\C^{\times})^{\ell}\simeq R(v_0,v_1)\times R(v_1,v_2)\times \cdots \times R(v_{\ell-1},v_{\ell})\hookrightarrow R(u,w).
\end{equation}
\end{corollary}
Here we used the fact that one-dimensional Richardson varieties $R(v_{i-1},v_i)$ are isomorphic to $\C^{\times}$. If true, Conjecture \ref{conj:braid-varieties}   would imply that all tori \eqref{eq: richadson chain torus} are cluster tori in $R(u,w)$.

\subsection{Splicing double Bott--Samelson varieties} While we do not have a proof of Conjecture \ref{conj:braid-varieties} in full generality, in some cases we are able to prove it, namely, those cases where we start with a \emph{double Bott--Samelson variety}. For each positive braid $\br \in \Br^{+}_{k}$, the double Bott--Samelson variety $\BS(\br)$ is defined as a space of flag configurations dictated by the braid $\br$ in a way similar to, but subtly different from, the definition of a braid variety. In fact, we have isomorphisms
\begin{equation}\label{eq:iso-bs-braid-intro}
\bsbra: \BS(\br) \to X(\br\Delta), \qquad \bsbrb: \BS(\br) \to X(\Delta\br)
\end{equation}
where $\longest$ is a positive braid lift of the longest element $w_0 \in S_k$. Double Bott--Samelson varieties were introduced by Elek and Lu in \cite{EL21}, and a cluster structure on them was studied in \cite{GY06b, SW}. Given a braid decomposition $\br = \br^1\br^2$ as above, we define an open set $\CU_{r_1}=\CU(\br^1, \br^2) \subseteq \BS(\br)$ and show the following. 

\begin{theorem}\label{thm: main-splicing-dbs-intro}
Let $\br = \br^1\br^2 \in \Br^{+}_{k}$ be a positive braid. Then, we have an isomorphism:
\begin{equation}\label{eq:BS-splicing-intro}
\Psi_{r_1}: \BS(\br^1) \times \BS(\br^2) \xrightarrow{\cong} \CU_{r_1}.
\end{equation}
Moreover, the map \eqref{eq:BS-splicing-intro} and the open set $\CU_{r_1}$ satisfy the properties predicted by Conjecture \ref{conj:braid-varieties}.
\end{theorem}

We remark that, upon some identifications made possible by \eqref{eq:iso-bs-braid-intro}, the map \eqref{eq:BS-splicing-intro} is a special case of the maps \eqref{eq:main-braid-varieties}. Recall that we denote by $\Phi_{r_1, w}$ the inverse to $\Psi_{r_1, w}$. Similarly, denote by $\Phi_{r_1}$ the inverse to $\Psi_{r_1}$ from \eqref{eq:BS-splicing-intro}. 

\begin{theorem}\label{thm:comparison-braid-dbs-intro}
    Let $\br = \br^1\br^2$ be a positive braid. The following diagrams commute:
    \[
    \begin{tikzcd}
\BS(\br) \supseteq \CU_{r_1} \arrow{rr}{\Phi_{r_{1}}} \arrow{d}{\bsbra} & & \BS(\br^1) \times \BS(\br^2) \arrow{d}{\bsbrb \times \bsbra} \\
X(\br\Delta) \supseteq \CU_{r_1, e} \arrow{rr}{\Phi_{r_{1}, e}} & & X(\Delta\br^1) \times X(\br^2\Delta)
    \end{tikzcd}
    \]
    \[
    \begin{tikzcd}
    \BS(\br) \supseteq \CU_{r_1} \arrow{rr}{\Phi_{r_{1}}} \arrow{d}{\bsbrb} & & \BS(\br^1) \times \BS(\br^2) \arrow{d}{\bsbrb \times \bsbra} \\
X(\Delta\br) \supseteq \CU_{r_1+\ell(w_0), w_0} \arrow{rr}{\Phi_{r_{1}+\ell(w_0), w_0}} & & X(\Delta\br^1) \times X(\br^2\Delta)
    \end{tikzcd}
    \]
    Here $\varphi_1,\varphi_2$ are isomorphisms from \eqref{eq:iso-bs-braid-intro}.
\end{theorem}



We visually describe the splicing map in the following example, but leave the technical details needed to verify that the map is a quasi-cluster isomorphism to Example \ref{ex:intro splice}.
\begin{example}\label{ex:intro}
Consider $\br = {\color{blue}\sigma_3^{2}\sigma_2^{2}\sigma_1\sigma_3\sigma_2\sigma_1^{2}}{\color{red}\sigma_3\sigma_2\sigma_1\sigma_2\sigma_3\sigma_2\sigma_1}$, where ${\color{blue} \br^{1}}$ and ${\color{red} \br^{2}}$ are as indicated by colors. We have the quiver $Q_{\br}$:


\begin{center}
    \begin{tikzpicture}[scale=0.9]
    \draw[color=blue] (0,4) to[out=0,in=180] (1, 3) to[out=0,in=180] (2,4) to (5,4) to[out=0,in=180](6,3) to[out=0,in=180](7,2)to[out=0,in=180] (8,1) to[out=0,in=180] (9,2); 
    \draw[color=blue] (0,3) to[out=0,in=180] (1,4) to[out=0,in=180] (2,3) to[out=0,in=180] (3,2) to[out=0,in=180] (4,3) to (5,3) to[out=0,in=180] (6,4) to (9,4);
    \draw[color=blue] (0,2) to (2,2) to[out=0,in=180] (3,3) to[out=0,in=180] (4,2) to[out=0,in=180] (5,1) to (7,1) to[out=0,in=180] (8,2) to[out=0,in=180] (9,1);
    \draw[color=blue] (0,1) to (4,1) to[out=0,in=180] (5,2) to (6,2) to[out=0,in=180] (7,3) to (9,3);

    \draw[color=red] (9,4) to[out=0,in=180] (10,3) to[out=0,in=180] (11,2) to[out=0,in=180] (12,1) to (15,1) to[out=0,in=180] (16,2);
    \draw[color=red] (9,3) to[out=0,in=180] (10,4) to (13,4) to[out=0,in=180] (14,3) to[out=0,in=180] (15,2) to[out=0,in=180] (16,1);
    \draw[color=red] (9,2) to (10,2) to[out=0,in=180] (11,3) to (12,3) to[out=0,in=180] (13, 2) to (14,2) to[out=0,in=180] (15,3) to (16,3);
    \draw[color=red] (9,1) to (11,1) to[out=0,in=180] (12,2) to[out=0,in=180](13,3) to[out=0,in=180] (14,4) to (16,4);

    \node at (1, 3.5) {$1$};
    \node at (2, 3.5) {$2$};
    \node at (6, 3.5) {$6$};
    \node at (10, 3.5) {$10$};
    \node at (14.2, 3.5) {{\color{blue}$\boxed{14}$}};

    \node at (3, 2.5) {$3$};
    \node at (4, 2.5) {$4$};
    \node at (7, 2.5) {$7$};
    \node at (11, 2.5) {$11$};
    \node at (13, 2.5) {$13$};
    \node at (15.2, 2.5) {{\color{blue}$\boxed{15}$}};

    \node at (5, 1.5) {$5$};
    \node at (8, 1.5) {$8$};
    \node at (9, 1.5) {$9$};
    \node at (12, 1.5) {$12$};
    \node at (16.2, 1.5) {{\color{blue}$\boxed{16}$}};

    \draw[thick, ->] (1.2, 3.5) to (1.8, 3.5);
    \draw[thick, ->] (2.2, 3.5) to (5.8, 3.5);
    \draw[thick, ->] (6.2, 3.5) to (9.8, 3.5);
    \draw[thick, ->] (10.2, 3.5) to (13.8, 3.5);

    \draw[thick, ->] (3.2, 2.5) to (3.8, 2.5);
    \draw[thick, ->] (4.2, 2.5) to (6.8, 2.5);
    \draw[thick, ->] (7.2, 2.5) to (10.8, 2.5);
    \draw[thick, ->] (11.2, 2.5) to (12.8, 2.5);
    \draw[thick, ->] (13.2, 2.5) to (14.8, 2.5);

    \draw[thick, ->] (5.2, 1.5) to (7.8, 1.5);
   \draw[thick, ->] (8.2, 1.5) to (8.8, 1.5);
    \draw[thick, ->] (9.2, 1.5) to (11.8, 1.5);
    \draw[thick, ->] (12.2, 1.5) to (15.8, 1.5);

    \draw[thick, ->] (3.8, 2.7) to (2.2, 3.4);
    \draw[thick, ->] (5.8, 3.3) to (4.2, 2.6);
    \draw[thick, ->] (6.8, 2.7) to (6.2, 3.4);
    \draw[thick, ->] (9.8, 3.3) to (7.2, 2.6);
    \draw[thick, ->] (12.8, 2.7) to (10.2, 3.4);
    \draw[thick, ->] (13.8, 3.3) to (13.2, 2.6);

    \draw[thick, ->] (4.8, 1.7) to (4.2, 2.4);
    \draw[thick, ->] (6.8, 2.3) to (5.2, 1.6);
    \draw[thick, ->] (8.8, 1.7) to (7.2, 2.4);
    \draw[thick, ->] (10.8, 2.3) to (9.2, 1.6);
    \draw[thick, ->] (11.8, 1.7) to (11.2, 2.4);
    \draw[thick, ->] (14.8, 2.3) to (12.2, 1.6);
    \end{tikzpicture}
\end{center}

The open set $\CU_{r_1}$ is the cluster variety corresponding to the following quiver, which is obtained by freezing the vertices $6, 7, 9$. Note that these are the vertices on the right of the rightmost appearance of a crossing in ${\color{blue} \br^1}$. 

\begin{center}
    \begin{tikzpicture}[scale=0.9]
    \draw[color=blue] (0,4) to[out=0,in=180] (1, 3) to[out=0,in=180] (2,4) to (5,4) to[out=0,in=180](6,3) to[out=0,in=180](7,2)to[out=0,in=180] (8,1) to[out=0,in=180] (9,2); 
    \draw[color=blue] (0,3) to[out=0,in=180] (1,4) to[out=0,in=180] (2,3) to[out=0,in=180] (3,2) to[out=0,in=180] (4,3) to (5,3) to[out=0,in=180] (6,4) to (9,4);
    \draw[color=blue] (0,2) to (2,2) to[out=0,in=180] (3,3) to[out=0,in=180] (4,2) to[out=0,in=180] (5,1) to (7,1) to[out=0,in=180] (8,2) to[out=0,in=180] (9,1);
    \draw[color=blue] (0,1) to (4,1) to[out=0,in=180] (5,2) to (6,2) to[out=0,in=180] (7,3) to (9,3);

    \draw[color=red] (9,4) to[out=0,in=180] (10,3) to[out=0,in=180] (11,2) to[out=0,in=180] (12,1) to (15,1) to[out=0,in=180] (16,2);
    \draw[color=red] (9,3) to[out=0,in=180] (10,4) to (13,4) to[out=0,in=180] (14,3) to[out=0,in=180] (15,2) to[out=0,in=180] (16,1);
    \draw[color=red] (9,2) to (10,2) to[out=0,in=180] (11,3) to (12,3) to[out=0,in=180] (13, 2) to (14,2) to[out=0,in=180] (15,3) to (16,3);
    \draw[color=red] (9,1) to (11,1) to[out=0,in=180] (12,2) to[out=0,in=180](13,3) to[out=0,in=180] (14,4) to (16,4);

    \node at (1, 3.5) {$1$};
    \node at (2, 3.5) {$2$};
    \node at (6.05, 3.5) {{\color{blue}$\boxed{6}$}};
    \node at (10, 3.5) {$10$};
    \node at (14.2, 3.5) {{\color{blue}$\boxed{14}$}};

    \node at (3, 2.5) {$3$};
    \node at (4, 2.5) {$4$};
    \node at (7.05, 2.5) {{\color{blue}$\boxed{7}$}};
    \node at (11, 2.5) {$11$};
    \node at (13, 2.5) {$13$};
    \node at (15.2, 2.5) {{\color{blue}$\boxed{15}$}};

    \node at (5, 1.5) {$5$};
    \node at (8, 1.5) {$8$};
    \node at (9.05, 1.5) {{\color{blue}$\boxed{9}$}};
    \node at (12, 1.5) {$12$};
    \node at (16.2, 1.5) {{\color{blue}$\boxed{16}$}};

    \draw[thick, ->] (1.2, 3.5) to (1.8, 3.5);
    \draw[thick, ->] (2.2, 3.5) to (5.8, 3.5);
    \draw[thick, ->] (6.3, 3.5) to (9.8, 3.5);
    \draw[thick, ->] (10.2, 3.5) to (13.8, 3.5);

    \draw[thick, ->] (3.2, 2.5) to (3.8, 2.5);
    \draw[thick, ->] (4.2, 2.5) to (6.8, 2.5);
   \draw[thick, ->] (7.3, 2.5) to (10.8, 2.5);
    \draw[thick, ->] (11.2, 2.5) to (12.8, 2.5);
    \draw[thick, ->] (13.2, 2.5) to (14.8, 2.5);

    \draw[thick, ->] (5.2, 1.5) to (7.8, 1.5);
   \draw[thick, ->] (8.2, 1.5) to (8.8, 1.5);
    \draw[thick, ->] (9.3, 1.5) to (11.8, 1.5);
    \draw[thick, ->] (12.2, 1.5) to (15.8, 1.5);

    \draw[thick, ->] (3.8, 2.7) to (2.2, 3.4);
    \draw[thick, ->] (5.8, 3.3) to (4.2, 2.6);
    \draw[thick, ->] (9.8, 3.3) to (7.3, 2.6);
    \draw[thick, ->] (12.8, 2.7) to (10.2, 3.4);
    \draw[thick, ->] (13.8, 3.3) to (13.2, 2.6);

    \draw[thick, ->] (4.8, 1.7) to (4.2, 2.4);
    \draw[thick, ->] (6.8, 2.3) to (5.2, 1.6);
    \draw[thick, ->] (10.8, 2.3) to (9.3, 1.6);
    \draw[thick, ->] (11.8, 1.7) to (11.2, 2.4);
    \draw[thick, ->] (14.8, 2.3) to (12.2, 1.6);
    \end{tikzpicture}
\end{center}

On the other hand, the product $\BS(\br^1) \times \BS(\br^2)$ is the cluster variety corresponding to the disjoint union of the quivers $Q_{\br^1}$ and $Q_{\br^2}$:

\begin{center}
    \begin{tikzpicture}[scale=0.9]
    \draw[color=blue] (0,4) to[out=0,in=180] (1, 3) to[out=0,in=180] (2,4) to (5,4) to[out=0,in=180](6,3) to[out=0,in=180](7,2)to[out=0,in=180] (8,1) to[out=0,in=180] (9,2); 
    \draw[color=blue] (0,3) to[out=0,in=180] (1,4) to[out=0,in=180] (2,3) to[out=0,in=180] (3,2) to[out=0,in=180] (4,3) to (5,3) to[out=0,in=180] (6,4) to (9,4);
    \draw[color=blue] (0,2) to (2,2) to[out=0,in=180] (3,3) to[out=0,in=180] (4,2) to[out=0,in=180] (5,1) to (7,1) to[out=0,in=180] (8,2) to[out=0,in=180] (9,1);
    \draw[color=blue] (0,1) to (4,1) to[out=0,in=180] (5,2) to (6,2) to[out=0,in=180] (7,3) to (9,3);

    \draw[color=red] (9,4) to[out=0,in=180] (10,3) to[out=0,in=180] (11,2) to[out=0,in=180] (12,1) to (15,1) to[out=0,in=180] (16,2);
    \draw[color=red] (9,3) to[out=0,in=180] (10,4) to (13,4) to[out=0,in=180] (14,3) to[out=0,in=180] (15,2) to[out=0,in=180] (16,1);
    \draw[color=red] (9,2) to (10,2) to[out=0,in=180] (11,3) to (12,3) to[out=0,in=180] (13, 2) to (14,2) to[out=0,in=180] (15,3) to (16,3);
    \draw[color=red] (9,1) to (11,1) to[out=0,in=180] (12,2) to[out=0,in=180](13,3) to[out=0,in=180] (14,4) to (16,4);

    \node at (1, 3.5) {$1$};
    \node at (2, 3.5) {$2$};
    \node at (6.05, 3.5) {{\color{blue}$\boxed{6}$}};
    \node at (10, 3.5) {$10$};
    \node at (14.2, 3.5) {{\color{blue}$\boxed{14}$}};

    \node at (3, 2.5) {$3$};
    \node at (4, 2.5) {$4$};
    \node at (7.05, 2.5) {{\color{blue}$\boxed{7}$}};
    \node at (11, 2.5) {$11$};
    \node at (13, 2.5) {$13$};
    \node at (15.2, 2.5) {{\color{blue}$\boxed{15}$}};

    \node at (5, 1.5) {$5$};
    \node at (8, 1.5) {$8$};
    \node at (9.05, 1.5) {{\color{blue}$\boxed{9}$}};
    \node at (12, 1.5) {$12$};
    \node at (16.2, 1.5) {{\color{blue}$\boxed{16}$}};

    \draw[dashed] (9,0.5) to (9,1.24);
    \draw[dashed] (9,1.76) to (9, 4.5);

    \draw[thick, ->] (1.2, 3.5) to (1.8, 3.5);
    \draw[thick, ->] (2.2, 3.5) to (5.8, 3.5);
    \draw[thick, ->] (10.2, 3.5) to (13.8, 3.5);

    \draw[thick, ->] (3.2, 2.5) to (3.8, 2.5);
    \draw[thick, ->] (4.2, 2.5) to (6.8, 2.5);
    \draw[thick, ->] (11.2, 2.5) to (12.8, 2.5);
    \draw[thick, ->] (13.2, 2.5) to (14.8, 2.5);

    \draw[thick, ->] (5.2, 1.5) to (7.8, 1.5);
   \draw[thick, ->] (8.2, 1.5) to (8.8, 1.5);
    \draw[thick, ->] (12.2, 1.5) to (15.8, 1.5);

    \draw[thick, ->] (3.8, 2.7) to (2.2, 3.4);
    \draw[thick, ->] (5.8, 3.3) to (4.2, 2.6);
    \draw[thick, ->] (12.8, 2.7) to (10.2, 3.4);
    \draw[thick, ->] (13.8, 3.3) to (13.2, 2.6);

    \draw[thick, ->] (4.8, 1.7) to (4.2, 2.4);
    \draw[thick, ->] (6.8, 2.3) to (5.2, 1.6);
    \draw[thick, ->] (11.8, 1.7) to (11.2, 2.4);
    \draw[thick, ->] (14.8, 2.3) to (12.2, 1.6);
    \end{tikzpicture}
\end{center}
\end{example}


\subsection{Connections to other work}


The splicing maps for double Bott--Samelson varieties are predicted by results in link homology. In particular, the work of Trinh \cite{Trin21} relates the equivariant Borel--Moore homology of $\BS(\br)$ (together with the weight filtration) to the Khovanov--Rozansky homology $\HHH^{a=0}(\beta)$. The corresponding multiplication maps in link homology
$$
\HHH^{a=0}(\beta^1)\otimes \HHH^{a=0}(\beta^2)\to \HHH^{a=0}(\beta^1\beta^2)
$$
are well known and very useful, see e.g. \cite{GH} and references therein.

More generally,    the equivariant Borel-Moore homology (again with the weight filtration) of the variety $X(\br)$ is related to the Khovanov-Rozansky homology $\HHH^{a = 0}(\br\Delta^{-1})$. In the setting of Theorem \ref{thm:main-braid-varieties}, note that the braid  
\[
\left(\Delta^{-1}\lift{w^{-1}w_0}\br^1\right) \cdot \left(\br^2\lift{w}\Delta^{-1}\right)
\]
is conjugate to $\br^1\br^2\Delta^{-1}$, so they represent the same link. Thus, we have a map in link homology
$$
\HHH^{a = 0}\left(\Delta^{-1}\lift{w^{-1}w_0}\br^1\right)\otimes \HHH^{a = 0}\left(\br^2\lift{w}\Delta^{-1}\right)\to \HHH^{a = 0}\left(\br^1\br^2\Delta^{-1}\right)
$$
which suggests the splicing map $X\left(\lift{w^{-1}w_0}\br^1\right) \times X\left(\br^2\lift{w}\right) \to X\left(\br^1\br^2\right)$ as in \eqref{eq:main-braid-varieties}.

Another motivation for splicing maps comes from \cite{eberhardt-stroppel} that studies homology of (parabolic) open Richardson varieties. Let $u,v$ be two permutations in $S_k$ such that $u\le v$ in Bruhat order. As mentioned above, the open Richardson variety $R(u,v)$ is isomorphic to a certain braid variety:
$$
R(u,v)\simeq X\left(\lift{v}\cdot \lift{u^{-1}w_0}\right)
$$
It is known that compactly supported cohomology of $R(u,v)$ is closely related to the  $\Ext$ group between the Verma modules in the category $\mathcal{O}$ for $\mathfrak{gl}_k$:
\begin{equation}
\label{eq: ext}
H^*_c(R(u,v))\simeq \Ext_{\mathcal{O}}^*(\Delta_u,\Delta_v).
\end{equation}
The homological grading and the weight filtration on the left hand side correspond to  two gradings on the right hand side, see \cite[Theorem 12.5]{eberhardt-stroppel}, \cite{GLcatalan} and references therein. 

Given $u\le v\le w$ in $S_k$, \cite[Section 3]{eberhardt-stroppel} constructs a rational map
\begin{equation}
\label{eq: richardson intro}
R(u,v)\times R(v,w)\to R(u,w)
\end{equation}
which corresponds under \eqref{eq: ext} to compositions of extensions
$$
\Ext_{\mathcal{O}}^*(\Delta_u,\Delta_v)\otimes \Ext_{\mathcal{O}}^*(\Delta_v,\Delta_w)\rightarrow 
\Ext_{\mathcal{O}}^*(\Delta_u,\Delta_w),
$$
see \cite[Corollary 3.3]{eberhardt-stroppel}. This can be compared with Theorem \ref{thm:richardson-intro}, however we do not know whether our maps coincide with those of \cite{eberhardt-stroppel} and plan to investigate this in the future.

\subsection{Organization of the paper} Sections \ref{sec: background}--\ref{sec:braid and dbs} are mostly preparatory: in Section \ref{sec: background} we give the necessary background on relative positions of flags and its relation to matrix minors. In Section \ref{sec: cluster} we recall the definition of cluster algebras and cluster quasi-isomorphisms. We define braid and double Bott-Samelson varieties in Section \ref{sec:braid and dbs}. In particular, in Section \ref{sec:cluster-structure} we give details on the cluster structure on braid and double Bott-Samelson varieties obtained in \cite{CGGLSS, GLSBS22, GLSB, SW}.

The technical heart of the paper is Section \ref{sec:splicing-braid}. In Section \ref{sec:splicing}, we define the open sets $\mathcal{U}_{r_1, w}$, see \eqref{eq:def-U}, and prove Theorem \ref{thm:main-braid-varieties} as Theorem \ref{thm:splicing-braid-varieties}. In Section \ref{sec:conj-properties} we elaborate on the conjectural properties of the splicing map \eqref{eq:main-braid-varieties}, giving a more precise version of Conjecture \ref{conj:braid-varieties} as Conjecture \ref{conj:splicing-braid-varieties}. We elaborate on the relations between the different parts of this conjecture, and show a weaker version of the conjecture for the case $w = w_0$, and illustrate this with examples. Finally, in Section \ref{sec:richardson} we specialize the map \eqref{eq:main-braid-varieties} to the case of open Richardson varieties and prove Theorem \ref{thm:richardson-intro}.

In Section \ref{sec:splicing-DBS} we deal with the case of double Bott-Samelson varieties. We show Theorem \ref{thm: main-splicing-dbs-intro} in two parts: first, in Theorem \ref{thm:splicing-dbs} we construct the map \eqref{eq:BS-splicing-intro} and show that it is an isomorphism; second, we study the cluster-theoretic properties of \eqref{eq:BS-splicing-intro} in Section \ref{sec:cluster-properties-splicing}, showing the second part of Theorem \ref{thm: main-splicing-dbs-intro} as Theorem \ref{thm:quasi-cluster-dbs}. Finally, Theorem \ref{thm:comparison-braid-dbs-intro} is proved as Lemmas \ref{lem: bs vs braid splicing compatible 1} and \ref{lem: bs vs braid splicing compatible 2}. 

\subsection{Notations}

Given a sequence of matrices $(A_1,\ldots,A_k)$ and a matrix $M$, we write $M(A_1,\ldots,A_k)=(MA_1,\ldots,MA_k)$. Similarly, given a sequence of flags $(\CF^1,\ldots,\CF^k)$ and a matrix $M$, we write $M(\CF^1,\ldots,\CF^k)=(M\CF^1,\ldots,M\CF^k)$.

\section*{Acknowledgments}

The authors would like to thank Roger Casals, Mikhail Gorsky, Thomas Lam, Melissa Sherman-Bennett, David Speyer, Catharina Stroppel, and Daping Weng for the useful discussions. 

Eugene Gorsky, Soyeon Kim and Tonie Scroggin were partially supported by the NSF grant DMS-2302305. Tonie Scroggin was partially supported by UC President's Pre-Professoriate Fellowship (PPPF). Jos\'e Simental was partially supported by CONAHCyT project CF-2023-G-106 and UNAM’s PAPIIT Grant IA102124. 

\section{Background}
\label{sec: background}

We recall some standard facts we will use throughout the paper, in the way setting up notation and conventions.

\subsection{Braids and Demazure product} We will work with the positive braid monoid on $k$ strands $\Br_k^{+}$:
\[
\Br_k^{+} = \left\langle \sigma_1, \dots, \sigma_{k-1} \mid \sigma_i\sigma_j = \sigma_j\sigma_i \; \text{if} \; |i-j|>1, \  \sigma_i\sigma_{i+1}\sigma_i = \sigma_{i+1}\sigma_i\sigma_{i+1}, i = 1, \dots, k-2\right\rangle. 
\]
We have a surjective homomorphism $\pi: \Br_k^{+} \to S_k$ given by $\pi(\sigma_i) = s_i$, where $s_1, \dots, s_{k-1}$ are the simple transpositions in the symmetric group $S_k$. We also have the Demazure product $\delta: \Br_k^{+} \to S_k$, defined inductively by:
\begin{equation}\label{eq:demazure}
\delta(e) = e, \; \delta(\br\sigma_i) = \begin{cases} \delta(\br)s_i & \text{if} \; \delta(\br)s_i > \delta(\br) \\ \delta(\br) & \text{else}, \end{cases}
\end{equation}
Note that, as opposed to $\pi$, the Demazure product $\delta$ is not a morphism of monoids. If $w \in S_k$ we denote by $\lift{w} \in \Br_k^{+}$ the unique lift of minimal length (either under the projection $\pi$ or the Demazure product $\delta$) of $w$ to $\Br_k^{+}$. We will denote by $w_0 \in S_k$ the longest element, and its lift $\lift{w_0}$ will be denoted by $\longest$.  If $v, w \in S_k$ we define their Demazure product by
\begin{equation}\label{eq:demazure-product}
v \dem w := \delta(\lift{v}\cdot\lift{w}).
\end{equation}

We finish this section with the following result.

\begin{lemma}\label{lem:demazure-bounds}
    Let $v, w \in S_k$. The following are equivalent.
    \begin{itemize}
        \item[(a)] $v \dem w = w_0$.
        \item[(b)] $v \geq w_0w^{-1}$ in Bruhat order.
        \item[(c)] $w \geq v^{-1}w_0$ in Bruhat oder.
    \end{itemize}
\end{lemma}
\begin{proof}
We show that (a) implies (b). First, we remark that the Demazure product $\delta(\br)$ is the longest reduced word contained in $\br$. Definition \eqref{eq:demazure} is a greedy method to compute it -- another equally effective greedy method is to read the word $\br$ in the opposite direction. That being said, we have $v \dem w = s_{i_1}\cdots s_{i_k}w$, where $\sigma_{i_1}\cdots \sigma_{i_k}$ is a reduced subexpression of $\lift{v}$. If $v \dem w = w_0$, then $s_{i_1}\cdots s_{i_k} = w_0w^{-1}$, so $v \geq w_0w^{-1}$ as needed.
Now we show that (b) implies (a). If $\lift{w_0w^{-1}}$ appears as a reduced subexpression in $\lift{v}$, then a reduced subexpression for $w_0$ appears as a subexpression in $\lift{v}\cdots \lift{w}$, and we obtain $v \dem w = w_0$. This proves that (a) and (b) are equivalent. The equivalence between (a) and (c) is proved similarly. 
\end{proof}

\subsection{Flags and relative position} We will denote by $\Fl_k$ the variety of complete flags in the $k$-dimensional complex vector space $\C^k$. For a matrix $M \in \GL(k)$ with columns $m_1, \dots, m_k \in \C^k$, we define the flag
\[
\CF(M) = \left(\{0\} \subseteq \subs{m_1} \subseteq \subs{m_1,m_2} \subseteq \cdots \subseteq \subs{m_1, \dots, m_{k-1}} \subseteq \subs{m_1, \dots, m_k} = \C^k\right).
\]
The assignment $M \mapsto \CF(M)$ gives rise to the usual identification $\Fl_k = \GL(k)/\borel(k)$, where $\borel(k) \subseteq \GL(k)$ is the subgroup of upper triangular matrices. Note that the group $\GL(k)$ acts on $\Fl_k$ by multiplication on the left, or, equivalently,
\[
g.(\{0\} \subseteq F_1 \subseteq \cdots \subseteq F_{k-1} \subseteq \C^k) = (\{0\} \subseteq g(F_1) \subseteq \cdots \subseteq g(F_{k-1}) \subseteq \C^k),
\]
so that $g.\CF(M) = \CF(gM)$.

We will denote by $e_1, \dots, e_k \in \C^k$ the usual standard basis. Given an element $w \in S_k$, we define the coordinate flag
\[
\CF(w) := (\{0\} \subseteq \subs{e_{w(1)}} \subseteq \subs{e_{w(1)}, e_{w(2)}} \subseteq \cdots \subseteq \subs{e_{w(1)}, \dots, e_{w(k-1)}} \subseteq \C^k),
\]
in particular, we have the standard and antistandard flags:
\[
\CF^{\std} := \CF(e), \qquad \CF^{\ant} := \CF(w_0). 
\]

We say that two flags $\CF$ and $\CF'$ are in (relative) position $w \in S_k$, and write $\CF \relpos{w} \CF'$ if there exists $g \in \GL(k)$ such that $g\CF = \CF^{\std}$ and $g\CF' = \CF(w)$. Such $g$ is uniquely defined up to left multiplication by an element of $\borel(k) \cap w\borel(k)w^{-1}$. 

Note that given two flags $\CF, \CF' \in \Fl_k$ there exists a unique $w \in S_k$ such that $\CF \relpos{w} \CF'$, and such $w$ is determined by
\[
\dim(F_i \cap F'_j) = \#\left(\{1, \dots, i\}\cap\{w(1), \dots, w(j)\}\right).
\]

In particular,
\[
\CF \relpos{s_i} \CF' \; \text{if and only if} \; F_i \neq F'_i \; \text{and} \; F_j = F'_j \; \text{for} \; j \neq i. 
\]

Relative position of flags is closely related to the Bruhat decomposition of the group $\GL(k)$:
\[
\GL(k) = \bigsqcup_{w \in S_k} \borel(k)w\borel(k).
\]
Indeed, note that if $M, N \in \GL(k)$ then $\CF(M) \relpos{w} \CF(M')$ if and only if there exist upper triangular matrices $U_1, U_2 \in \borel(k)$ and $g \in \GL(k)$ such that $gM = U_1, gM' = wU_2$, and this is in turn equivalent to $M^{-1}M \in \borel(k)w\borel(k).$ The Bruhat decomposition satisfies the following multiplicative property. If $w \in S_k$ and $i = 1, \dots, k-1$, we have
\begin{equation}\label{eq:bruhat-dec}
(\borel w \borel)(\borel s_i\borel) = \begin{cases} \borel ws_i\borel & \; \text{if} \; \ell(ws_i) > \ell(w), \\ (\borel ws_i\borel)\sqcup (\borel w \borel) & \text{else}.\end{cases}
\end{equation}

This implies the following standard lemma that we will use repeatedly. For a proof, see e.g. \cite[Lemma 3.2]{CGGLSS}.

\begin{lemma}\label{lem:dec-rel-pos}
Let $\CF, \CF' \in \Fl_k$ and $w \in S_k$. The following are equivalent:
\begin{enumerate}
\item $\CF \relpos{w} \CF'$.
\item There exist a reduced decomposition $w = s_{i_1}\cdots s_{i_r}$ and flags $\CF^1, \dots, \CF^{r-1}$ such that
\[
\CF \relpos{s_{i_1}} \CF^1 \relpos{s_{i_2}} \cdots \relpos{s_{i_{r-1}}} \CF^{r-1} \relpos{s_{i_r}} \CF'. 
\]
\item For any reduced decomposition $w = s_{i_1}\cdots s_{i_r}$ there exist flags $\CF^1, \dots, \CF^{r-1}$ such that
\[
\CF \relpos{s_{i_1}} \CF^1 \relpos{s_{i_2}} \cdots \relpos{s_{i_{r-1}}} \CF^{r-1} \relpos{s_{i_r}} \CF'. 
\]
\end{enumerate}
Moreover, given a reduced decomposition of $w$  the flags $\CF^1, \dots, \CF^{r-1}$ in (3) are unique. 
\end{lemma}

The set $\Gamma_w$ of pairs of flags $(\CF,\CF')$ such that $\CF\relpos{w}\CF'$ is a locally closed subvariety of $\Fl_k\times \Fl_k$. The intermediate flags $\CF^1,\ldots,\CF^r$ depend on  $(\CF,\CF')$ algebraically, that is, define $(r-1)$ regular maps from $\Gamma_w$ to $\Fl_k$.

\begin{corollary}
\label{cor: length additive}
(a) Suppose that $w=w_1w_2$ where $w,w_1,w_2\in S_k$, and $\ell(w)=\ell(w_1)+\ell(w_2)$. Then $\CF\relpos{w}\CF'$ if and only if there exists a flag $\CF''$ such that 
$$
\CF\relpos{w_1}\CF''\relpos{w_2}\CF'
$$
In this case, $\CF''$ is unique. \\

(b) Assume that we have $\CF \relpos{w_1} \CF' \relpos{w_2} \CF''$. Then, $\CF \relpos{w} \CF''$ with $w \leq w_1\dem w_2$, where $w_1\dem w_2$ is defined by \eqref{eq:demazure-product}.
\end{corollary}
\begin{proof}
Part (a) is immediate from Lemma \ref{lem:dec-rel-pos}. For part (b), assume $\CF = \CF(M)$, $\CF' = \CF(M')$ and $\CF'' = \CF(M'')$. So we have $M^{-1}M' \in \borel w_1\borel$ and $(M')^{-1}M'' \in \borel w_2\borel$. Now $M^{-1}M'' \in (\borel w_1\borel)(\borel w_2\borel)$ and the result follows from \eqref{eq:bruhat-dec}.
\end{proof}

\subsection{Transverse flags} We will say that two flags $\CF$ and $\CF'$ are transverse and write  $\CF \pitchfork \CF'$ if $\CF \relpos{w_0} \CF'$. 
Note that
\[
\CF \pitchfork \CF' \; \text{if and only if} \; F_i \cap F'_{k-i} = \{0\} \; \text{for all} \; i, \; \text{if and only if} \; F_i + F'_{k-i} = \C^k \; \text{for all} \; i. 
\]

\begin{lemma}\label{lem:gen-LU-dec}
    Let $M$ be a nonsingular matrix and $w \in S_k$. Then, $\CF(M) \pitchfork \CF(w)$ if and only if the matrix $M$ admits a decomposition of the form
    \[
    M = ww_0LU,
    \]
    where $L$ is lower-triangular and $U$ is upper-triangular. Such decomposition is unique upon requiring that $L$ has $1$'s on the diagonal.
\end{lemma}
\begin{proof}
    We have that $\CF(M) \pitchfork \CF(w)$ if and only if there exists an element $g \in \GL(k)$ such that $gM \in \borel$ and $gw \in w_0\borel$. If such an element exists, then for some $U_1,U_2\in \borel$ we get $g = w_0U_1w^{-1}$   and $M = g^{-1}U_2 = (w_0U_1w^{-1})^{-1}U_2 = wU_1^{-1}w_0U_2 = ww_0(w_0U_1^{-1}w_0)U_2$, so such a decomposition of $M$ exists with $L=w_0U_1^{-1}w_0$. %
    Conversely, if $M = ww_0LU$ then setting $g = L^{-1}w_0w^{-1}$ we have that $gM = U \in \borel$ and $gw = L^{-1}w_0 \in w_0\borel$. Finally, the uniqueness claim follows from the uniqueness of the LU-decomposition. 
\end{proof}

\subsection{Braid matrices}

For $i = 1, \dots, k-1$ and a formal variable $z$ we define the matrix $B_i(z)$ to be
\[
B_i(z) = \begin{pmatrix}1&\cdots&&&\ldots&0\\ \vdots&\ddots&&&&\vdots\\ 0&\cdots&z&-1&\cdots&0\\ 0&\cdots&1&0&\cdots&0\\ \vdots&&&&\ddots&\vdots\\ 0&\cdots&&&\cdots&1\end{pmatrix}
\]
where the non-identity part of $B_i(z)$ is located in the $i$-th and $(i+1)$-st row and columns.

The matrix $B_i(z)$ is important for us due to the following well-known (and easy) lemma. 

\begin{lemma}
\label{lem: relpos via matrices}
Let $M \in \GL(n)$ be a nondegenerate matrix and let $i = 1, \dots, n-1$ Then, $\CF \relpos{s_i} \CF(M)$ if and only if there exists a (necessarily unique) $z \in \C$ such that $\CF = \CF(MB_i(z))$. 
\end{lemma}


The following lemma is also easy to check, see \cite[Corollary 3.9]{CGGLSS} and \cite[Lemma 2.20]{CGGS}.

\begin{lemma}
\label{lem: push right}
If $U$ is an upper-triangular matrix and $z\in \C$ then there exist unique upper-triangular matrix $U'$ and $z'\in \C$ such that 
$$
UB_i(z)=B_i(z')U'
$$
Furthermore, the diagonal entries of $U'$ are permuted from the ones of $U$ by $s_i$.
\end{lemma}

\subsection{Minors}
\label{subsec:minors}

We  use notation $[i]=\{1,\ldots,i\}$ and $u[i]=\{u(1),\ldots,u(i)\}$ for $u\in S_k$. 
If $I, J \subseteq [k]$ are sets of the same cardinality and $M \in \GL(k)$, we denote by  $\minor_{I, J}(M)$ 
the determinant of the $|I| \times |J|$-submatrix of $M$ obtained by deleting the rows (resp. columns) not belonging to $I$ (resp. $J$). For $i = 1, \dots, k$, we have the principal minor
\begin{equation}\label{eq:minors-transversality}
\minor_{i}(M) := \minor_{[i], [i]}(M),
\end{equation}
i.e., $\minor_{i}(M)$ is the determinant of the upper-left justified $i \times i$-submatrix of $M$. Thus, for example, $\minor_1(M) = m_{11}$ and $\minor_k(M) = \det(M)$. It is a classical result that a matrix admits an $LU$ decomposition if and only if all its principal minors are nonzero. 

More generally, we have the following.

\begin{lemma}
\label{lem: relpos w minors}
a) For $w \in S_k$ and $M \in \GL(k)$  we have that $\CF(M) \pitchfork \CF(w)$ if and only if the minors 
\[
\minor_{ww_0[i], [i]}(M) = (-1)^{\ell(w)}\minor_{[i],[i]}(w_0w^{-1}M)
\]
are nonzero for all $i = 1, \dots, k$.   

b) If $\CF^{\std} \xrightarrow{w} \CF(M)$, then the minors $\minor_{w[i], [i]}(M)$ are nonzero for  $i = 1, \dots, k$. 


c) If $\minor_{w_[i], [i]}(M) \neq 0$ for every $i = 1, \dots, k$, and $\CF^{\std} \relpos{v} \CF(M)$, then $v \not< w$. 
\end{lemma}

\begin{proof}
Part (a) follows from  Lemma \ref{lem:gen-LU-dec} and \eqref{eq:minors-transversality}. 

For part (b) observe that the condition $\CF^{\std} \xrightarrow{w} \CF(M)$ implies
\[
\CF(ww_0)\relpos{w_0w^{-1}} \CF^{\std} \relpos{w} \CF(M)
\]
so by Corollary \ref{cor: length additive}(a) we get $\CF(ww_0) \pitchfork \CF(M)$. Now the statement follows from (a). 

For part (c), we have that $\CF(ww_0)\pitchfork \CF(M)$ and $\CF(ww_0) \relpos{w_0w^{-1}} \CF^{\std} \relpos{v} \CF(M)$. By Corollary \ref{cor: length additive}(b) and using the fact that $w_0$ is the longest element in $S_k$, we have $w_0 = (w_0w^{-1}) \dem v$. But if $v < w$ then $(w_0w^{-1}) \dem v < w_0$, so we obtain $v \not< w$. 
\end{proof}

\begin{definition}
For $w \in S_k$ we define the subset
\[
\CU(w) := \{\CF \in \Fl_k \mid \CF \pitchfork \CF(w)\} \subseteq \Fl_k.
\]
\end{definition} 

\begin{lemma}\label{lem:open-cover-flag-variety}
    For all $w \in S_k$, the subset $U(w)$ is open in $\Fl_k$ and moreover
    \[
    \Fl_k = \bigcup_{w \in S_k}\CU(w).
    \]
\end{lemma}
\begin{proof}
    By Lemma \ref{lem: relpos w minors}, $\CU(w)$ is given by the nonvanishing of the minors $\minor_{ww_0[i], [i]}$, $i = 1, \dots, k$, so it is clearly open in $\Fl_k$. Now, consider a flag $\F = \F(M) \in \Fl_k$, and pick any element $v \in S_k$. We have $\CF(M) \relpos{w_1} \CF(v)$ for some $w_1 \in S_k$, and let $w_0 = w_1w_2$ be a length-additive decomposition. 
    We get
    $$
    \CF(M)\relpos{w_1}\CF(v)\relpos{w_2}\CF(vw_2)
    $$
    and by Corollary \ref{cor: length additive}(a) this implies that $\CF(M) \relpos{w_0} \CF(vw_2)$, i.e., $\CF(M) \in \CU(vw_2)$. 
\end{proof}

\begin{lemma}\label{lem:cauchy-binet-minors}
Let $M \in \gl(k)$. Let $i, j = 1, \dots, k-1$ and assume $i \neq j$. Let $I \subseteq [n]$ be a set with $|I| = i$. Then, $\minor_{I,[i]}(MB_j(z)) = \minor_{I,[i]}(M)$. 
\end{lemma}
\begin{proof}
    By the Cauchy-Binet formula,
    \[
    \minor_i(MB_j(z)) = \sum_{\substack{J \subseteq [n] \\ |J| = i}} \minor_{I, J}(M)\minor_{J, [i]}(B_j(z)). 
    \]
    Since $i \neq j$, we have that $\Delta_{J, [i]}(B_j(z)) \neq 0$ if and only if $J = [i]$, in which case $\Delta_{[i], [i]}(B_j(z)) = 1$. Thus,
    \[
    \minor_{I,[i]}(MB_j(z)) = \minor_{I,[i]}(M)\minor_{[i], [i]}(B_j(z)) = \minor_{I,[i]}(M). 
    \]
\end{proof}

\section{Cluster algebras}
\label{sec: cluster}
\subsection{Definition} We recall the definition of a cluster algebra \cite{FZcluster}. For this paper, we will only need to restrict ourselves to the skew-symmetric case.

An \emph{ice quiver} $Q$ is a quiver with finite vertex set $Q_0$, i.e., a finite directed graph with which we allow multiple edges between vertices but no loops nor directed two cycles. We specify that a special subset $Q_0^{f}$ of the vertices of $Q$ is declared to be \emph{frozen} whereas an element in $Q_0 \setminus Q_0^{f}$ is declared to be \emph{mutable}.

Given that $Q$ is an ice quiver, with $|Q_0| = n+m$, where we distinguish $n$ vertices as mutable and $m$ vertices as frozen. We consider a field $\mathcal{F}$ of transcendence degree $n+m$ over $\C$. A \emph{seed} $\Sigma = (Q, \mathbf{x})$ consists of:
\begin{enumerate}
    \item The ice quiver $Q$, and
    \item A set $\mathbf{x} = \{x_i \mid i \in Q_0\}$ that is a transcendental basis for $\mathcal{F}$, i.e. $\mathcal{F} = \C(x_i \mid i \in Q_0)$. We say that $\mathbf{x}$ is the set of \emph{cluster variables} of $\Sigma$.
\end{enumerate}

Given a seed $\Sigma = (Q, \mathbf{x})$ and a mutable vertex $k \in Q_0$, the \emph{mutation} of $\Sigma$ in the direction $k$ is the seed $\mu_k(\Sigma) = (\mu_k(Q), \mu_k(\mathbf{x}))$ where:
\begin{enumerate}
    \item $\mu_k(\mathbf{x}) = (\mathbf{x}\setminus\{x_k\})\cup\{x'_k\}$, where $x'_k \in \mathcal{F}$ is defined by
    \[
    x_kx_k' = \prod_{i \to k}x_i + \prod_{k \to j}x_j.
    \]
    \item $\mu_k(Q)$ has the same vertex set as $Q$, but the arrows change via the following three-step procedure:
    \begin{enumerate}
        \item Reverse all arrows incident with $k$.
        \item For any pair of arrows $i \to k \to j$ in $Q$, create a new arrow $i \to j$.
        \item If the previous two steps have created any $2$-cycles, then remove the arrows which form a maximal collection of disjoint $2$-cycles. 
    \end{enumerate}
\end{enumerate}
We call two seeds $\Sigma$ and $\Sigma'$ \emph{mutation equivalent} if there is a finite sequence of mutations from one seed to the other. Mutation at a fixed vertex $k$ is an involutive operation and therefore, mutation equivalence is well-defined. We will denote by $\seeds(\Sigma)$ the set of all seeds which are mutation equivalent to $\Sigma$. 

\begin{definition}
    Let $\Sigma$ be a seed. The \emph{cluster algebra} $A(\Sigma)$ is the subalgebra of the field $\mathcal{F}$ generated by the sets of cluster variables in all seeds mutation equivalent to $\Sigma$, as well as by $x_{k}^{-1}$ for $k \in Q_0^{f}$.
    We say that a commutative algebra $A$ \emph{admits a cluster structure} if there exists a seed $\Sigma$ such that $A \cong A(\Sigma)$. Similarly, we say that an affine algebraic variety $X$ admits a cluster structure if the coordinate algebra $\C[X]$ admits a cluster structure.  
\end{definition}

\subsection{Quasi-cluster morphisms} 
\label{sec: quasi cluster}
It is possible that there exist two non-mutation equivalent seeds $\Sigma,\Sigma'$ such that $A(\Sigma)\cong A\cong A(\Sigma')$, i.e., the cluster structures for a commutative algebra $A$ are generally not unique. 



\begin{example}
\label{ex: quasi cluster}
Let $Q$ be the quiver $${\color{blue}{a}} \to 1 \to {\color{blue} {b}},$$  where the frozen vertices are shown in blue. Let $x_a, x_1, x_b$ be the corresponding cluster variables, then the associated cluster algebra $A(\Sigma)$ is defined as
\[
A(\Sigma) = \C[x_1, x_1', x_a^{\pm 1}, x_b^{\pm 1}]/(x_1x_1' = x_a + x_b)
\]
Now, let $Q'$ be the quiver $$ {\color{blue}{a}} \to 1 \qquad {\color{blue} {b}}$$ Let $y_a, y_1, y_b$ be the corresponding cluster variables, then $A(\Sigma')$ is
\[
A(\Sigma') = \C[y_1, y_1', y_a^{\pm 1}, y_b^{\pm 1}]/(y_1y_1' = 1 + y_a)
\]
Although these quivers are not mutation equivalent, there is an isomorphism between their cluster algebras given by the assignment $y_1 \mapsto x_1x_b^{-1}, y_a \mapsto x_ax_b^{-1}$, $y_b \mapsto x_b$ and $y'_1 \mapsto x'_1$. 
\end{example}

Let $A(\Sigma)$ and $A(\Sigma')$ be the cluster algebras associated to the seeds $\Sigma$ and $\Sigma'$, respectively. As demonstrated in Example \ref{ex: quasi cluster}, the algebras $A(\Sigma)$ and $A(\Sigma')$ may be isomorphic even if the seeds $\Sigma$ and $\Sigma'$ are not mutation equivalent. Following Fraser \cite{Fraser}, see also \cite[Section 5.2]{LamSpeyerI}, we define an interesting class of morphisms between cluster algebras of perhaps non-mutation-equivalent seeds.

Given a seed $\seed$ and a mutable vertex $i$, we define the exchange ratio $\widehat{y}_i$ as the ratio $$\widehat{y}_i=\dfrac{\prod_{j\rightarrow i}x_j^{\#\{j\rightarrow i\}}}{\prod_{i\rightarrow j}x_j^{\#\{i\rightarrow j\}}}.$$ 

\begin{definition}\label{def:quasi-homo}\cite{Fraser,FSB22}
Let $A(\seed)$ and $A(\seed')$ be cluster algebras of rank $n+m$, each with $m$ frozen variables. Let $\mathbf{x} = \{x_1, \dots, x_{n+m}\}$ be the cluster variables of $\seed$, and $\mathbf{x} = \{x'_1, \dots x'_{n+m}\}$ be the cluster variables of $\seed'$. A \newword{quasi-cluster isomorphism} is an algebra isomorphism $f: A(\seed) \to A(\seed')$ satisfying the following conditions:
\begin{enumerate}
\item For each frozen variable $x_j \in \mathbf{x}$, $f(x_j)$ is a Laurent monomial in the frozen variables of $\mathbf{x}'$.
\item For each mutable variable $x_i \in \mathbf{x}$, $f(x_i)$ coincides with $x'_i$, up to multiplication by a Laurent monomial in the frozen variables of $\mathbf{x}'$.
\item The exchange ratios are preserved, i.e., for each mutable variable $x_i$ of $\seed$, $f(\widehat{y}_i) = \widehat{y}'_i$.
\end{enumerate}
\end{definition}

\begin{remark}\label{rmk:Fraser}
By \cite[Corollary 4.5]{Fraser}, if a quasi-isomorphism as in Definition \ref{def:quasi-homo} exists, then the mutable parts of the quivers $Q$ and $Q'$ coincide, i.e., the quivers $Q$ and $Q'$ coincide after deleting all frozen variables (and arrows incident to them). See also \cite[Section 5.2]{LamSpeyerI}.
\end{remark}

One can check that the map in Example \ref{ex: quasi cluster} is a quasi-cluster isomorphism.
We remark that if properties (1)--(3) are stable under mutations, i.e., if (1)--(3) holds for two seeds $\seed, \seed'$, then it also holds for $\mu_i(\seed), \mu_i(\seed')$ for every mutable vertex $i$, see \cite[Section 5.2]{LamSpeyerI}. While a quasi-cluster isomorphism $f: A(\seed) \to A(\seed')$ does not send cluster variables to cluster variables, it does preserve all the cluster-theoretic geometric information such as cluster tori, and induces an isomorphism between upper cluster algebras, see \cite[Section 2.4]{CGSS}.  


To finish this section, let us see how cluster quasi-morphisms can be obtained from maps between exchange matrices. For this, let us first observe that the information of an ice quiver $Q$ with $n$ mutable vertices $1, \dots, n$ and $m$ frozen ones $n+1, \dots, n+m$ can be codified into an \emph{extended exchange matrix}, that is, an $(n+m) \times n$-matrix $\widetilde{B}$, defined by
\[
\widetilde{b}_{i,j} = \#\{\text{arrows} \; i \to j\} - \#\{\text{arrows} \; j \to i\}.
\]
Note that the top part (formed by the first $n$ rows) of $\widetilde{B}$ is a skew-symmetric matrix $B$, called the \emph{principal part} of $\widetilde{B}$.

\begin{lemma}\label{lem:quasi-iso-exchange-matrix}\cite[Theorem 5.7]{LamSpeyerI}
Let $\Sigma = (Q, \mathbf{x}), \Sigma' = (Q', \mathbf{z})$ be two seeds, both with $n$ mutable and $m$ frozen variables, and extended exchange matrices $\widetilde{B}, \widetilde{B}'$. Assume that there exists an $(n+m)\times (n+m)$ integer matrix $R$ of block triangular form
\begin{equation}\label{eq:block-triangular-R}
R = \begin{pmatrix}
\id_n & 0 \\ P & Q
\end{pmatrix}.
\end{equation}
such that $R\widetilde{B} = \widetilde{B}'$. If  $\det(Q) = \pm 1$, then the assignment
\begin{equation}\label{eq:quasi-iso-matrix}
\Phi(x_j) = \prod_{i = 1}^{n+m}z_{i}^{r_{i,j}}
\end{equation}
extends to a cluster quasi-isomorphism $\Phi: A(\Sigma) \to A(\Sigma')$. Moreover, if $\Phi: A(\Sigma) \to A(\Sigma')$ is a cluster quasi-isomorphism, then there exists a seed $\Sigma''$ mutation equivalent to $\Sigma'$ and a matrix $R$ of the form $\eqref{eq:block-triangular-R}$ such that $R\widetilde{B} = \widetilde{B}''$ and the quasi-isomorphism is given by \eqref{eq:quasi-iso-matrix}.
\end{lemma}


\section{Braid and double Bott-Samelson varieties}\label{sec:braid and dbs}

In this section, we recall the definition of the braid variety, both via configurations of flags and as an affine variety given by an explicit set of equations. 

\subsection{Definition via flags} 


Recall that $\Br_k^{+}$ is the positive braid monoid on $k$ strands, with generators $\sigma_{1}, \dots, \sigma_{k-1}$.

\begin{definition}\label{def:braid-variety}
    Let $\br = \sigma_{i_1}\cdots \sigma_{i_r} \in \Br^{+}_{k}$ be a positive braid. The \newword{braid variety} $X(\br)$ is the variety consisting of $(r+1)$-tuples of flags $(\CF^0, \dots, \CF^{r})$ satisfying the following relative position conditions:
    \begin{enumerate}
        \item $\CF^0 = \CF^{\std}$ and $\CF^r = \CF^{\ant}$.
        \item For every $j = 1, \dots, r$, $\CF^{j-1} \relpos{s_{i_j}} \CF^{j}$. 
    \end{enumerate}
\end{definition}

Let us remark that, up to a canonical isomorphism, $X(\br)$ depends only on the braid $\br$ and not on its presentation as a product of generators, this follows from Lemma \ref{lem:dec-rel-pos}. We also remark that $X(\br)$ is nonempty if and only if we have $\delta(\br) = w_0$, in which case $X(\br)$ is a smooth, affine algebraic variety of dimension $r - \binom{k}{2} = r-\ell(w_0)$, see \cite{CGGS2,CGGLSS}.

\begin{remark}
        We could define the braid variety by replacing condition (1) of Definition \ref{def:braid-variety} by $\CF^0 = \CF^{\std}$ and $\CF^{r} = \CF(\delta(\br))$. With this definition, we always have that $X(\br)$ is affine, smooth, nonempty and of dimension $r - \ell(\delta(\br))$. However, by \cite[Lemma 3.4]{CGGLSS}, we lose no generality by assuming that $\delta(\br) = w_0$.  We will always assume that $\delta(\br) = w_0$.
\end{remark}

\begin{definition}\label{def:dbs-flags}
    Let $\br = \sigma_{i_1}\cdots \sigma_{i_r} \in \Br^{+}_{k}$ be a positive braid. The \newword{double Bott--Samelson variety} $\BS(\br)$ is the variety consisting of $(r+1)$-tuples of flags $(\CF^0, \dots, \CF^{r})$ satisfying the following relative position conditions. 
    \begin{enumerate}
        \item $\CF^0 = \CF^{\std}$ and $\CF^r \pitchfork \CF^{\ant}$.
        \item For every $j = 1, \dots, r$, $\CF^{j-1} \relpos{s_{i_j}} \CF^{j}$. 
    \end{enumerate}
\end{definition}

Note that the definitions of $X(\br)$ and $\BS(\br)$ are superficially very similar, with the crucial difference that in $X(\br)$ we require $\CF^{r} = \CF^{\ant}$ and in $\BS(\br)$ we require $\CF^{r} \pitchfork \CF^{\ant}$. We can, nevertheless, realize every double Bott-Samelson variety as a braid variety in at least two different ways.

\begin{lemma}\label{lem:bs-to-braid}
We have isomorphisms
\[
\bsbra: \BS(\br) \to X(\br\Delta), \qquad \bsbrb: \BS(\br) \to X(\Delta\br).
\]
\end{lemma}
\begin{proof}
If we have a chain of flags $(\CF^0, \dots, \CF^r) \in \BS(\br)$ then, since $\CF^r \pitchfork \CF(w_0)$ by Lemma \ref{lem:dec-rel-pos} there exists a unique chain $\CF^r \relpos{s_{a_1}} \CF^{r+1} \relpos{s_{a_2}} \cdots \relpos{s_{a_{\ell(w_0)}}} \CF(w_0)$, where $\Delta = \sigma_{a_1}\cdots \sigma_{a_{\ell(w_0)}}$, so setting $\bsbra(\CF^0, \dots, \CF^r)=(\CF^0, \dots, \CF^r, \CF^{r+1}, \dots, \CF(w_0))$ we obtain an element of $X(\br\Delta)$. Conversely, if $(\CF^0, \dots, \CF^{r}, \CF^{r+1}, \dots, \CF(w_0)) \in X(\br\Delta)$, it is easy to see that $(\CF^0, \dots, \CF^{r}) \in \BS(\br)$, so $\bsbra$ is an isomorphism.

Let us now construct the map $\bsbrb^{-1}: X(\Delta\br) \to \BS(\br)$. Take an element \[\left(\CF^0, \dots, \CF^{\ell(w_0)}, \CF^{\ell(w_0)+1}, \dots, \CF^{\ell(w_0) + r} = \CF(w_0)\right) \in X(\Delta\br).\] Note that $\CF^{\ell(w_0)} \in \borel w_0 \borel/\borel$, so it is of the form $\CF(U_1w_0U_2)$,  where $U_1, U_2$ are upper triangular and $U_1$ is unique provided it has $1$'s on the diagonal. Now consider  the sequence of flags
\begin{equation}\label{eq:phi-2-inverse}
w_0U_1^{-1}\left(\CF^{\ell(w_0)}, \dots, \CF^{\ell(w_0) + r}\right). 
\end{equation}
Note that
$$
w_0U_1^{-1}\CF^{\ell(w_0) + r}=w_0U_1^{-1}\CF(w_0)\pitchfork \CF^{\ant}
$$
so \eqref{eq:phi-2-inverse} defines a point in $\BS(\beta)$, and we define it to be $\bsbrb^{-1}(\CF^0, \dots, \CF^{\ell(w_0)+r})$.

To construct the map $\bsbrb$, take an element $(\CF^0, \dots, \CF^r) \in \BS(\br)$. Since $\CF^r \pitchfork \CF(w_0)$, we have $\CF^r = \CF((w_0V_1w_0)V_2)$, where $V_1$, $V_2$ are upper triangular and $V_1$ is unique provided it has $1$'s on the diagonal. Note that $V_1^{-1}w_0\CF^r = \CF(w_0V_2) = \CF(w_0)$, and that $\CF^{\std} \relpos{w_0} V_1^{-1}w_0\CF^0$. So there exist unique flags $\widetilde{\CF}^{1}, \dots, \widetilde{\CF}^{\ell(w_0) - 1}$ so that
\[
\bsbrb(\CF^0, \dots, \CF^r) = \left(\CF^{\std}, \widetilde{\CF}^{1}, \dots, \widetilde{\CF}^{\ell(w_0) - 1}, V_1^{-1}w_0\CF^0, \dots, V_1^{-1}w_0\CF^{r}\right) \in X(\Delta\br).
\]
It is straightforward to see that $\bsbrb$ and $\bsbrb^{-1}$ are indeed inverse maps. 
\end{proof}

\subsection{Definition via equations} 


 For a positive braid $\br = \sigma_{i_1}\cdots \sigma_{i_r}$, define the matrix
\[
B_{\br}(z_1, \dots, z_r) = B_{i_1}(z_1)\cdots B_{i_r}(z_r). 
\]

This allows us to present the braid variety $X(\br)$ explicitly via equations. 

\begin{corollary}
\label{cor: braid}
Let $\br=\sigma_{i_1}\cdots \sigma_{i_r}$ be a braid, and assume the Demazure product  is $\delta(\br) = w_0$. Then
$$
X(\br)=\{(z_1,\ldots,z_r):w_0B_{\beta}(z_1,\ldots,z_r)\ \text{is upper-triangular}\}.
$$
\end{corollary}

Indeed, given $(z_1, \dots, z_r)$, by Lemma \ref{lem: relpos via matrices} one gets the sequence of flags $\CF^0 = \CF^{\std}$, $\CF^1 = \CF(B_{i_1}(z_1))$, $\CF^2 = \CF(B_{i_1}(z_1)B_{i_2}(z_2))$ and so on. To obtain a similar description for the double Bott-Samelson variety $\BS(\br)$, we use Lemma \ref{lem:gen-LU-dec}.

\begin{corollary}
\label{cor: BS}
Let $\br$ be a positive braid. The 
double Bott-Samelson variety is
\[
\BS(\br) = \{(z_1, \dots, z_r) \in \C^r \mid B_{\br}(z_1, \dots, z_r) \; \text{admits an LU decomposition}\}. 
\]
\end{corollary}

Equivalently, using principal minors we have
\[
\BS(\br) = \{ (z_1, \dots, z_r)\in \C^r \mid \minor_{[i],[i]}(B_{\br}(z)) \neq 0 \; \text{for every} \; i = 1, \dots, k\}
\]
so that $\BS(\br)$ is a principal open subvariety of $\C^r$. If $w \in S_k$ then $\BS(\lift{w})$ coincides with the double Bruhat cell of $w$, that is, the open Richardson variety $R(e,w)$.

The following lemmas give some useful examples of explicit braid matrices.
 

\begin{lemma}\label{lem:partial-coxeter}
    Consider the partial Coxeter element $\cox(i) = \sigma_{k-1}\cdots \sigma_{k-i}$. Then,
    \[
    B_{\cox(i)}(z_1, \dots, z_i) = \begin{pmatrix}
        \id_{(k-i-1) \times (k-i-1)} & \mathbf{0}_{(k-i-1) \times 1} & \mathbf{0}_{(k-i-1) \times 1} & \dots & \mathbf{0}_{(k-i-1) \times 1} \\
        \mathbf{0}_{1\times (k-i-1)} & z_i & -1 & \dots & 0 \\
        \mathbf{0}_{1 \times(k-i-1)} & z_{i-1} & 0 & \dots & 0 \\
        \vdots & \vdots & \vdots & \ddots & \vdots  \\
        {\mathbf{0}_{1 \times (k-i-1)}} & z_1 & 0 & \dots & -1 \\
        {\mathbf{0}}_{1 \times (k- i-1)} & 1 & 0 & \dots & 0 
            \end{pmatrix}
    \]
    where $\id_{(k-i-1) \times (k-i-1)}$ is the identity matrix of size $k-i-1$, $\mathbf{0}_{1 \times (k-i-1)}$ is the row vector of size $k-i-1$ with only $0$'s, and similarly for $\mathbf{0}_{(k-i-1) \times 1}$. 
\end{lemma}
\begin{proof}
    By induction on $i$, the case $i = 1$ is simply the definition of the braid matrix $B_{k-1}(z_1)$ and the induction step is a simple computation.
\end{proof}

The next result is analogous to \cite[(2.5)]{CGGS} and can be shown by induction using Lemma \ref{lem:partial-coxeter} above.

\begin{lemma}\label{lem:braid-matrix-w0}
Consider the braid word $\Delta = (\sigma_{k-1}\cdots \sigma_1)(\sigma_{k-1}\cdots \sigma_2)\cdots (\sigma_{k-1}\sigma_{k-2})\sigma_{k-1}$. Then,
\begin{equation}\label{eq:w0}
B_{\Delta}(z_1, \dots, z_{\binom{k}{2}}) = \begin{pmatrix} z_{k-1} & -z_{2k-3} & z_{3k-6} & \dots & (-1)^{k-1}z_{\binom{k}{2} - 1} & (-1)^{k}z_{\binom{k}{2}} & (-1)^{k+1} \\ z_{k-2} & -z_{2k-4} & z_{3k-7} & \dots & (-1)^{k-1}z_{\binom{k}{2} - 2} & (-1)^{k} &  0 \\ z_{k-3} & -z_{2k-5} & z_{3k-8} & \dots & (-1)^{k-1} & 0 & 0  \\ \vdots & \vdots & \vdots & \ddots  & \vdots & \vdots & \vdots \\ z_{2} & -z_{k} & 1 & \dots & 0 & 0 & 0 \\ z_{1} & -1 & 0 & \dots & 0 & 0 & 0 \\ 1 & 0 & 0 & \dots & 0 & 0 & 0 \end{pmatrix}
\end{equation}
In particular, $w_0B_{\Delta}(z_1, \dots, z_{\binom{k}{2}})$ is a lower-triangular matrix, and $B_{\Delta}(z_1, \dots, z_{\binom{k}{2}})w_0$ is an upper-triangular matrix. 
\end{lemma}

\begin{lemma}\label{lem:going-back-to-standard}
Let $(z_1, \dots, z_r) \in X(\br)$ and let $\widetilde{\Delta}$ be any braid of length $\ell(w_0)$ such that $\pi(\widetilde{\Delta}) = w_0$. Then, there exist unique functions $y_1, \dots, y_{\ell(w_0)} \in \C[X(\br)]$ such that the matrix
\[
B_{\br}(z_1, \dots, z_r)B_{\widetilde{\Delta}}(y_1 \dots, y_{\ell(w_0)})
\]
is diagonal. 
\end{lemma}
\begin{proof}
Since any two reduced expressions for $w_0$ are related by braid moves, it is enough to prove this for a single reduced braid lift of $w_0$, and for $$\widetilde{\Delta} = (\sigma_{k-1}\cdots \sigma_1)(\sigma_{k-1}\cdots \sigma_2)(\sigma_{k-1}\sigma_{k-2})(\sigma_{k-1})$$ this follows immediately from Lemma \ref{lem:braid-matrix-w0}. 
\end{proof}

We can obtain a description in coordinates of the isomorphism $\varphi_2: \BS(\br) \to X(\Delta\br)$.

\begin{corollary}\label{cor:phi2-coordinates}
There exist unique functions $p_1, \dots, p_{\ell(w_0)} \in \C[\BS(\br)]$ such that the map
\[
\varphi_2: \BS(\br) \to X(\Delta\br), \qquad \varphi_2(z_1, \dots, z_r) = \left(p_1, \dots, p_{\ell(w_0)}, z_1, \dots, z_r\right)
\]
is an isomorphism. 
\end{corollary}
\begin{proof}
If $(z_1, \dots, z_r) \in \BS(\br)$, then the braid matrix $B_{\br}(z_1, \dots, z_r)$ admits an LU decomposition, say $B_{\br}(z_1, \dots, z_r) = LU$, equivalently $w_0L^{-1}B_{\br}(z_1, \dots, z_r) = w_0U$. By Lemma \ref{lem:braid-matrix-w0}, we can find unique functions $p_1, \dots, p_{\ell(w_0)}$ such that $w_0L^{-1} = B_{\Delta}(p_1, \dots, p_{\ell(w_0)})$ and $\varphi_2: \BS(\br) \to X(\Delta\br)$ is well-defined.

Note that the inverse map is given by 
\[\varphi_2^{-1}(z_1, \dots, z_{\ell(w_0)}, z_{\ell(w_0)+1}, \dots, z_{\ell(w_0)+r}) = (z_{\ell(w_0)+1}, \dots, z_{\ell(w_0)+r}).\]
Indeed, if $B_{\Delta}(z_1, \dots, z_{\ell(w_0)})B_{\br}(z_{\ell(w_0)+1}, \dots, z_{\ell(w_0)+r}) = w_0U_1$ for some upper triangular matrix $U_1$, then $B_{\br}(z_{\ell(w_0)+1}, \dots, z_{\ell(w_0)+r}) = B_{\Delta}^{-1}(z_1,\dots, z_{\ell(w_0)})w_0U_1$, that has an LU-decomposition by Lemma \ref{lem:braid-matrix-w0}. It is straightforward to verify that $\varphi_2, \varphi_2^{-1}$ are indeed inverses of each other.
\end{proof}

\subsection{Cluster structure}\label{sec:cluster-structure} Let us now describe the cluster structure on braid varieties obtained in \cite{CGGLSS}, see also \cite{GLSB, GLSBS22}. This cluster structure generalizes the one obtained for double Bott-Samelson varieties in \cite{SW}, that we will also describe. 

\subsubsection{Cluster structure on braid varieties: variables}  We will define an initial cluster $\mathbf{x}_{X(\br)}$ first by means of its non-vanishing locus. For each $m = 1, \dots, r$, let $\beta_{m} = \sigma_{i_1}\cdots \sigma_{i_m}$, so that $\beta_{r} = \beta$. 

\begin{definition}
The left-to-right inductive torus  (also known as Deodhar torus) $\ltrt_{\br} \subseteq X(\br)$ consists of those elements $(\CF^0 = \CF^{\std}, \CF^1, \dots, \CF^r = \CF^{\ant})$ such that
\[
\CF^{0} \relpos{\delta(\br_{m})} \CF^m
\]
for all $m = 1, \dots, r$.
\end{definition}

We can give a system of coordinates on $\ltrt_{\br}$ as follows. 

\begin{lemma}[Proposition 2.12, \cite{GLSB}]
 The torus $\ltrt_{\br}$ is given by the non-vanishing of the minors    
 \begin{equation}\label{eq:minors-cluster-monomials}
\minor_{\delta(\br_{m})[i_m], [i_m]}\left(B_{\br_{m}}\!\left(z_1, \dots, z_m\right)\right).
\end{equation}
\end{lemma}
\begin{proof}
If $(z_1, \dots, z_r) \in \ltrt_{\br}$ then all minors \eqref{eq:minors-cluster-monomials} are nonzero by Lemma \ref{lem: relpos w minors} (b). Conversely, assume all minors \eqref{eq:minors-cluster-monomials} are nonzero. We show by induction on $m$ that $\CF^{\std} \relpos{\delta(\br_m)} \CF(B_{\br_m}(z_1, \dots, z_m))$.  The base of induction is $m = 0$, where $\br_m$ is the identity braid and $B_{\br_m}(z_1,\dots, z_m)$ is the identity matrix. For brevity, we will denote $\CF^m := \CF(B_{\br_m}(z_1, \dots, z_m))$. 

Assume $\CF^0 \xrightarrow{\delta(\br_m)}\CF^m$. Note that we have
\[
\CF^{0} \xrightarrow{\delta(\br_m)} \CF^m \xrightarrow{s_{i_{m+1}}} \CF^{m+1}.
\]
If $\delta(\br_{m+1}) = \delta(\br_m)s_{i_{m+1}}$ then using Corollary \ref{cor: length additive}(a) we obtain $\CF^0 \xrightarrow{\delta(\br_{m+1})} \CF^{m+1}$, as needed. If, on the other hand, $\delta(\br_{m+1}) = \delta(\br_m)$ then using Corollary \ref{cor: length additive}(b) we have that $\CF^{0} \xrightarrow{v} \CF^{m+1}$ with $v \leq \delta(\br_m)\dem s_{i_{m+1}} = \delta(\br_{m})$. Thus, by Lemma \ref{lem: relpos w minors}(c), to conclude that $v = \delta(\br_m) = \delta(\br_{m+1})$, it is enough to show that $\minor_{\delta(\br_{m})[j], [j]}\!\left(B_{\br_{m}}(z_1, \dots, z_m)B_{s_{i_{m+1}}}(z_{m+1})\right)$ is nonzero for all $j = 1, \dots, k$. But by Lemma \ref{lem:cauchy-binet-minors}, if $j \neq i_{m+1}$ then
\[
\minor_{\delta(\br_{m})[j], [j]}\!\left(B_{\br_{m}}(z_1, \dots, z_m)B_{s_{i_{m+1}}}(z_{m+1})\right) = \minor_{\delta(\br_{m})[j], [j]}\!\left(B_{\br_{m}}(z_1, \dots, z_m)\right) \neq 0,
\]
where the last inequality follows from the induction assumption. If $j = i_{m+1}$, it is our assumption that $\minor_{\delta(\br_{m})[i_{m+1}], [i_{m+1}]}\!\left(B_{\br_{m}}(z_1, \dots, z_m)B_{s_{i_{m+1}}}(z_{m+1})\right) \neq 0$, and the result follows.
\end{proof}

One might then expect that the minors \eqref{eq:minors-cluster-monomials} are the cluster variables in an initial seed for $X(\beta)$. 
However, this does not work because in general \eqref{eq:minors-cluster-monomials} is not an irreducible polynomial in $z_1, \dots, z_m$. However, the irreducible factors of \eqref{eq:minors-cluster-monomials} are in fact the cluster variables defining the torus $\ltrt_{\br}$ and the minors \eqref{eq:minors-cluster-monomials} are cluster monomials. Moreover, one can obtain \eqref{eq:minors-cluster-monomials} from the cluster variables in an upper uni-triangular fashion, and we only need to consider the minors \eqref{eq:minors-cluster-monomials} for those $m$ so that $\delta(\br_{m}) = \delta(\br_{m-1})$. To summarize, for each $m = 1, \dots, r$ such that $\delta(\br_{m}) = \delta(\br_{m-1})$, the minor \eqref{eq:minors-cluster-monomials} has a unique irreducible factor that has not appeared in such a minor for a smaller index, this irreducible factor appears with multiplicity one, and it is a cluster variable. See \cite{CGGLSS, GLSB,GLSBS22} and also \cite{Ingermanson} for details.

The quiver $Q$ forming a seed with the cluster described above can be obtained using the \emph{Lusztig cycles} in the left-to-right inductive weave of $\br$. Since we will not need this level of detail we will not go into it, and instead refer the reader to \cite[Section 4]{CGGLSS}. 

\subsubsection{Cluster structure on double Bott-Samelson varieties: variables}

We use the isomorphism $\varphi_2: \BS(\br) \to X(\Delta\br)$ from Lemma \ref{lem:bs-to-braid} in order to translate the cluster structure on $X(\Delta\br)$ described above to $\BS(\br)$. We remark that the resulting cluster structure is the one obtained by Shen and Weng in \cite{SW}, which predates the construction of cluster structures on general braid varieties, see \cite[Section 4.8 and Proposition 4.20]{CGGLSS}. 

First, we make some simplifications. By Corollary \ref{cor:phi2-coordinates}, 
we have $\varphi_2:(z_1, \dots, z_r)\mapsto \left(p_1, \dots, p_{\ell(w_0)}, z_1, \dots, z_r\right)$, where $\pp=\{p_1, \dots, p_{\ell(w_0)}\}$ are uniquely defined functions of $\zz=\{z_1, \dots, z_r\}$.

Second, $\delta((\Delta\br)_{m}) = \delta((\Delta\br)_{m-1})$ if and only if $m > \ell(w_0)$, and in this case the Demazure product is precisely $w_0$. We denote $j=m-\ell(w_0)$, so that the $m$-th letter in $\Delta\beta$ is $i_{j}$.

Therefore the cluster variables are the irreducible factors of
\[
\minor_{w_0[i_j], [i_j]}\left(B_{(\Delta\br)_{m}}\left(p_1, \dots, p_{\ell(w_0)},z_1, \dots, z_j\right)\right), \qquad j =  1, \dots,  \ell(\br).
\]
Note, however, that
\[
\begin{array}{rl}
 \minor_{w_0[i_j],[i_j]}\left(B_{(\Delta\br)_{s}}\!\left(\pp,z_1, \dots, z_j\right)\right) =  & \minor_{w_0[i_j], [i_j]}
 \left(B_{\Delta}(\pp)B_{\br_{j}}\!\left(z_{1}, \dots, z_j\right)\right)  \\
    = & \minor_{w_0[i_j], [i_j]}\left(w_0L(\pp)B_{\br_{j}}\!\left(z_{1}, \dots, z_j\right)\right) \\
    = & \minor_{[i_j], [i_j]}\left(L(\pp)B_{\br_{j}}\!\left(z_{1}, \dots, z_j\right)\right) \\
    = & \minor_{[i_j], [i_j]}\left(B_{\br_{j}}\!\left(z_{1}, \dots, z_j\right)\right)
\end{array}
\]
where we have used that $L(\pp)$ is lower triangular with $1$'s on the diagonal, and Lemma \ref{lem:braid-matrix-w0}. 
We conclude that the cluster variables in $\BS(\br)$ are the irreducible factors of 
\begin{equation}\label{eq:cluster-vars-dbs}
x_j := \minor_{[i_j],[i_j]}(B_{\br_{j}}(z_1, \dots, z_j)). 
\end{equation}
However, by \cite[Lemma 3.30]{SW}, the function $x_j$ is already irreducible. We summarize the previous discussion in the following result.

\begin{lemma}[\cite{SW}]
The cluster variables in $\BS(\br)$ are as in \eqref{eq:cluster-vars-dbs}, with $j = 1, \dots, r$.
\end{lemma}


Since $\minor_{[i],[i]}(U) \neq 0$ for any matrix $U \in \borel(k)$ and any $i = 1, \dots, k$, we immediately obtain the following.

\begin{corollary}
    The common nonvanishing locus of $x_j, j = 1, \dots, r$, is
    \[
    \ltrt_{\Delta\br}=\left\{(\CF^0, \CF^1, \dots, \CF^r) \in \BS(\br) \mid \CF^{i} \pitchfork \CF(w_0) \; \text{for all} \; i\right\}.
    \]
\end{corollary}

For $s = 1, \dots, k-1$, let $\last(s)$ be such that $i_{\last(s)} = s$ and $i_j \neq s$ for $j > \last(s)$. In words, $\last(s)$ is the position of the rightmost appearance of $\sigma_s$ in $\br$. If $s$ does not appear in $\br$, we define $\last(s)$ to be a formal symbol $\oslash$. If $\last(s) \neq \oslash$ then by Lemma \ref{lem:cauchy-binet-minors}, we have
\[
x_{\last(s)} = \minor_{[s],[s]}(B_{\br_{\last(s)}}(z_1, \dots, z_{\last(s)})) = \minor_{[s],[s]}(B_{\br}(z_1, \dots, z_r)),
\]
so that, by the definition of $\BS(\br)$, $x_{\last(s)}$ is invertible in $\C[\BS(\br)]$. In fact, the frozen variables in $\BS(\br)$ are $\{x_{\last(s)} \mid s = 1, \dots, k-1, \last(s) \neq \oslash\}$. For convenience and future use, we set $x_{\oslash} = 1$. 

Let us describe the frozen variables in terms of the LU-decomposition of $B_{\br}(z_1, \dots, z_r)$. Let $B_{\br}(z_1, \dots, z_r) = LU$, where $L$ has $1$'s in the diagonal and the diagonal of $U$ is $\diag(u_1, \dots, u_k)$, where each $u_i \in \C[\BS(\br)]$ is a function on $\BS(\br)$. Note that \[\det(B_{\br}(z_1, \dots, z_r)) = 1 = u_1\cdots u_k\] so that each $u_i$ is a unit in $\C[\BS(\br)]$ and, moreover, 
\begin{equation}\label{eq:frozen-in-terms-of-upper}
x_{\last(s)} = \minor_{[s],[s]}(B_{\br}(z_1, \dots, z_r)) = u_1\cdots u_s.
\end{equation}

Expressing $u$'s in terms of the frozen variables, we have $u_1 = x_{\last(1)}$, $u_s = x_{\last(s)}x_{\last(s-1)}^{-1}$ for $s = 2, \dots, k-1$ and $u_k = u_1^{-1}\cdots u_{k-1}^{-1} = x_{\last(k-1)}^{-1}$. 

\subsubsection{Cluster structure on double Bott-Samelson varieties: quiver}
Next, we describe the quiver for double Bott-Samelson variety $\BS(\beta)$.

The quiver $Q_{\br}$ can be read directly from the braid diagram of $\br$ drawn horizontally. The vertices of $Q_{\br}$ are in bijection with the connected components of $\mathbb{R}^2$ minus the braid diagram of $\br$ which are bounded on the left: each connected component corresponds to the crossing directly to its left. Around each crossing, we have the following configuration of half arrows 

\begin{center}
\begin{tikzpicture}
\draw [color=lightgray!200] (0,1)..controls (0.55,1.1) and (0.55,0)..(1.1,0);
\draw [color=lightgray!200] (0,0)..controls (0.55,0) and (0.55,1.1)..(1.1,1.1);
\draw[->, thick,dashed] (0.1,0.5) to (1,0.5);
\draw[->, thick,dashed] (0.1,0.6) to (1,0.6); 
\draw[->, thick,dashed] (1.2, 0.7) to (0.65,1.2);
\draw[->, thick,dashed] (1.2, 0.4) to (0.65,-0.1);
\draw[<-, thick,dashed] (-0.1, 0.7) to (0.45,1.2);
\draw[<-, thick,dashed] (-0.1, 0.4) to (0.45,-0.1);
\end{tikzpicture}
\end{center}

\noindent which produces a quiver that possibly has oriented $2$-cycles, and the quiver $Q_{\br}$ is obtained after removing a maximal collection of disjoint oriented $2$-cycles. 

To see that half arrows add up to an integer number of arrows, it is useful to rephrase the construction of  $Q_{\br}$. 
The vertices of the quiver are in bijection with the letters of $\br$. Let us say that the vertex $j$ has color $i$ and write $\colour(j)=i$ if $i_j = i$. The arrows in $Q_{\br}$ are of two types: \emph{mixed} (between vertices of different colors) and \emph{unmixed} (between vertices of the same color). 

If $j_1 < j_2$ are vertices of the same color, then there is an unmixed arrow $j_1 \to j_2$ if and only if there does not exist $j_1 < k < j_2$ of the same color as $j_1, j_2$. Put it more succinctly: there is an unmixed arrow pointing right between two consecutive appearances of the same color. These are all the unmixed arrows.

\begin{center}
\begin{tikzpicture}
 \draw (0,1) to[out=0,in=180] (1,0); 
 \draw (0,0) to[out=0,in=180] (1,1);
 \draw [dashed] (1,1)--(3,1);
 \draw [dashed] (1,0)--(3,0);
  \draw (3,1) to[out=0,in=180] (4,0); 
 \draw (3,0) to[out=0,in=180] (4,1);
 \draw (1,0.5) node {$j_1$};
 \draw (4,0.5) node {$j_2$};
 \draw [->,thick] (1.2,0.5)--(3.8,0.5);
\end{tikzpicture}
\end{center}

Let us now describe the mixed arrows. Assume $j_1 < j_2$ have different colors. If there is a mixed arrow $j_2 \to j_1$, then $\colour(j_2) = \colour(j_1) \pm 1$. Now, let us say that $j_1 < j_2$ and $\colour(j_2) = \colour(j_1) \pm 1$. Then, there is an arrow $j_2 \to j_1$ if and only if there exists $j_1 < j_2 < j'_1$ such that
\begin{itemize}
    \item $\colour(j_1) = \colour(j'_1)$ and, moreover, $j_1$ and $j'_1$ are consecutive appearances of this color.
    \item there does not exist $j_2 < k < j'_1$ of the same color as $j_2$. 
\end{itemize}

\begin{center}
\begin{tikzpicture}
 \draw (0,1) to[out=0,in=180] (1,0); 
 \draw (0,0) to[out=0,in=180] (1,1);
 \draw [dashed] (1,1)--(2,1);
 \draw [dashed] (1,0)--(3,0);
\draw [dashed] (0,2)--(2,2);
 
\draw [dashed] (3,2)--(4,2);
   \draw (2,1) to[out=0,in=180] (3,2); 
 \draw (2,2) to[out=0,in=180] (3,1);
  \draw (3,1) to[out=0,in=180] (4,0); 
 \draw (3,0) to[out=0,in=180] (4,1);
 \draw (1,0.5) node {$j_1$};
 \draw (4,0.5) node {$j'_1$};
 \draw (3,1.5) node {$j_2$};
 \draw [->,thick] (3,1.3)--(1.3,0.5);
\end{tikzpicture}
\qquad
\qquad
\begin{tikzpicture}
 \draw (0,-1) to[out=0,in=180] (1,0); 
 \draw (0,0) to[out=0,in=180] (1,-1);
 \draw [dashed] (1,-1)--(2,-1);
 \draw [dashed] (1,0)--(3,0);
\draw [dashed] (0,-2)--(2,-2);
 
\draw [dashed] (3,-2)--(4,-2);
   \draw (2,-1) to[out=0,in=180] (3,-2); 
 \draw (2,-2) to[out=0,in=180] (3,-1);
  \draw (3,-1) to[out=0,in=180] (4,0); 
 \draw (3,0) to[out=0,in=180] (4,-1);
 \draw (1,-0.5) node {$j_1$};
 \draw (4,-0.5) node {$j'_1$};
 \draw (3,-1.5) node {$j_2$};
 \draw [->,thick] (3,-1.3)--(1.3,-0.5);
\end{tikzpicture}
\end{center}

Finally, the frozen vertices are $\{\last(k) | k = 1, \dots, n-1\}$. 


\section{Splicing braid varieties} \label{sec:splicing-braid}

\subsection{The splicing map}\label{sec:splicing} We now describe splicing maps for general braid varieties.  Let us fix the braid $\br=\sigma_{i_1}\cdots \sigma_{i_r}$. Furthermore, let $\br^1=\sigma_{i_1}\cdots \sigma_{i_{r_1}}$ and $\br^2=\sigma_{i_{r_1+1}}\cdots \sigma_{i_r}$ so that $\br=\br^1\br^2$. For fixed $w \in S_k$, we consider the subset
\begin{equation}\label{eq:def-U}
\CU_{r_1,w}(\br) := \left\{\left(\CF^0, \dots, \CF^r\right) \in X(\br) \mid \CF^{r_1} \pitchfork \CF(w_0w)\right\}.
\end{equation}
By Lemma \ref{lem:open-cover-flag-variety}, $\CU_{r_1, w} \subseteq X(\br)$ is open and
\[
X(\br) = \bigcup_{w \in S_k}\CU_{r_1, w}.
\]
Note, however, that for a fixed $w \in S_k$ the open set $\CU_{r_1, w}$ may be empty. Below, we will obtain necessary and sufficient conditions for $\CU_{r_1,w}(\br)$ to be nonempty, see Corollary \ref{cor: U nonempty}. Assume for the time being that $\CU_{r_1,w}(\br) \not= \emptyset$ and let $(\CF^{0}, \dots, \CF^r) \in \CU_{r_1,w}(\br)$. We fix a reduced expression for $w_0$ that has a reduced word for $w$ as a prefix:
\begin{equation}
\label{eq: choice w0}
w_0 = s_{a_1}\cdots s_{a_{\ell(w_0)}},\quad w = s_{a_1}\cdots s_{a_{\ell(w)}.}
\end{equation}
By Lemma \ref{lem:dec-rel-pos} there exists a unique collection of flags (all of them coordinate flags) $\widetilde{\CF}^1, \dots, \widetilde{\CF}^{\ell(w_0)} = \CF^{\std}$ so that $\widetilde{\CF}^{\ell(w)}=\CF(w_0w)$ and we have the following configuration of flags

\[\begin{tikzcd}
	{\mathcal{F}^{\mathrm{std}}} & {\mathcal{F}^1} & \cdots & \textcolor{blue}{\mathcal{F}^{r_1}} & \cdots & {\mathcal{F}^{\mathrm{ant}}} \\
	{\mathcal{F}^{\mathrm{std}}} & {\widetilde{\mathcal{F}}^{\ell(w_0)-1}} & \cdots & \textcolor{blue}{\mathcal{F}(w_0w)} & \cdots & {\mathcal{F}^{\mathrm{ant}}}
	\arrow["{s_{i_1}}", from=1-1, to=1-2]
	\arrow["{s_{i_2}}", from=1-2, to=1-3]
	\arrow["{s_{i_{r_1}}}", from=1-3, to=1-4]
	\arrow["{s_{i_{r_1+1}}}", from=1-4, to=1-5]
	\arrow["{s_{i_{r}}}", from=1-5, to=1-6]
	\arrow[Rightarrow, no head, from=1-6, to=2-6]
	\arrow["{s_{a_1}}"', from=2-6, to=2-5]
	\arrow["{s_{a_{\ell(w)}}}"', from=2-5, to=2-4]
	\arrow["w", curve={height=-12pt}, from=2-6, to=2-4]
	\arrow[from=2-4, to=2-3]
	\arrow[from=2-3, to=2-2]
	\arrow["{s_{a_{\ell(w_0)}}}"', from=2-2, to=2-1]
	\arrow[Rightarrow, no head, from=1-1, to=2-1]
\end{tikzcd}\]
In particular, $\widetilde{\CF}^{\ell(w)}=\CF(w_0w)$.
Since $\CF^{r_1} \pitchfork \CF(w_0w)$, there exist 
\begin{enumerate}
\item $g_1 \in \GL(k)$ such that $g_1\CF^{r_1} = \CF^{\ant}$, $g_1\CF(w_0w) = \CF^{\std}$.
    \item $g_2 \in \GL(k)$ such that $g_2\CF^{r_1} = \CF^{\std}$, $g_2\CF(w_0w) = \CF^{\ant}$. 
\end{enumerate}
\begin{remark}\label{rmk:ambiguity}
Given that $\borel\cap w_0\borel w_0$ is the subgroup of diagonal matrices, the elements $g_1, g_2$ are unique up to multiplication on the left by a diagonal matrix. We will clarify this ambiguity in the course of the proof of Theorem \ref{thm:splicing-braid-varieties} below. 
\end{remark}

Now we define: 
\begin{align}\label{eq:Phi1}
\Phi^1\left(\CF^0, \dots, \CF^r\right) = g_1\left(\CF(w_0w), \cdots, \CF_{\std}, \CF^1, \dots, \CF^{r_1}\right) \in X\left(\lift{w^{-1}w_0}\br^1\right),
\\
\label{eq:Phi2} \Phi^2\left(\CF^0, \dots, \CF^r\right) = g_2\left(\CF^{r_1}, \CF^{r_1+1}, \dots, \CF^{r}, \widetilde{\CF}^{1}, \dots, \CF(w_0w)\right) \in X\left(\br^2\lift{w}\right).  
\end{align} 

Schematically, and up to the translation by $g_1$ and $g_2$, the maps $\Phi_1$ and $\Phi_2$ are defined as follows: 
\begin{equation}\label{eq:def-phi1-phi2}
\begin{tikzpicture}
\node at (0,0) {$\CF^{\std}$};
\draw[->] (0.5,0) to (1.5,0);
\node at (1,0.15) {\scriptsize $s_{i_1}$};
\node at (2,0) {$\CF^{1}$};
\draw[->] (2.5,0) to (3.5,0);
\node at (3,0.15) {\scriptsize $s_{i_2}$};
\node at (4,0) {$\cdots$};
\draw[->] (4.5,0) to (5.5,0);
\node at (4.75,0.15) {\scriptsize $s_{i_{r_1}}$};
\node at (6,0) {$\CF^{r_1}$};
\draw[->] (6.5,0) to (7.5,0);
\node at (6.95,0.15) {\scriptsize $s_{i_{r_1+1}}$};
\node at (8,0) {$\cdots$};
\draw[->] (8.5,0) to (9.5,0);
\node at (9,0.15) {\scriptsize $s_{i_{r}}$};
\node at (10,0) {$\CF^{\ant}$};
\draw[double equal sign distance] (10, -0.1) to (10, -0.9);
\node at (10,-1.3) {$\CF^{\ant}$};
\draw[->] (9.5, -1.3) to (8.5, -1.3);
\node at (9,-1.15) {\scriptsize $s_{a_1}$};
\node at (8,-1.3) {$\cdots$};
\draw[->] (7.5, -1.3) to (6.8, -1.3);
\node at (7.1,-1.15) {\scriptsize $s_{a_{\ell(w)}}$};
\node at (6,-1.3) {$\CF(w_0w)$};
\draw[->] (5.2, -1.3) to (4.5, -1.3);
\node at (4,-1.3) {$\cdots$};
\draw[->] (3.5, -1.3) to (2.8, -1.3);
\node at (2,-1.3) {$\widetilde{\CF}^{\ell(w_0)-1}$};
\draw[->] (1.3, -1.3) to (0.5, -1.3);
\node at (0.9,-1.15) {\scriptsize $s_{a_{\ell(w_0)}}$};
\node at (0, -1.3) {$\CF^{\std}$};
\draw[double equal sign distance] (0, -0.1) to (0, -0.9);

\draw[dashed, red] (10.5, 0.5) to (5,0.5) to (5, -1.7) to (10.5, -1.7);
\node at (10.7, -0.6) {\color{red}$\Phi^2$};

\draw[dashed, blue] (-0.5, 0.7) to (7, 0.7) to (7, -2) to (-0.5, -2);
\node at (-0.7, -0.6) {\color{blue}$\Phi^1$};
\end{tikzpicture}
\end{equation}


\begin{theorem}\label{thm:splicing-braid-varieties}
    The map $\Phi_{r_1, w} = (\Phi^1, \Phi^2): \CU_{r_1,w}(\br) \to X\left(\lift{w^{-1}w_0}\br^1\right) \times X\left(\br^2\lift{w}\right)$ is an isomorphism of algebraic varieties.
\end{theorem}


\begin{proof} 
We break down the proof in several steps. In what follows, we fix the reduced expression \eqref{eq: choice w0} for $w_0$. Let $y_1, \dots, y_{\ell(w_0)} \in \C[X(\br)]$ be the functions constructed in Lemma \ref{lem:going-back-to-standard} for the lift $\lift{w_0}$. We denote $\yy=\{y_1, \dots, y_{\ell(w_0)}\}$ and 
$$
\yy^L=\{y_{\ell(w)+1}, \dots, y_{\ell(w_0)}\},\ \yy^R=\{y_{1}, \dots, y_{\ell(w)}\}.$$ Also, if $\zz = (z_1, \dots, z_r) \in X(\br)$ we will set 
$$
\zz^L=\{z_{1}, \dots, z_{r_1}\},\ \zz^R=\{z_{r_1+1}, \dots, z_{r}\}.$$
and $M = M(\zz^L) = B_{\br^1}(z_1, \dots, z_{r_1})$. Note that, by definition, $\zz \in \CU_{r_1, w}(\beta)$ if and only if $\CF(M(\zz^L)) \pitchfork \CF(w_0w)$. Also, $MB_{\beta^2}(\zz^R)=B_{\beta}(\zz)$.

In the course of defining the maps $\Phi^1, \Phi^2$, we glossed over the fact that the elements $g_1, g_2$ are only defined up to left multiplication by a diagonal matrix. Thus, in the first two steps we carefully choose $g_1$ and $g_2$. 

{\it Step 1: Explicit construction of the map $\Phi^2$.}
  Assume $\zz \in \CU_{r_1, w}(\beta)$. 
  By Lemma \ref{lem:gen-LU-dec}, $M$ admits a decomposition
\begin{equation}\label{eq:gen-LU-dec-phi2}
M=M(\zz^L) = (w_0ww_0)L(\zz)U(\zz)=(w_0ww_0)LU
\end{equation}
where $L=L(\zz)$ is lower-triangular with $1$'s on the diagonal and $U=U(\zz)$ is upper-triangular. Note that the entries of $L(\zz)$ and $U(\zz)$ are rational functions on $\C[X(\br)]$ which are regular on $\CU_{r_1,w}(\beta)$. Now we consider the sequence of matrices
\begin{equation}\label{eq:phi2-matrices}
M\left(I, B_{i_{r_1+1}}\left(z_{r_1+1}\right), \dots, B_{\br^2}\left(\zz^R\right), B_{\br^2}\left(\zz^R\right)B_{a_1}(y_1), \dots, B_{\br^2}\left(\zz^R\right)B_{\lift{w}}\left(\yy^R\right)\right),
\end{equation}
that projects to the flags on the red part of the diagram \eqref{eq:def-phi1-phi2}. Multiplying these matrices on the left by $((w_0ww_0)L(\zz))^{-1}$ we obtain the sequence of matrices
\begin{equation}\label{eq:phi2-matrices-U}
U\left(I, B_{i_{r_1+1}}\left(z_{r_1+1}\right), \dots, B_{\br^2}\left(\zz^R\right), B_{\br^2}\left(\zz^R\right)B_{a_1}(y_1), \dots, B_{\br^2}\left(\zz^R\right)B_{\lift{w}}\left(\yy^R\right)\right),
\end{equation}
where $U = U(\zz)$ as in \eqref{eq:gen-LU-dec-phi2}. Since $\CF(B_{\br}\left(\zz\right)B_{\lift{w}}(\yy^R)) = \CF(w_0w)$, we have that there exists an upper triangular matrix $U_1$ such that 
$$
B_{\br}\left(\zz\right)B_{\lift{w}}(\yy^R)=MB_{\br^2}\left(\zz^R\right)B_{\lift{w}}\left(\yy^R\right) =w_0wU_1,$$ so that $UB_{\br^2}\left(\zz^R\right)B_{\lift{w}}\left(\yy^R\right) = L^{-1}w_0U_1 \in w_0\borel$, i.e. $\CF(UB_{\br^2}\left(\zz^R\right)B_{\lift{w}}\left(\yy^R\right)) = \CF(w_0)$ and the projection of \eqref{eq:phi2-matrices-U} to the flag variety  defines an element of $X(\br^2\underline{w})$, that we define to be $\Phi^2$. Note that the element $g_2$ from \eqref{eq:Phi2} (see also Remark \ref{rmk:ambiguity}) is precisely 
\begin{equation}
\label{eq: g2}
g_2=((w_0ww_0)L(\zz))^{-1}.    
\end{equation}
In coordinates,  we \lq\lq slide the matrix $U$ to the right\rq\rq \, using Lemma \ref{lem: push right} to obtain regular functions  $\widetilde{\zz^R}=\{\widetilde{z}_{r_1}, \widetilde{z}_{r_1+1}, \dots, \widetilde{z}_r\},\widetilde{\yy^R}=\{ \widetilde{y}_1, \dots, \widetilde{y}_{\ell(w)}\}$ on $U_{r_1,w}$ such that
$$
UB_{\br^2}\left(\zz^R\right)B_{\lift{w}}\left(\yy^R\right)=B_{\br^2\lift{w}}\left(\widetilde{\zz^R},\widetilde{\yy^R}\right)U'
$$
and we define $\Phi^2(\zz) = \left(\widetilde{z}_{r_1+1}, \dots, \widetilde{z}_r, \widetilde{y}_1, \dots, \widetilde{y}_{r_1}\right) \in X\left(\br^2\lift{w}\right)$. \\

{\it Step 2: Explicit construction of the map $\Phi^1$.} The map $\Phi^1$ is constructed analogously to $\Phi^2$, as follows. We now consider the sequence of matrices
\begin{equation}\label{eqn:flag-sequence-phi1}
\left(B_{\br}(\zz)B_{\lift{w}}(\yy^R), \dots, B_{\br}(\zz)B_{\lift{w_0}}(\yy) = U_2 | B_{i_1}(z_1), \dots, M\right)
\end{equation}
that project down to the flags in the blue part of the diagram \eqref{eq:def-phi1-phi2}. The vertical line indicates the fact that, while the flags $\CF^{\std}$ on the left extreme of \eqref{eq:def-phi1-phi2} are equal, the matrix $B_{\br}(\zz)B_{\lift{w_0}}(\yy)$ does not need to be the identity matrix; we only know that $B_{\br}(\zz)B_{\lift{w_0}}(\yy)$ is an upper triangular matrix $U_2$. 
In order to circumvent this problem, we write the matrices $\left(B_{i_1}(z_1), \dots, B_{\br^1}(\zz^L) = M\right)$ as $U_2U_2^{-1}\left(B_{i_1}(z_1), \dots, B_{\br^1}(\zz^L)\right)$ and slide the upper triangular matrix $U_2^{-1}$ to the right, so we obtain regular functions $\zz'^L=\{z'_1, \dots, z'_{r_1}\}$  on $\CU_{r_1, w}(\beta)$ so that the sequence of matrices
\begin{equation}\label{eqn:flag-sequence-phi1-2}
\left(B_{\br}(\zz)B_{\lift{w}}(\yy^R), \dots, B_{\br}(\zz)B_{\lift{w_0}}(\yy) = U_2, U_2B_{i_1}(z'_1), \dots, U_2B_{\br^1}(\zz'^L)\right)
\end{equation}
    project to the same flags as the sequence \eqref{eqn:flag-sequence-phi1}. Note that $U_2B_{\br^1}(\zz'^L) = M(\zz)U_3(\zz)$, where $U_3(\zz)$ is an upper triangular matrix, so we have a decomposition
\begin{equation}\label{eq:LU-dec-phi1}
U_2B_{\br^1}(z'_1, \dots, z'_{r_1}) = (w_0ww_0)L(\zz)U(\zz)U_3(\zz)
\end{equation}
where the matrix $L(\zz)$ is the same as in the decomposition \eqref{eq:gen-LU-dec-phi2}. So we can multiply all matrices in the sequence \eqref{eqn:flag-sequence-phi1-2} by $w_0((w_0ww_0)L(\zz))^{-1}$ and now we proceed as in Step 1. Note that the element $g_1$ from \eqref{eq:Phi1} (see also Remark \ref{rmk:ambiguity}) is given by
\begin{equation}
\label{eq: g1}
g_1=w_0((w_0ww_0)L(\zz))^{-1}.   
\end{equation}
\\

Steps $1$ and $2$ show that $\Phi_{r_1,w}$ is indeed a regular map. Note that, if $g_1$ and $g_2$ denote the translating elements as in \eqref{eq:Phi1} and \eqref{eq:Phi2}, we obtain that $g_2 = w_0g_1$. \\

{\it Step 3. Construction of $\Phi_{r_1,w}^{-1}$.} 
Let us construct the map $\Phi_{r_1,w}^{-1}$. For this, we take
\[
\F = \left(\F_{\std}, \F_1, \dots, \F_{\ell(\br^2)+ \ell(w)} = \F_{\ant}\right) \in X\left(\br^2\lift{w}\right)
\]
\[
\CG = \left(\CG_0 = \F_{\std}, \CG_1, \dots, \CG_{\ell(w^{-1}w_0) + \ell(\br^1)} = \F_{\ant}\right) \in X\left(\lift{w^{-1}w_0}\br^1\right).
\]
and arrange these flags as follows:
\begin{equation}\label{eq:def-inverse}
\begin{tikzcd}
\node at (0,0) {\CF^{\ell(w^{-1}w_0)+1}};
\draw[->] (1,0.15) to (1.5,0.15);
\node at (2,0) {\cdots};
\draw[->] (2.5,0.15) to (3.5,0.15);
\node at (4,0) {\CF^{\ant}};
\draw[double equal sign distance] (4.5,0.15) to (5,0.15);
\node at (5.5,0) {w_0\CG^{0}};
\draw[->] (6,0.15) to (7,0.15);
\node at (7.5,0) {\cdots};
\draw[->] (8,0.15) to (9,0.15);
\node at (9.8,0) {w_0\CG^{\ell(\br^2)}};
\draw[->] (9.8, -0.1) to (9.8, -0.9);
\node at (10,-1.3) {w_0\CG^{\ell(\br^2)+1}};
\draw[->] (9, -1.15) to (8, -1.15);
\node at (7.5,-1.3) {\cdots};
\draw[->] (7, -1.15) to (6.5, -1.15);
\node at (5,-1.3) {\CF^{\std} = w_0\CG^{\ell(\br^2\lift{w})}};
\draw[->] (3.5, -1.15) to (2.8, -1.15);
\node at (2,-1.3) {\cdots};
\draw[->] (1.5, -1.15) to (0.8, -1.15);
\node at (0, -1.3) {\CF^{\ell(w^{-1}w_0)}};
\draw[<-] (0, -0.1) to (0, -0.9);

\node at (2,0.6) {\color{blue} \br^1};
\draw[dashed, color=blue] (-1, -0.2) -- (4.5, -0.2) -- (4.5, 1) -- (-1, 1) -- cycle; 

\node at (2, -1.7) {\color{teal} \lift{w^{-1}w_0}};
\draw[dashed, color=teal] (4.3, -0.7
 ) -- (4.3, -2) -- (-1, -2) -- (-1, 0.5) -- (1
, 0.5) -- (1, -0.7) -- cycle;

\node at (7.5, 0.6) {\color{red} \br^2};
\draw[dashed, color=red] (5,1) -- (10.5,1) -- (10.5, -0.2) -- (5, -0.2) -- cycle;

\node at (7.5, -1.7) {\color{orange} \lift{w}};
\draw[dashed, color=orange] (4.7, -2) -- (11, -2) -- (11, 0.5) -- (9.1, 0.5) -- (9.1, -0.7) -- (4.7, -0.7) -- cycle;
\end{tikzcd}
\end{equation}

Since $\lift{w_{}}\cdot \lift{w^{-1}w_0}$ is a reduced lift of the longest element $w_0$, we see (moving on the orange and then teal parts of \eqref{eq:def-inverse}) that the flags $w_0\CG^{\ell(\br^2)}$ and $\CF^{\ell(w^{-1}w_0)+1}$ 
are transverse. So we can find $g \in \GL(k)$ such that $g\CF^{\ell(w^{-1}w_0)+1} = \CF^{\std}$ and $gw_0\CG(\ell(\br^2)) = \CF^{\ant}$. Translating all the flags in \eqref{eq:def-inverse} by $g$, we see that:
\begin{itemize}
\item[(1)] The flags on the top row constitute an element of $X(\br^1\br^2) = X(\br)$.
\item[(2)] Since $g\CF^{\ell({w^{-1}w_0})+1} = \CF^{\std}$, $gw_0\CG^{\ell(\br^2)} = \CF^{\ant}$ and $\lift{w_{}}\cdot\lift{w^{-1}w_0}$ is a reduced lift of $w_0$, the translations by $g$ of all the flags in the bottom row are coordinate flags, so the flag $g\CF^{\std}$ is $\CF(w_0w)$ and $g\CF^{\ant} \pitchfork g\CF^{\std} = \CF(w_0w)$, i.e., the element of $X(\br)$ obtained in (1) belongs to $\CU_{r_1,w}(\beta)$. 
\end{itemize}
It only remains to specify the element $g \in \GL(k)$ uniquely, as it is only defined up to multiplication on the left by a diagonal matrix. For this, we specify concrete matrices that project to the flags in \eqref{eq:def-inverse}. For the flags in $\CF$, these are the braid matrices $B_{i_1}(z_1), \dots, B_{\br^1}(z_1, \dots, z_{r_1}) = w_0V$, while the matrices for the flags in $\CG$ are the braid matrices translated by the upper triangular matrix $V$. Once we have specified $M_1$ and $M_2$ such that $\CF^{\ell(w^{-1}w_0)+1} = \CF(M_1)$ and $w_0\CG^{\ell(\br^2)} = \CF(M_2)$, the fact that $\CF(M_1) \pitchfork \CF(M_2)$ means that $M_2^{-1}M_1$ belongs to $\borel w_0 \borel$ and thus there exists a unique decomposition $M_2^{-1}M_1 = U'w_0V'$ with $U'$ unipotent. We then take $g = w_0(U')^{-1}M_2^{-1}$. We leave details to the reader. 
\end{proof}

\begin{remark}
    The proof of Theorem \ref{thm:splicing-braid-varieties} can be simplified by taking a slightly different realization of the braid variety $X(\br)$ using weighted flags as in \cite{GLSB, GLSBS22}, see also \cite{comparisonpaper}. Let $\unipotent \subseteq \GL(k)$ be the subgroup of upper uni-triangular matrices. A weighted flag is an element of $\GL(k)/\unipotent$. Two weighted flags $h\unipotent$ and $h'\unipotent$ are said to be in strong relative position $w \in S_k$ if there exists $g \in \GL(k)$ such that $g(h\unipotent) = \unipotent$ and $g(h'\unipotent) = w\unipotent$. We denote this by $h\unipotent \relpos{w} h'\unipotent$. Similarly, $h\unipotent$ and $h'\unipotent$ are said to be in weak relative position $w \in S_k$ if there exist $g \in \GL(k)$ and a diagonal matrix $t$ such that $g(h\unipotent) = \unipotent$ and $g(h'\unipotent) = tw\unipotent$. We denote this relation by $h\unipotent \wrelpos{w} h'\unipotent$. Given a braid $\br = \sigma_{i_1}\cdots \sigma_{i_r}$ we then have an isomorphism
    \[
    X(\br) \cong \left\{h_0\unipotent \relpos{s_{i_1}} h_1\unipotent \relpos{s_{i_2}} \cdots \relpos{s_{i_r}} h_r\unipotent \wrelpos{w_0} h_0\unipotent\right\}/\GL(k)
    \]
    where $\GL(k)$ acts diagonally on all weighted flags $h_1\unipotent, \dots, h_r\unipotent$. The set $\CU_{r_1, w}(\beta)$ is then defined to consist of those chains of weighted flags such that $h_{r_1}\unipotent \wrelpos{w_0} w_0w\unipotent$. An advantage of working with weighted flags is that, since $\unipotent \cap (w_0\borel w_0) = \{\id\}$, now we have uniqueness of the translating element in all the arguments used above.  
\end{remark}


\begin{remark}
A priori, the construction of the map $\Phi_{r_1,w}$ depends on the choice of the reduced expression \eqref{eq: choice w0}. However, one can check that different choices of the reduced expression lead to, essentially, the same map. More precisely, if we have two reduced expressions $\lift{w}$, $\lift{w}'$ of $w$, choose reduced expressions $\lift{w_0}, \lift{w_0}'$ that contain $\lift{w}, \lift{w}'$ as a prefix, respectively. We have canonical isomorphisms $X\left(\br^2\lift{w}\right) \to X\left(\br^2\lift{w}'\right)$ and $X\left(\lift{w^{-1}w_0}\br^1\right) \to X\left(\lift{w^{-1}w_0}'\br^1\right)$ such that the following diagram commutes
\begin{equation*}
    \begin{tikzcd}
        \node at (0,0) {\CU_{r_1, w}(\br)};
        \node at (-4, -1) {X\left(\lift{w^{-1}w_0}\br^1\right) \times X\left(\br^2\lift{w}\right)};
        \node at (4, -1) {X\left(\lift{w^{-1}w_0}'\br^1\right) \times X\left(\br^2\lift{w}'\right).};
        \draw[->] (-1.7, -0.85) -- (1.6, -0.85);
        \draw[->] (-1, -0.2) -- (-3.5, -0.7);
        \draw[->] (0.8, -0.2) -- (4.35, -0.65);
    \end{tikzcd}
\end{equation*}
In particular, the flags in the bottom part of the diagram \eqref{eq:def-inverse} are determined (for a given reduced expression \eqref{eq: choice w0}) by $\CF^{\ell(w^{-1}w_0)+1}$ and $w_0\CG^{\ell(\br^2)}.$ Furthermore, the element $g$ is completely determined by these two flags.
\end{remark}

From Theorem \ref{thm:splicing-braid-varieties} and its proof we obtain the following result.

\begin{corollary}
\label{cor: U nonempty}
The open set $\CU_{r_1, w}(\beta) \subseteq X(\br)$ is nonempty if and only if $$\delta\left(\lift{w^{-1}w_0}\br^1\right) = w_0 =\delta\left(\br^2\lift{w}\right).$$
By Lemma \ref{lem:demazure-bounds}, this happens if and only if $\delta(\br^1) \geq w_0ww_0$ and $\delta(\br^2) \geq w_0w^{-1}$.
\end{corollary}

\subsection{Conjectural properties}\label{sec:conj-properties} We conjecture that the set $\CU_{r_1, w}(\beta)$, and the map $\Phi_{r_1,w}=\Phi_1 \times \Phi_2$ satisfy various desirable cluster-theoretic properties. To state this conjecture precisely, we need some notation. If $Q$ is an ice quiver, we denote by $Q^{\mathrm{uf}}$ the quiver obtained by deleting the frozen vertices and all arrows adjacent to them. We also denote by $Q_1$ the ice quiver from the cluster structure on $\C\left[X\left(\lift{w^{-1}w_0}\br^1\right)\right]$, and by $Q_2$ the ice quiver associated to $\C\left[X\left(\br^2\lift{w}\right)\right]$. 

\begin{conjecture}\label{conj:splicing-braid-varieties}
Assume $\CU_{r_1, w}(\br)$ is nonempty, and consider the cluster structure on $\C[X(\br)]$ from Section \ref{sec:cluster-structure}. There exists a seed $\Sigma = (Q, \mathbf{x})$ satisfying the following properties:
\begin{enumerate}
    \item There exist cluster variables $x_{a_1}, \dots, x_{a_s} \in \mathbf{x}$ such that $\CU_{r_1, w}(\br)$ is the common nonvanishing locus of $x_{a_1}, \dots, x_{a_s}$.
    \item The variety $\CU_{r_1, w}(\br)$ admits a cluster structure, and an initial seed $\widehat{\Sigma} = \left(\widehat{Q}, \widehat{\mathbf{x}}\right)$ is given by freezing the variables $x_{a_1}, \dots, x_{a_s}$ in the seed $\Sigma = (Q, \mathbf{x})$.
    \item The quiver $\widehat{Q}^{\mathrm{uf}}$ is mutation equivalent to the disjoint union of the quivers $Q_1^{\mathrm{uf}
    }$ and $Q_2^{\mathrm{uf}}$.
    \item Note that the variety $X\left(\lift{w^{-1}w_0}\br^1\right) \times X\left(\br^2\lift{w}\right)$ admits a natural product cluster structure. Then, the map $\Phi_{r_1,w}: \CU_{r_1, w}(\br) \to X\left(\lift{w^{-1}w_0}\br^1\right) \times X\left(\br^2\lift{w}\right)$ is a cluster quasi-isomorphism.
    \end{enumerate}
\end{conjecture}

The items (1)--(4) in Conjecture \ref{conj:splicing-braid-varieties} are not independent.

\begin{lemma}\label{lem:conditional-conjecture}
    Assume (1) of Conjecture \ref{conj:splicing-braid-varieties} holds. Then, we have that (4) $\Rightarrow$ (3) $\Rightarrow$ (2). Moreover, (3) implies that the cluster structures on $X\left(\lift{w^{-1}w_0}\br^1\right) \times X\left(\br^2\lift{w}\right)$ and $\CU_{r_1, w}(\br)$ are abstractly cluster quasi-isomorphic (but it does not guarantee that the map $\Phi_{r_1,w}$ is a cluster quasi-isomorphism.)
\end{lemma}
\begin{proof}
Assume (1). Cluster quasi-isomorphisms do not affect the mutable part of the cluster structure, see Remark \ref{rmk:Fraser}, so (4) implies (3). 

Assume (3) holds. By \cite[Theorem 7.13]{CGGLSS}, the cluster structures on $\C[X(\lift{w^{-1}w_0}\br^1]$ and $\C[X(\br^2\underline{w})]$ are locally acyclic, so the cluster structure given by the seed $\widehat{\Sigma}$ is also locally acyclic, see \cite[Proposition 3.10]{Muller}. 
From here,  (2) follows from Lemma 3.4 and Theorem 4.1 in \cite{Muller}. 

Let us show that, in the presence of (1), (3) already implies that the cluster structures on $X(\lift{w^{-1}w_0}\br^1) \times X(\br^w\lift{w}_{})$ and $\CU_{r_1, w}(\br)$ are cluster quasi-isomorphic. For this, we use exchange matrices and Lemma \ref{lem:quasi-iso-exchange-matrix} above. We have the following extended exchange matrices:
\begin{itemize}
\item $\widetilde{B} = \begin{pmatrix} B \\ \hline C\end{pmatrix}$, the extended exchange matrix for the cluster structure on $X(\br)$. By (1) and (2), the extended exchange matrix $\widetilde{B}^{\circ}$ for the cluster structure on $\CU_{r_1, \br}$ is given by deleting some columns on $\widetilde{B}$ and freezing (moving to the bottom) the corresponding rows. By \cite[Corollary 8.5]{CGGLSS}, the exchange matrix $\widetilde{B}$ has really full rank, meaning that its rows span $\Z^{n}$, where $n$ is the number of mutable variables in $X(\br)$. Note that it follows that $\widetilde{B}^{\circ}$ has really full rank as well.
\item $\widetilde{B}_1 = \begin{pmatrix} B_1 \\ \hline C_1\end{pmatrix}$, the extended exchange matrix for the cluster structure on $X(\lift{w^{-1}w_0}\br^1)$. We also have the extended exchange matrix $\widetilde{B}_2$ for the cluster structure on $X(\br^2\lift{w_{}})$. Note that the extended exchange matrix for the cluster structure on the product $X(\lift{w^{-1}w_0}\br^1) \times X(\br^2\lift{w})$ has the form
\[
\widetilde{B}^{\times} = \begin{pmatrix} B_1 & 0 \\ 0 & B_2 \\ \hline C_1 & 0 \\ 0 & C_2  \end{pmatrix}.
\]
\end{itemize} 
Note that by (3), we can assume that, maybe after mutations, the matrix $\widetilde{B}^{\circ}$ has the following form
\[
\widetilde{B}^{\circ} = \begin{pmatrix} B_1 & 0 \\ 0 & B_2 \\ \hline C_{11} & C_{12} \\ C_{21} & C_{22} \end{pmatrix} 
\]
And our job is to produce a matrix $R$ as in Lemma \ref{lem:quasi-iso-exchange-matrix} such that $R\widetilde{B}^{\times} = \widetilde{B}^{\circ}$. We consider first the following matrices
\[
\widetilde{B}_{1}^{\times} := \begin{pmatrix} B_1 \\ 0 \\ \hline C_1 \\ 0\end{pmatrix}, \qquad \widetilde{B}_1^{\circ} = \begin{pmatrix} B_1 \\ 0 \\ \hline C_{11} \\ C_{12} \end{pmatrix}.
\]
Since $\widetilde{B}_1$ has really full rank, the same is true for $\widetilde{B}_1^{\times}$. So we can express the rows of $C_{11}$ and $C_{12}$ as an integer linear combination of the rows of $B_1$ and $C_1$. This means that we can find a matrix $R_1 = \begin{pmatrix} \id & 0 \\ P_1  & Q_1\end{pmatrix}$ so that $R_1\widetilde{B}_1^{\times} = \widetilde{B_1}^{\circ}$. 
Note that $Q_1$ is invertible over $\Z$, since the matrix $\widetilde{B}_1^{\circ}$ has really full rank. Similarly, we can find $R_2 = \begin{pmatrix} \id & 0 \\ P_2 & Q_2\end{pmatrix}$ such that $R_2\widetilde{B}_2^{\times} = \widetilde{B}_2^{\circ}$. It is now easy to see that 
\[R = \begin{pmatrix} \id & 0 & 0 & 0 \\ 0 & \id & 0 & 0 \\ P_1 & 0 & Q_1 & 0 \\ 0 & P_2 & 0 & Q_2\end{pmatrix}
\]
satisfies the required properties.
\end{proof}

\begin{remark}\label{rem: frozen inequality}
Let $f, f_1,f_2$ respectively denote the numbers of frozen variables for $X(\beta),X\left(\lift{w^{-1}w_0}\br^1\right)$ and $X\left(\br^2\lift{w}\right)$.
By \cite[Theorem 1.3]{GLS} the group $\C[X]^{\times}$ of global invertible functions on a cluster variety $X$ is an 
abelian group generated by monomials in frozens and nonzero scalars. In particular, number of frozen variables does not depend on a choice of a seed and is an invariant of a cluster variety. Theorem \ref{thm:splicing-braid-varieties} and the inclusion 
$$
\C[X(\beta)]^{\times}\hookrightarrow \C[\CU_{r_1,w}]^{\times}
$$
imply the inequality
\begin{equation}
\label{eq: frozen inequality}
f_1+f_2\ge f.
\end{equation}

Assuming Conjecture \ref{conj:splicing-braid-varieties}, we can compute the number $s$ of cluster variables $x_{a_1},\ldots,x_{a_s}$ that need to be frozen in (1) or (2).   
 By comparing the number of frozen variables for $\CU_{r_1,w}(\beta)$ and their image under $\Phi_{r_1,w}$, we arrive at the equation 
$
f+s=f_1+f_2,
$
so
\[
s = f_1+f_2-f.
\]
\end{remark}


We can verify (1), (2) and (3) of Conjecture \ref{conj:splicing-braid-varieties} in the extreme cases of $w = e$ and $w = w_0$. 
\begin{proposition}
Assume $w = e$ or $w = w_0$. Then, (1), (2) and (3) of Conjecture \ref{conj:splicing-braid-varieties} holds.
\end{proposition}
\begin{proof}
Assume $w = w_0$, so that $\CU_{r_1, w_0}(\br)$ is given by the condition that $\CF^{r_1}$ is transverse to $\CF^{\std}$, and $\CU_{r_1, w_0}(\br)$ is nonempty if and only if $\delta(\br^1) = w_0$. In this case, the seed $\Sigma$ predicted by Conjecture \ref{conj:splicing-braid-varieties} is the left-to-right inductive seed, whose corresponding cluster torus is precisely the Deodhar torus $\ltrt_{\br}$.  It follows from \cite[Section 7.1]{CGGLSS} that a cluster variable $x_{a}$ is nowhere vanishing on $\CU_{r_1, w_0}(\br)$ if and only if the Lusztig cycle associated to $x_a$ intersects the horizontal slice of the inductive weave right after obtaining $\delta(\br^1) = w_0$, and that these are the cluster variables predicted by Conjecture \ref{conj:splicing-braid-varieties} (1). By construction, cf. \cite[Section 4]{CGGLSS} the  mutable part of the quiver obtained after freezing these cluster variables is equal to the disjoint union of the mutable parts of the quivers $Q_1$ and $Q_2$, so (3) is valid and by Lemma \ref{lem:conditional-conjecture} (2) also holds. Moreover, again by Lemma \ref{lem:conditional-conjecture}, the cluster structures on $\CU_{r_1, w_0}$ and $X(\br^1) \times X(\br^2w_0)$ are abstractly quasi-isomorphic. The case $w = e$ is similar, taking the right-to-left inductive weave instead. 
\end{proof}

\begin{example}
Consider the braid word $\br = \sigma_2\sigma_1\sigma_3\sigma_2\sigma_2\sigma_3\sigma_1\sigma_2\sigma_2\sigma_1\sigma_3\sigma_2$, that was considered in \cite[Section 11.4]{CGGLSS}. As explained in \emph{loc. cit.}, for the left-to-right inductive seed, we have the cluster variables
$$x_1=z_5,\;\, x_2=-z_6z_7 + z_5z_8,\;\, x_3=-z_6z_7z_9+z_5z_8z_9-z_5, \;\,
x_4=-z_6z_9 + z_5z_{10},\;\,x_5=-z_7z_9 + z_5z_{11},$$ $$x_6= z_6z_7z_{10}z_{11} - z_5z_8z_{10}z_{11} - z_6z_7z_9z_{12} + z_5z_8z_9z_{12} - z_8z_9 + z_7z_{10} + z_6z_{11} - z_5z_{12} + 1,$$

the variables $x_4, x_5, x_6$ are frozen and the quiver $Q$ is
\begin{center}
    \begin{tikzcd}
    x_1 \arrow{rr} \arrow{dr} \arrow{drrr} & & x_2 \arrow {rr} & & x_3 \arrow[out=160, in=20, "2"]{llll} \arrow{dr}  &  \\ & {\color{blue} x_4} \arrow{urrr} & & {\color{blue} x_5} \arrow{ur} & & {\color{blue} x_6}.
    \end{tikzcd}
\end{center}

Let us take $r_1 = 9$, so that $\br^1 = \sigma_2\sigma_1\sigma_3\sigma_2\sigma_2\sigma_3\sigma_1\sigma_2\sigma_2$ and $\br^2 = \sigma_1\sigma_3\sigma_2$, and $w = w_0$, so that $\CU_{9, w_0}(\br)$ is given by the condition that $\CF^9$ is transverse to $\CF^{\std}$. The flag $\CF^9$ is the flag associated to the matrix $B_{\beta^1}(z_1, \dots, z_9)$:
\begin{equation}\label{eq:long-example-1}
 \begin{pmatrix} -z_4z_5+z_2z_7+1 & z_4z_6z_9 - z_2z_8z_9 + z_2 & -z_4z_6+z_2z_8 & -z_4 \\ -z_3z_5+z_1z_7 & z_3z_6z_9-z_1z_8z_9+z_1+z_9 & -z_3z_6 + z_1z_8 -1 & -z_3 \\ z_7 & -z_8z_9 + 1 & z_8 & 0 \\ z_5 & -z_6z_9 & z_6 & 1  \end{pmatrix}
\end{equation}
So $\CF^9$ is transverse to $\CF^{\std}$ if and only if the lower-left justified minors of \eqref{eq:long-example-1} are nonzero. These are,
\[
\begin{vmatrix} z_5 \end{vmatrix} = x_1, \quad \begin{vmatrix}z_7 & -z_8z_9+1 \\ z_5 & -z_6z_9\end{vmatrix} = x_3,\] \[\begin{vmatrix} -z_3z_5+z_1z_7 & z_3z_6z_9-z_1z_8z_9+z_1+z_9 & -z_3z_6 + z_1z_8 -1  \\ z_7 & -z_8z_9 + 1 & z_8  \\ z_5 & -z_6z_9 & z_6  \end{vmatrix} = x_1, 
\]

so that $\CU_{9, w_0}(\br)$ is the cluster variety associated to the seed

\begin{center}
    \begin{tikzcd}
    {\color{blue}x_1} \arrow{rr}  & & x_2 \arrow {rr} & & {\color{blue} x_3}   &  \\ & {\color{blue} x_4}  & & {\color{blue} x_5}  & & {\color{blue} x_6}.
    \end{tikzcd}
\end{center}
\vspace{0.25cm}
Note that the top row of this seed is isomorphic to a seed for $X(\br^1)$, while the bottom part of the seed is isomorphic to a seed for $X(\lift{w_0}\br^2)$, so (1), (2) and (3) of Conjecture \ref{conj:splicing-braid-varieties} hold. 

Combinatorially, splicing amounts to cutting the left-to-right inductive weave along the dotted line in Figure \ref{fig:weave-splicing}.

\begin{figure}
    \centering
    \includegraphics[scale=1.2]{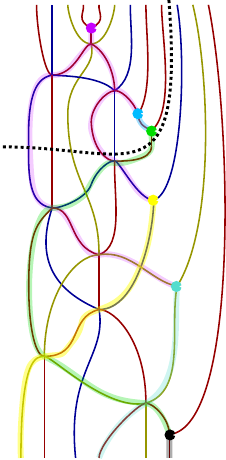}
    \caption{The left-to-right inductive weave for the braid $\br = \sigma_2\sigma_1\sigma_3\sigma_2\sigma_2\sigma_3\sigma_1\sigma_2\sigma_2\sigma_1\sigma_3\sigma_2$. Taking $r_1 = 9$, the part of the weave above the dotted line is an inductive weave for $\br^1$, while the part of the weave below the dotted line is an inductive weave for $\lift{w_0}\br^2$, and the braid varieties $X(\lift{w_0}\br^2)$ and $X(\br^2\lift{w_0})$ are cluster quasi-isomorphic.}
    \label{fig:weave-splicing}
\end{figure}
\end{example}

\begin{example}
Now consider the braid word $\br = \sigma_2\sigma_1\sigma_3\sigma_2\sigma_2\sigma_3\sigma_1\sigma_1\sigma_2\sigma_3\sigma_2\sigma_2\sigma_3\sigma_2\sigma_1$. The quiver for the left-to-right inductive weave is
\begin{center}
    \begin{tikzcd}
        x_1 \arrow[out=20,in=160]{rrrrrr} \ar{drrr} & & x_2 \arrow{dr} & & x_3 \arrow[out=340,in=200]{rrrr} & & x_4 \arrow{ll} \arrow{rd} & & x_5 \arrow{rr} & & x_6 \arrow[out=160, in=20]{llll} \arrow{llllld} \\
         & & & {\color{blue} x_9}  & & {\color{blue} x_8} \arrow{ulll} \arrow{ulllll} & & {\color{blue} x_7} \arrow{urrr} & & & 
    \end{tikzcd}
\end{center}
and the extended exchange matrix is:
\begin{equation}\label{eq:exchange-matrix-example}
\widetilde{B} = \begin{pmatrix}
0 & 0 & 0 & 1 & 0 & 0 \\ 
0 & 0 & 0 & 0 & 0 & 0 \\
0 & 0 & 0 & -1 & 1 & 0 \\ 
-1 & 0 & 1 & 0 & 0 & -1 \\
0 & 0 & -1 & 0 & 0 & 1 \\
0 & 0 & 0 & 1 & -1 & 0 \\
\hline 
0 & 0 & 0 & -1 & 0 & 1 \\
1 & 1 & 0 & 0 & 0 & -1\\
-1 & -1 & 0 & 0 & 0 & 0
\end{pmatrix}.
\end{equation}
Now let us take $r = 11$. The flag $\CF^{11}$ is associated to the matrix $B_{\br_{11}}(z_1, \dots, z_{11})$, and a computation (using e.g. Sage) shows that the lower-left justified minors of this matrix are:
\[
x_1x_2, \qquad x_5, \qquad x_4, \qquad 1.
\]
So that the cluster structure on the open set $\CU_{11,e}$ is obtained by freezing the cluster variables $x_1, x_2, x_4, x_5$. Note that there are no arrows between $x_3$ and $x_6$, so the mutable part of the quiver indeed becomes disconnected after freezing.

Freezing the corresponding rows, and deleting the corresponding columns, in \eqref{eq:exchange-matrix-example} we obtain the following extended exchange matrix
\begin{equation}\label{eq:exchange-matrix-after-freezing}
\widetilde{B}^{\circ} = \begin{matrix} \mathsmaller{\mathsmaller{3}} \\ \mathsmaller{\mathsmaller{6}} \\\hline  \mathsmaller{\mathsmaller{1}}\\ \mathsmaller{\mathsmaller{2}} \\ \mathsmaller{\mathsmaller{4}} \\ \mathsmaller{\mathsmaller{5}} \\ \mathsmaller{\mathsmaller{7}} \\ \mathsmaller{\mathsmaller{8}} \\ \mathsmaller{\mathsmaller{9}} \end{matrix}\!\begin{pmatrix} 0 & 0 \\ 0 & 0 \\ \hline 0 & 0 \\ 0 & 0 \\ 1 & -1 \\ -1 & 1 \\ 0 & 1 \\ 0 & -1 \\ 0 & 0 \end{pmatrix}.
\end{equation}
On the other hand, the extended exchange matrices for the cluster structures on $X(\br^1)$, $X(\Delta\br^2)$, and $X(\br^1) \times X(\Delta\br^2)$ are, respectively:
\[
\widetilde{B}_1 = \begin{pmatrix} 0 \\ \hline 0 \\ 0 \\ 1 \\ -1\end{pmatrix}, \qquad \widetilde{B}_2 = \begin{pmatrix} 0 \\ \hline 1 \\ -1 \\ 0 \end{pmatrix}, \qquad \widetilde{B}^{\times} = \begin{pmatrix} 0 & 0 \\ 0 & 0 \\ \hline 0 & 0 \\ 0 & 0 \\ 1 & 0 \\ -1 & 0 \\ 0 & 1 \\ 0 & -1 \\ 0 & 0\end{pmatrix}.
\]
We verify that $\widetilde{B}^{\circ}$ and $\widetilde{B}^{\times}$ define quasi-equivalent cluster structures. Following the strategy of Lemma \ref{lem:quasi-iso-exchange-matrix}, we need to find a $9 \times 9$ invertible integer matrix $R$ of the form $R = \begin{pmatrix} \id_{2} & 0 \\ P & Q\end{pmatrix}$ such that $R\widetilde{B}^{\circ} = \widetilde{B}^{\times}$. It is easy to see that taking $P = 0$ and
\[
Q = \begin{pmatrix}
  1 & 0 & 0 & 0 & 0 & 0 & 0 \\
  0 & 1 & 0 & 0 & 0 & 0 & 0  \\
  0 & 0 & 1 & 0 & 1 & 0 & 0 \\
  0 & 0 & 0 & 1 & 0 & 1 & 0 \\
  0 & 0 & 0 & 0 & 1 & 0 & 0 \\
  0 & 0 & 0 & 0 & 0 & 1 & 0 \\
  0 & 0 & 0 & 0 & 0 & 0 & 1
\end{pmatrix}
\]
works. Note, however, that the induced quasi-isomorphism preserves the mutable variables $x_3, x_6$, so it is unlikely to coincide with the isomorphism $\Phi_{11, e}$.
\end{example}

\subsection{Open Richardson varieties}\label{sec:richardson} We apply our results to the setting of open Richardson varieties in the flag variety. Recall that if $u \leq w \in S_k$ are permutations, the open Richardson variety is
\[
R(u, w) := \left\{\CF \in \Fl_k \mid \CF^{\std} \relpos{w} \CF \xrightarrow{u^{-1}w_0}\CF^{\ant}\right\}.
\]
The variety $R(u,w)$ is nonempty provided $u \leq w$, in which case it is an affine, smooth variety of dimension $\ell(w) - \ell(u)$. In fact, we have an isomorphism
\begin{equation}\label{eq:iso-braid-richardson}
X\left(\lift{w_{}}\cdot\lift{u^{-1}w_0}\right) \to R(u, w),
\end{equation}
that takes a sequence of flags $(\CF^0 = \CF^{\std}, \CF^1, \dots, \CF^{\ant})$ to the flag $\CF^{\ell(w)}$, cf. \cite[Theorem 3.14]{CGGLSS}.

\begin{proposition}\label{prop:Richardson}
    Let $u \leq v \leq w \in S_k$. Then, there exists an open embedding
\begin{equation}\label{eq:Richardson}
\Psi_{u, v, w}: R(u,v) \times R(v,w) \to R(u,w).
\end{equation}
whose image is the set
\begin{equation}\label{eq:Richardson-image}
\CU_{u, v, w} = \{\CF \in R(u, w) \mid \CF\pitchfork \CF(vw_0)\}.
\end{equation}
Moreover, if $u, v, w \in S_k$ are permutations so that $\CU_{u, v, w}$ is nonempty, then $u \leq v \leq w$. 
\end{proposition}
\begin{proof}
    We use the isomorphism \eqref{eq:iso-braid-richardson} together with a special case of the maps from Theorem \ref{thm:splicing-braid-varieties}, so we work with the variety $X\left(\lift{w_{}}\cdot \lift{u^{-1}w_0}\right) \cong R(u,w)$.  In the setting of Section \ref{sec:splicing} we set $r_1 = \ell(w)$, so that $\br^1 = \lift{w}$ and $\br^2 = \lift{u^{-1}w_0}$. Consider $v^{\ast} := w_0vw_0$. Then,
\[
\CU_{r_1, v{^{\ast}}} = \left\{\left(\CF^0 = \CF^{\std}, \dots, \CF^{\ant}\right) \in X\left(\lift{w_{}}\cdot\lift{u^{-1}w_0}\right) \mid \CF^{\ell(w)} \pitchfork \CF\left(w_0v^{\ast}\right)\right\} 
\]
maps isomorphically onto $\CU_{u, v, w}$ upon the isomorphism $X\left(\lift{w_{}}\cdot\lift{u^{-1}w_0}\right) \cong R(u,w)$. On the other hand, by Theorem \ref{thm:splicing-braid-varieties} we have
\[
\CU_{r_1, v^{\ast}} \cong X\left(\lift{(v^{\ast})^{-1}w_0}\cdot \lift{w_{}^{}}\right) \times X\left(\lift{u^{-1}w_0^{}}\cdot\lift{v^{\ast}_{}}\right).
\]
Using cyclic rotation isomorphisms (see \cite[Section 5.5]{CGGLSS}) 
we moreover have $$X\left(\lift{u^{-1}w_0^{}}\cdot \lift{v^{\ast}_{}}\right) \cong X\left(\lift{v_{}}\cdot\lift{u^{-1}w_0}\right) \cong R(u,v),$$ while $$X\left(\lift{(v^{\ast})^{-1}w_0}\cdot \lift{w_{}^{}}\right) \cong X\left(\lift{w_{}}\cdot\lift{v^{-1}w_0}\right) \cong R(v,w).$$ So $\CU_{r_1, v^{\ast}} \cong R(u, v) \times R(v,w)$ and we obtain the embedding \eqref{eq:Richardson}. The last claim of the statement of the proposition follows from Corollary \ref{cor: U nonempty}.
\end{proof}

\begin{remark}
    In \cite[Definition 3.1]{eberhardt-stroppel}, an open embedding $R(u,v) \times R(v,w) \to R(u,w)$ is obtained. We do not know if it coincides with the map constructed in Proposition \ref{prop:Richardson}.
\end{remark}

\begin{remark}
\label{rem: reduced is richardson}
    Note that if $\br = \br^1\br^2$, with $\br^1$ and $\br^2$ both reduced words, then the condition $\delta(\br) = w_0$ forces $\br^1 = \lift{w}$ and $\br^2 = \lift{v^{-1}w_0}$ for some $v \leq w$. Thus, any splicing map with $\br^1$ and $\br^2$ both reduced is equivalent to one of the form \eqref{eq:Richardson}. 
\end{remark}

We can use Proposition \ref{prop:Richardson} to give a (non-explicit) formula for the number of frozen variables on Richardson varieties. For $v \leq w$, write $f_{v,w}$ for the number of frozen variables on the Richardson variety $R(v,w)$. 

\begin{corollary}\label{cor:frozen-richardson}
Let $v \leq w \in S_k$. Let $s_{v,w}$ be the number of distinct irreducible factors appearing in the minors $\minor_{v[i], [i]}(B_{\lift{w}}(z_1, \dots, z_{\ell(w)}))$ which are not themselves principal minors of the matrix $B_{\lift{w}}(z_1, \dots, z_{\ell(w)})$. Then,
\begin{equation}\label{eq:Richardson-frozens}
f_{v,w} = f_{
e,w} - f_{e,v} + s_{v,w}.
\end{equation}
\end{corollary}
\begin{proof}
    We have an embedding $R(e,v) \times R(v,w) \to R(e,w)$ whose image is the set $\CU_{e,v,w}$ of flags $\CF \in R(e,w)$ that are transverse to $\CF(vw_0)$. The set of all flags in $R(e,w)$ is parametrized by those matrices $B_{\lift{w}}(z_1, \dots, z_{\ell(w)})$ with nonvanishing principal minors, and a flag $B_{\lift{w}}(z_1, \dots, z_{\ell(w)})$ is transverse to $\CF(w_0v)$ if and only if $\minor_{v[i],[i]}(B_{\lift{w}}(z_1, \dots, z_{\ell(w)})$ is nonzero for every $i = 1, \dots, k$, cf. Lemma \ref{lem: relpos w minors}(a). So $s_{v,w}$ is precisely the number of irreducible elements in $\C[R(e,w)]$ that we have to localize in order to obtain $\C[\CU_{e,v,w}]$ and we get
    \[
    f_{e,v} + f_{v,w} = f_{e,w} + s_{v,w},
    \]
    which proves the result.
\end{proof}

\begin{remark}
    For $w \in S_k$, the number $f_{e,w}$ is simply the number of different simple generators of $S_k$ appearing in one (equivalently, any) reduced expression for $w$. So by \eqref{eq:Richardson-frozens} the computation of $f_{v,w}$ is equivalent to that of $s_{v,w}$, which still remains a challenging problem.
\end{remark}

\begin{example}\label{ex:frozens-richardson}
    Let us consider $k = 4$, $v = s_2$ and $w = s_3s_2s_1s_2s_3$. We have
    \[
    B_{\lift{w}}(z_1, \dots, z_5) = \begin{pmatrix} z_3 & -z_4 & z_5 & -1 \\ z_2 & -1 & 0 & 0 \\ z_1 & 0 & -1 & 0 \\ 1 & 0 & 0 & 0 \end{pmatrix}.
    \]
    Note that $v[i] \neq [i]$ if and only if $i = 2$, so we only need to check the irreducible factors of $\minor_{[1,3],[1,2]}(B_{\lift{w}}(z_1, \dots, z_5)) = z_1z_4$. Neither $z_1$ nor $z_4$ is a principal minor, so we obtain $s_{v,w} = 2$ and $f_{v,w} = 3 - 1 + 2 = 4$. Since $\dim(R(v,w)) = 4$, we obtain that $R(v,w) \cong (\C^{\times})^{4}$ is a $4$-dimensional torus.
    The left-to-right inductive initial seed for the cluster structure on $R(e,w)$ is
    \begin{center}
        \begin{tikzcd}
          {\color{blue} z_3} \arrow{r} & z_2 \arrow{r} & {\color{blue} z_4z_2 - z_3} \arrow{r} & z_1 \arrow{r} &    {\color{blue} z_1z_5 + z_3 - z_2z_4}
        \end{tikzcd}
    \end{center}
    Note that mutating at $z_2$ we obtain $z_2' = z_4$, so there is a seed in $R(e,w)$ with mutable variables $z_1, z_4$. 
\end{example}

\begin{remark}\label{rmk:growth-frozens}
Note that Corollary \ref{cor:frozen-richardson} is the easiest to apply when $v$ is a simple transposition $s_i$: in this case we only need to find irreducible factors of a single minor $\minor_{s_i[i], [i]}$, since $\minor_{s_i[j], [j]}$ is a principal minor for $j \neq i$.

Example \ref{ex:frozens-richardson} can be generalized as follows. Let $k = 2i$, and let $w = [k, 2, 3, \dots, k-1, 1]$, i.e., $w$ is the transposition $(1k)$. Let $v = s_2s_4\cdots s_{k-2}$. Up to signs in the pivotal $1$'s, the matrix $B_{\lift{w}}$ has the following form
\[
B_{\lift{w}} = \begin{pmatrix} p_{k-1} & p_{k} & p_{k+1} & \dots & p_{2k-3} & 1 \\ p_{k-2} & 1 & 0 & \dots & 0 & 0 \\ p_{k-3} & 0 & 1 & \dots & 0 & 0 \\
\vdots & \vdots & \vdots & \ddots & \vdots & \vdots \\ p_1 & 0 & 0 & \dots & 1 & 0 \\ 1 & 0 & 0 & \dots & 0 & 0\end{pmatrix} 
\]
where $p_1, \dots, p_{2k-3}$ are irreducible polynomials in the $z$-variables. Among the minors $\minor_{v[j], [j]}$, the ones that are not already principal minors are (up to signs):
\[
\minor_{v[2], [2]} = p_{k-3}p_{k}, \; \minor_{v[4], [4]} = p_{k-5}p_{k+2}, \dots, \minor_{v[2(i-1)], [2(i-1)]}= p_1p_{2k-2},
\]
so each of these minors contributes with $2$ irreducible factors. It follows that the number of frozen variables in $R(v,w)$ is
\[
f_{v,w} = k-1 - (i-1) +2(i-1) = k-1 + i-1 = k-1+\frac{k}{2} - 1,
\]
which coincides with $\dim(R(v,w))$. In particular, $R(v,w)$ is a torus.
\end{remark}

\begin{remark}
    For $u, v \in S_k$ with $u \leq v$, let $f_{u,v}$ denote the number of frozen variables in the cluster structure on $R(u,v)$. 
    By \eqref{eq: frozen inequality} we have the inequality
    \begin{equation}
    \label{eq:inequality-frozens-richardson}
    f_{u,v} + f_{v,w} \geq f_{u,w}.
    \end{equation}
    In forthcoming work of the second author \cite{soyeon}, $f_{u,v}$ is shown to be equal to the coefficient of the second highest degree of the so-called  R-polynomial $\mathrm{R}_{u,v}(q)$ associated to the pair $(u,v)$ which has interesting combinatorial interpretations, cf. \cite{Patimo}. In particular, \cite{soyeon} provides an alternative combinatorial proof of \eqref{eq:inequality-frozens-richardson}.  
\end{remark}

Note that all the results so far in this section have been independent of the validity of Conjecture \ref{conj:splicing-braid-varieties}. The next result explores the consequences of Conjecture \ref{conj:splicing-braid-varieties} in the Richardson setting.

\begin{proposition}
    Assume Conjecture \ref{conj:splicing-braid-varieties} holds. Then, for every $u \leq v \leq w \in S_k$ and every $i = 1, \dots, k$, the minors $\minor_{v[i], [i]}$ are cluster monomials in $R(u,w)$. Moreover, there exists a single cluster $\mathbf{x}$ so that all said minors are cluster monomials $\mathbf{x}$.
\end{proposition}
Note that, fixing $u$ and $w$, $\mathbf{x}$ may depend on the permutation $v$. 
\begin{proof}
If Conjecture \ref{conj:splicing-braid-varieties} (1) holds, then there exists a cluster $\mathbf{x}$ and cluster variables $x_{a_1}, \dots, x_{a_s} \in \mathbf{x}$ such that $\CU_{u,v,w}$ is the non-vanishing locus of $x_{a_1}\cdots x_{a_s}$. Moreover, if Conjecture \ref{conj:splicing-braid-varieties} (2) holds, then $\CU_{u,v,w}$ admits a cluster structure whose frozen variables are the frozen variables in $R(u,w)$, plus the variables $x_{a_1}, \dots x_{a_s}$. \\
Now, by Lemma \ref{lem: relpos w minors}(a), for every $i = 1, \dots, k$, the minor $\minor_{v[i], [i]}$ is nowhere vanishing on $\CU_{u, v, w}$. Thus, by \cite[Theorem 2.2]{GLS}, the restriction $\minor_{v[i], [i]}|_{\CU_{u,v,w}}$ is a monomial on the frozen variables of $\CU_{u,v,w}$ that, as we have seen, are frozen variables of $R(u,w)$ plus some cluster variables in $\mathbf{x}$. Now, $\C[R(u,w)] \subseteq \C[\CU_{u,v,w}]$ and the function $\minor_{v[i], [i]}$ is regular on $R(u,w)$. Since $\minor_{v[i],[i]}$ is a cluster monomial on $\CU_{u,v,w}$, it is also a cluster monomial on $R(u,w)$. 
\end{proof}

In the remainder of the paper, we will construct splicing maps satisfying all the properties of Conjecture \ref{conj:splicing-braid-varieties} in 
the special case of double Bott--Samelson varieties.

\section{Splicing double Bott-Samelson varieties}\label{sec:splicing-DBS}
\subsection{Setup}\label{sec:setup} Let $\br$ be a positive braid, and assume we have a decomposition $\br = \br^1\br^2$, where $\br^1 = \sigma_{i_1}\cdots \sigma_{i_{r_1}}$ and $\br^2 = \sigma_{i_{r_1 + 1}}\cdots \sigma_{i_r}$. For each $s = 1, \dots, k-1$, let $\last^1(s) \in \{1, \dots, r_1\}$ be such that $i_{\last^1(s)} = s$, and $i_j \neq s$ for $\last^1(s) < j \leq r_1$, i.e., $\last^1(s)$ is the rightmost appearance of $\sigma_s$ in $\br^1$. We will consider the cluster variable $x_{\last^1(s)} \in \C[X(\br)]$. If $\sigma_s$ does not appear in $\br^1$, we will simply set $x_{\last^1(s)} = 1$. 

\subsection{Splicing} In the setup of Section \ref{sec:setup}, we have the following result.

\begin{theorem}\label{thm:splicing-dbs}
Let $\CU_{r_1}(\br) \subseteq \BS(\br)$ be the locus where none of the cluster variables $x_{\last^1(s)}$ vanish for $s = 1, \dots, k-1$. Then 
\begin{enumerate}
    \item We have $(z_1, \dots, z_r) \in \CU_{r_1}(\br)$ if and only if  $(z_1, \dots, z_{r_1}) \in \BS\left(\br^1\right)$
    \item We have an isomorphism of algebraic varieties
\[
\Phi_{r_1}: \CU_{r_1}(\br) \buildrel \cong \over \longrightarrow \BS\left(\br^1\right) \times \BS\left(\br^2\right). 
\]
\end{enumerate}
\end{theorem}
\begin{proof}
    Let $(z_1, \dots, z_r) \in \CU_{r_1}(\br)$. 
    By Lemma \ref{lem:cauchy-binet-minors} for $s = 1, \dots, k-1$ we have
    \[
    \minor_{s}(B_{\br^1}(z_1, \dots, z_{r_1})) = x_{\last^1(s)}
    \]
    so that the condition $(z_1, \dots, z_r) \in \CU_{r_1}(\br)$ is equivalent to $\minor_{s}(B_{\br^1}(z_1, \dots, z_{r_1}))\neq 0$ for all $s$ which is equivalent to $(z_1, \dots, z_{r_1}) \in \BS\left(\br^1\right)$. Now, since $(z_1, \dots, z_{r_1}) \in \BS\left(\br^1\right)$ we have an LU-decomposition
    \[
    B_{\br^1}(z_1, \dots, z_{r_1}) = L_1U_1
    \]
    where $L_1$ has $1$'s on the diagonal. Since $(z_1, \dots, z_r) \in \BS(\br)$ we also have an LU-decomposition $B_{\beta}(z_1, \dots, z_r) = LU$. Thus, we obtain
    \begin{equation}\label{eq:LUB=LU}
    L_1U_1B_{\beta^2}(z_{r_1 + 1}, \dots, z_r) = LU
    \end{equation}
    Now, by Lemma \ref{lem: push right}  we can write
    \begin{equation}\label{eq:slideU}
    U_1B_{\beta^2}(z_{r_1+1}, \dots, z_r) = B_{\beta^2}(z'_{r_1+1}, \dots, z'_r)U_1'
    \end{equation}
    for some 
    change of variables $z'_{r_1+1}, \dots, z'_r$. It follows from \eqref{eq:LUB=LU} that $(z'_{r_1+1}, \dots, z'_{r}) \in \BS\!\left(\beta^2\right)$. The map $\Phi_{r_1}: (z_1, \dots, z_{r_1}, z_{r_1+1}, \dots, z_r) \mapsto ((z_1, \dots, z_{r_1}), (z'_{r_1+1}, \dots, z'_r))$ gives the desired isomorphism.

To construct the inverse map, suppose that we are given $((z_1, \dots, z_{r_1}), (z'_{r_1+1}, \dots, z'_r))$ such that $B_{\beta^1}(z_1,\ldots,z_{r_1})=L_1U_1$ and $B_{\beta^2}(z'_{r_1+1}, \dots, z'_r)=L_2U_2$. As above, the matrices $L_1,L_2$ are lower-triangular with 1 on diagonal and $U_1,U_2$ are upper triangular, all these matrices are uniquely determined by $((z_1, \dots, z_{r_1}), (z'_{r_1+1}, \dots, z'_r))$.

By \eqref{eq:slideU} there exist $(z_{r_1+1},\ldots,z_{r})$ and an upper-triangular matrix $U''_1$ such that 
$$
U_1^{-1}B_{\beta^2}(z'_{r_1+1}, \dots, z'_r)=
B_{\beta^2}(z_{r_1+1}, \dots, z_r)U_1''
$$
which implies
$$
B_{\beta^2}(z'_{r_1+1}, \dots, z'_r)(U''_1)^{-1}=
U_1B_{\beta^2}(z_{r_1+1}, \dots, z_r).
$$
This determines the inverse map. We can check that the result lands in $\BS(\beta)$:
$$
B_{\beta}(z_1,\ldots,z_r)=B_{\beta^1}(z_1,\ldots,z_{r_1})B_{\beta^2}(z_{r_1+1}, \dots, z_r)=L_1U_1B_{\beta^2}(z_{r_1+1}, \dots, z_r)=
$$
$$
L_1B_{\beta^2}(z'_{r_1+1}, \dots, z'_r)(U''_1)^{-1}=L_1L_2U_2(U''_1)^{-1}=LU,
$$
where $L=L_1L_2$ and $U=U_2(U''_1)^{-1}$.    
\end{proof}

\begin{remark}
    We constructed the map $\Phi_{r_1}$ using coordinates since we will need this in order to show cluster-theoretic properties of $\Phi_{r_1}$. It is useful, however, to have a more conceptual understanding of it using flags, cf. Section \ref{sec:comparison} below. Note that $\CU_{r_1}(\br)$ is the locus of elements $(\CF^0, \dots, \CF^r) \in \BS(\br)$ such that $\CF^{r_1} \pitchfork \CF(w_0)$. In particular, $\CF^{r_1} = \CF(L_1U_1)$, and $L_1$ is uniquely determined provided it has $1$'s on the diagonal. Then,
    \[
    \Phi_{r_1}\left(\CF^0, \dots, \CF^r\right) = \left(\left(\CF^0, \dots, \CF^{r_1}\right), L_1^{-1}\left(\CF^{r_1}, \CF^{r_1+1}, \dots, \CF^r\right)\right) \in \BS\left(\br^1\right) \times \BS\left(\br^2\right). 
    \]
\end{remark}

Let $\Sigma_{\br}$ be the left inductive seed on $\BS(\br)$ as described in Section \ref{sec:cluster-structure}, and let $\widehat{\Sigma}_{\br}$ be the seed obtained upon freezing the variables $x_{\last^1(s)}$, $s = 1, \dots, k-1$. The following lemma says that the set $\CU_{r_1}(\br)$ together with the cluster variables $x_{\last^1(s)}$ satisfy the properties (1)--(3) of Conjecture \ref{conj:splicing-braid-varieties}. 

\begin{lemma}\label{lem:U-cluster-structure}
    We have $A(\widehat{\Sigma}_{\br}) \cong \C[\CU_{r_1}(\br)]$. 
\end{lemma}
\begin{proof}
    From the description of the quiver $Q_{\br}$ in Section \ref{sec:cluster-structure} it is easy to see that $\widehat{Q}_{\br}^{\mathrm{uf}}$ is isomorphic to the disjoint union $Q_{\br^1}^{\mathrm{uf}} \sqcup Q_{\br^2}^{\mathrm{uf}}$, so the result follows as in the proof of Lemma \ref{lem:conditional-conjecture} above.
\end{proof}

\subsection{Comparison to braid variety splicing}\label{sec:comparison} It is natural to ask whether the map $\Phi_{r_1}$ constructed in Theorem \ref{thm:splicing-dbs} is compatible with the braid variety splicing from Theorem \ref{thm:splicing-braid-varieties}. Recall the isomorphisms
$$
\varphi_1:\BS(\beta)\to X(\beta\Delta),\quad 
\varphi_2:\BS(\beta)\to X(\Delta\beta)
$$
from Lemma \ref{lem:bs-to-braid}.
We have the following results.

\begin{lemma}
\label{lem: bs vs braid splicing compatible 1}
    Let $\br = \br^1\br^2$ be a positive braid. The following diagram commutes:
    \[
    \begin{tikzcd}
    \BS(\br) \supseteq \CU_{r_1} \arrow{rr}{\Phi_{r_{1}}} \arrow{d}{\bsbra} & & \BS(\br^1) \times \BS(\br^2) \arrow{d}{\bsbrb \times \bsbra} \\
X(\br\Delta) \supseteq \CU_{r_1, e} \arrow{rr}{\Phi_{r_{1}, e}} & & X(\Delta\br^1) \times X(\br^2\Delta)
    \end{tikzcd}
    \]
\end{lemma}
\begin{proof}
    Let $(\CF^0, \dots, \CF^r) \in \CU_{r_1}$, so that $\bsbra(\CF^0, \dots, \CF^r) = (\CF^0, \dots, \CF^r, \widetilde{\CF}^{r+1}, \dots, \widetilde{\CF}^{r+\ell(w_0)})$, where $\widetilde{\CF}^{r+1}, \dots, \CF^{r+\ell(w_0)}$ are uniquely determined. Note that both $\CU_{r_1}$ and $\CU_{r_1,e}$ are defined by $\CF^{r_1}\pitchfork \CF(w_0)$, i.e., the isomorphism $\bsbra$ identifies $\CU_{r_1}$ with $\CU_{r_1, e}$.

    Let us write $\CF^{r_1} = \CF(L_1U_1)$, where $L_1$ is uniquely determined by the condition that it has $1$'s on the diagonal. Then,
    \begin{equation}\label{eq:one-side-square}
    \begin{array}{rl}
    (\bsbrb \times \bsbra)\circ\Phi_{r_1}(\CF^0, \dots, \CF^r) = & \left(\bsbrb\left(\CF^0, \dots, \CF^{r_1}\right), \bsbra\left(L_1^{-1}\CF^{r_1}, \dots, L_1^{-1}\CF^{r}\right)\right) \\
    = & \biggl(\left(\CG^0, \dots, \CG^{\ell(w_0)-1}, w_0L_1^{-1}\CF^0, \dots, w_0L_1^{-1}\CF^{r_1}\right), \\
     & \left(L_1^{-1}\CF^{r_1}, \dots, L_1^{-1}\CF^{r}, \widetilde{\CG}^1, \dots, \widetilde{\CG}^{\ell(w_0)}\right)\biggr)
    \end{array}
    \end{equation}
    where the flags $\CG^{0} = \CF^{\std}, \dots, \CG^{\ell(w_0)-1}$ and $\widetilde{\CG}^{1}, \dots, \widetilde{\CG}^{\ell(w_0)} = \CF^{\ant}$ are uniquely determined. 
    Indeed, to compute $\varphi_2(\CF^0, \dots, \CF^{r_1})$ by Lemma \ref{lem:bs-to-braid} we first write $\CF^{r_1}=\CF(L_1U_1)=\CF((w_0U'w_0)U_1)$ where $U'=w_0L_1w_0$. Then $(U')^{-1}w_0\CF^j=w_0L_1^{-1}\CF^j$ for $j=0,\ldots,r_1$.

    Now we have to compare \eqref{eq:one-side-square} with $\Phi_{r_{1},e}(\CF^0, \dots, \CF^{r}, \widetilde{\CF}^{r+1}, \dots, \widetilde{\CF}^{r + \ell(w_0)})$. First, we factor $\beta\Delta=\beta^1(\beta^2\Delta)$.
    To get the $X(\br^2\Delta)$ component we need to find an element $g_2$ so that simultaneously $g_2\CF^{r_1} = \CF^{\std}$ and $g_2\widetilde{\CF}^{r+\ell(w_0)} = \CF^{\ant}$.
   By \eqref{eq: g2} we get  $g_2 = L_1^{-1}$ and thus
$$\Phi^2\left(\CF^0, \dots, \CF^{r}, \widetilde{\CF}^{r+1}, \dots, \widetilde{\CF}^{r + \ell(w_0)}\right)=
g_2\left(\CF^{r_1}, \CF^{r_1+1}, \dots, \CF^{r}, \widetilde{\CF}^{r+1}, \dots, \widetilde{\CF}^{r+\ell(w_0)}\right)=$$
$$
\left(L_1^{-1}\CF^{r_1}, \dots, L_1^{-1}\CF^{r}, \widetilde{\CG}^1, \dots, \widetilde{\CG}^{\ell(w_0)}\right)=\varphi_2\circ \Phi_{r_1}(\CF^0,\ldots,\CF^r).
$$
 by the uniqueness of the flags $\widetilde{\CG}^1, \dots, \widetilde{\CG}^{\ell(w_0)}$.
 

    Let us now examine the $X(\Delta\br^1)$ component of $\Phi_{r_{1},e}(\CF^0, \dots, \CF^{r}, \widetilde{\CF}^{r+1}, \dots, \widetilde{\CF}^{r+\ell(w_0)})$. For this, we obtain a unique sequence of flags going from $\widetilde{\CF}^{r+\ell(w_0)} = \CF^{\ant}$ to $\CF^{\std}$ using the letters in $\Delta$, say $\widehat{\CF}^{1}, \dots, \widehat{\CF}^{\ell(w_0)-1}$ and choose $g_1=w_0L_1^{-1}$ by \eqref{eq: g1}. Then
    $$
    \Phi^1\left(\CF^0, \dots, \CF^{r}, \widetilde{\CF}^{r+1}, \dots, \widetilde{\CF}^{r+\ell(w_0)}\right)=
    g_1\left(\CF^{r+\ell(w_0)}, \widehat{\CF}^{1}, \dots, \widehat{\CF}^{\ell(w_0)-1}, \CF^{0}, \dots, \CF^{r_1}\right)=
    $$
    $$
    \left(\CG^0, \dots, \CG^{\ell(w_0)-1}, w_0L_1^{-1}\CF^0, \dots, w_0L_1^{-1}\CF^{r_1}\right)=\varphi_1\circ \Phi_{r_1}(\CF^0,\ldots,\CF^r).
    $$
by the uniqueness of the flags $\CG^0,\ldots,\CG^{\ell(w_0)-1}$. This finishes the proof.
    
\end{proof}

\begin{lemma}
\label{lem: bs vs braid splicing compatible 2}
    Let $\br = \br^1\br^2$ be a positive braid. The following diagram commutes:
    \[
    \begin{tikzcd}
    \BS(\br) \supseteq \CU_{r_1} \arrow{rr}{\Phi_{r_{1}}} \arrow{d}{\bsbrb} & & \BS(\br^1) \times \BS(\br^2) \arrow{d}{\bsbrb \times \bsbra} \\
X(\Delta\br) \supseteq \CU_{r_1+\ell(w_0), w_0} \arrow{rr}{\Phi_{r_{1}+\ell(w_0), w_0}} & & X(\Delta\br^1) \times X(\br^2\Delta)
    \end{tikzcd}
    \]
    Note that at the bottom we use factorization $\Delta\beta=(\Delta\beta^1)\beta^2$ and $w=w_0$. 
\end{lemma}

\begin{proof}
The composition $(\varphi_2\times \varphi_1)\circ \Phi_{r_1}$ is unchanged, so we need to compute the composition $\Phi_{r_1+\ell(w_0),w_0}\circ \varphi_2$ and compare it with \eqref{eq:one-side-square}. We follow the notations $\CF^{r_1}=\CF(L_1U_1)$ from Lemma \ref{lem: bs vs braid splicing compatible 1}.

Given $(\CF^0,\ldots,\CF^r)\in \BS(\beta)$, we have $\CF^r=\CF(L'U')$ where $L'$ is unique provided that it has 1's on diagonal. Then by Lemma \ref{lem:bs-to-braid} we get
$$
\varphi_2(\CF^0,\ldots,\CF^r)=\left(\CF^{\std},\ldots,\widetilde{\CF}^{\ell(w_0)-1},w_0(L')^{-1}\CF^0,\ldots,w_0(L')^{-1}\CF^r\right).
$$
Next, we look at the $(\ell(w)+r_1)$-st flag $$
w_0(L')^{-1}\CF^{r_1}=w_0(L')^{-1}\CF(L_1U_1)=\CF(M),\quad 
M=w_0(L')^{-1}L_1U_1.
$$
Next, according to \eqref{eq:gen-LU-dec-phi2} we need to write:
$M=w_0LU$
so that $L=(L')^{-1}L_1$ and $U=U_1$. By \eqref{eq: g2} and \eqref{eq: g1} we compute
$$
g_2=(w_0L)^{-1}=L_1^{-1}L'w_0,\ g_1=w_0L_1^{-1}L'w_0
$$
and 
$$
g_2(w_0(L')^{-1}\CF^j)=L_1^{-1}\CF^j,\ g_1(w_0(L')^{-1}\CF^j)=w_0L_1^{-1}\CF^j.
$$
Finally, by \eqref{eq:Phi1} and \eqref{eq:Phi2} we get
$$
\Phi^1\left(\CF^{\std},\ldots,\widetilde{\CF}^{\ell(w_0)-1},w_0(L')^{-1}\CF^0,\ldots,w_0(L')^{-1}\CF^r\right)=
$$
$$
g_1\left(\CF^{\std},\ldots,\widetilde{\CF}^{\ell(w_0)-1},w_0(L')^{-1}\CF^0,\ldots,w_0(L')^{-1}\CF^{r_1}\right)=
$$
$$
\left(\CG^0,\ldots,\CG^{\ell(w_0)-1},w_0L_1^{-1}\CF^0,\ldots,w_0L_1^{-1}\CF^{r_1}\right),
$$
and
$$
\Phi^2\left(\CF^{\std},\ldots,\widetilde{\CF}^{\ell(w_0)-1},w_0(L')^{-1}\CF^0,\ldots,w_0(L')^{-1}\CF^r\right)=
$$
$$
g_2\left(w_0(L')^{-1}\CF^{r_1},\ldots,w_0(L')^{-1}\CF^r,\widehat{\CF^1},\ldots,\widehat{\CF}^{\ell(w_0)}\right)=
$$
$$
\left(L_1^{-1}\CF^r_1,\ldots,L_1^{-1}\CF^{r},\widetilde{\CG}^1,\ldots,\widetilde{\CG}^{\ell(w_0)}\right)
$$
by uniqueness of the flags $\CG,\widetilde{\CG}$. Therefore
$\Phi_{r_1+\ell(w_0),w_0}\circ \varphi_2(\CF^0,\ldots,\CF^r)$ is also given by \eqref{eq:one-side-square}.
\end{proof}

 In \cite{GKSS} we defined a subclass of open positroid varieties $S^{\circ,1}_{\lambda/\mu}\subset \Gr(k,n)$ called {\em skew shaped positroids} and labeled by skew Young diagrams $\lambda/\mu$. By \cite[Theorem 3.5.13]{GKSS} we have an isomorphism
 $$
S^{\circ,1}_{\lambda/\mu}\simeq X(\Delta\beta)\simeq \BS(\beta).
 $$
 where $\beta$ is a certain $k$-strand braid determined by $\lambda/\mu$.
 
 We also considered cutting the diagram $\lambda/\mu$ into two skew diagrams $\lambda^{a,L}/\mu^{a,L}$ and $\lambda^{a,R}/\mu^{a,R}$. These correspond to braids $\beta_L$ and $\beta_R$ and we showed that in fact $\beta=\beta_L\beta_R$. Then we defined a splicing map
 $$
 S^{\circ,1}_{\lambda^L/\mu^L}\times S^{\circ,1}_{\lambda^R/\mu^R}\to \CU_a\subset S^{\circ,1}_{\lambda/\mu}
 $$
or equivalently
\begin{equation}
\label{eq: GKSS splicing}
X(\Delta\beta_L)\times X(\Delta\beta_R)\to \CU_a\subset X(\Delta\beta).
\end{equation}
Here $\CU_a$ is a certain open subset in $S^{\circ,1}_{\lambda/\mu}$. By \cite[Theorem 5.6.1]{GKSS} the map \eqref{eq: GKSS splicing} is a quasi-cluster equivalence.
We invite the reader to explore more combinatorial and linear-algebraic properties of the map  \eqref{eq: GKSS splicing} in \cite{GKSS}.

We claim that the map \eqref{eq: GKSS splicing} agrees with the map $\Phi_r$ from Theorem \ref{thm:splicing-dbs} in this special case. We  focus on \cite[Definition 5.2.1]{GKSS}. 
The braid variety $X(\Delta\beta)$ parametrizes chains of (framed) flags
$$
\Omega=\left[\CF^{W}\stackrel{\longest}{\dashrightarrow}\CF^{\Wop}\stackrel{\beta_{L}}{\dashrightarrow}\CF^a\stackrel{\beta_{R}}{\dashrightarrow}\CF^0\right],
$$
see \cite{GKSS} for unexplained notations.
Such a chain belongs to $\CU_a$ if and only if
 $\CF^a\pitchfork \CF^W$. In this case,  the  splicing map  sends $\Omega$ to
$$
\Omega^L=\left[\CF^{W}\stackrel{\longest}{\dashrightarrow}\CF^{\Wop}\stackrel{\beta_{L}}{\dashrightarrow}\CF^a\right],\quad  \Omega^R=\left[\CF^{W}\stackrel{\longest}{\dashrightarrow}\CF^{a}\stackrel{\beta_{R}}{\dashrightarrow}\CF^0\right].
$$
We can identify these with points in $X(\Delta\beta_L)$ and $X(\Delta\beta_R)$ respectively, up to multiplication by some $g\in \GL(k)$. Similarly to Lemmas \ref{lem: bs vs braid splicing compatible 1} and \ref{lem: bs vs braid splicing compatible 2} this agrees with $\Phi_{r_1}$ after applying the isomorphism $X(\Delta\beta_R)\simeq X(\beta_R\Delta)$.

\subsection{Cluster-theoretic properties of the map $\Phi_{r_1}$.}\label{sec:cluster-properties-splicing} By Lemma \ref{lem:U-cluster-structure} the algebra $\C[\CU_{r_1}(\br)]$ admits a natural cluster structure. We then examine the cluster-theoretic properties of the map $\Phi_{r_1}$. We consider the pullback
\[
\Phi_{r_1}^{\ast}: \C\!\left[\BS(\br^1) \times \BS(\br^2)\right] \to \C[\CU_{r_1}(\br)]
\]
that is an isomorphism of algebras.  We will denote the coordinates on $\BS(\br^1)$ by $z_1, \dots, z_{r_1}$ and the coordinates on $\BS(\br^2)$ by $z_{r_1+1}, \dots, z_r$. Then the isomorphism $\Phi_{r_1}^{\ast}$ is determined by:
\begin{equation}\label{eq:pullback}
\Phi^{\ast}_{r_1}(z_j) = z_j, j = 1, \dots, r_1, \qquad \Phi^{\ast}_{r_1\textbf{}}(z_j) = z'_j, j = r_1+1, \dots, r. 
\end{equation}
where $z'_{r_1+1}, \dots, z'_r$ are determined by \eqref{eq:slideU}.

On $\C[\CU_{r_1}(\br)]$ we have the cluster algebra structure obtained in Lemma \ref{lem:U-cluster-structure}. Since this comes comes from the open embedding $\CU_{r_1}(\br) \to \BS(\br)$, we call this the \emph{open cluster structure} on $\CU_{r_1}(\br)$, with quiver $Q^{\circ}_{\br}$. The quiver $Q^{\circ}_{\br}$ is obtained from the quiver $Q_{\br}$ by freezing the vertices corresponding to the rightmost appearance of a crossing in $\br^1$. 

On the other hand, on $\C\left[\BS(\br^1) \times \BS(\br^2)\right]$ we have a natural cluster structure, with cluster variables $\xun_1, \dots, \xun_{r_1}$ (coming from $\BS(\br^1)$) and $\x2_{r_1+1}, \dots, \x2_{r}$ (coming from $\BS(\br^2)$). Thus, we obtain a cluster structure on $\C[\CU_{r_1}(\br)]$, with cluster variables being $\phipul\left(\xun_{j}\right)$, $\phipul\left(\x2_{\ell}\right)$, $j = 1, \dots, r_1, \ell = r_{1}+1, \dots, r$. We call this the \emph{product cluster structure} on $\C[\CU_{r_1}(\br)]$. We denote by $\QP$ the quiver corresponding to this product cluster structure: it is simply the disjoint union of the quivers $Q_{\br^1}$ and $Q_{\br^2}$.  

Our goal is to relate the open and the product cluster structures. We start with some useful notations and definitions.

\begin{definition}
Let $B_{\br^1}(z_1,\ldots,z_{r_1})=L_1U_1$ as above. We denote the diagonal elements of $U_1$ by 
$\left(\u1_1, \dots, \u1_k\right)$. 
\end{definition}

Recall that by \eqref{eq:frozen-in-terms-of-upper} all $\u1_j$ are cluster monomials in the frozen variables $x_{\last^1(s)}$, $s = 1, \dots, k-1$. We can combine these to build more cluster monomials.

\begin{definition}
For $l=r_1+1,\ldots,r$ we define:
\begin{equation}\label{eq:monomial}
m_\ell = \prod_{t = 1}^{i_\ell}\left(\u1_{s_{i_{r_1+1}}\cdots s_{i_\ell}(t)}\right)^{-1}
\end{equation}
\end{definition}


Let us interpret this pictorially. Draw the braid $\br^2$ and look at the $\ell-r_1$-th crossing, 
counting from left to right. Right after this crossing, look at the bottom $i_\ell$-strands, and follow them to the left. The end labels of the strands are precisely the subindices of the $\u1$-factors appearing in \eqref{eq:monomial}. See Figure \ref{fig:computing-m} for an example. 

\begin{figure}[ht!]
    \centering
    \begin{tikzpicture}
        \node at (-0.2,5) {$\mathbf{5}$};
        \node at (-0.2,4) {$\mathbf{4}$};
        \node at (-0.2,3) {$\mathbf{3}$};
        \node at (-0.2,2) {$\mathbf{2}$};
        \node at (-0.2,1) {$\mathbf{1}$};

        \draw (0,5) to[out=0, in=180] (1,4) to[out=0, in=180] (2,3) to (4,3) to[out=0, in=180] (5,2) to (7,2) to[out=0, in=180] (8,3) to (10,3);
        \draw (0,4) to[out=0, in=180] (1,5) to (6,5) to[out=0, in=180] (7,4) to[out=0, in=180] (8,5) to (10,5);
        \draw (0,3) to (1,3) to[out=0, in=180] (2,4) to (5,4) to[out=0, in=180] (6,3) to (7,3) to[out=0, in=180] (8,2) to (10,2);
        \draw (0,2) to (2,2) to[out=0, in=180] (3,1) to[out=0, in=180] (4,2) to[out=0, in=180] (5,3) to[out=0, in=180] (6,4) to[out=0, in=180] (7,5) to[out=0, in=180] (8,4) to (10,4);
        \draw (0,1) to (2,1) to[out=0, in=180] (3,2) to[out=0, in=180] (4,1) to (10,1);

        \draw[color=red, thick] (0,5) to[out=0, in=180] (1,4) to[out=0, in=180] (2,3) to (4,3) to[out=0, in=180] (5,2) to (6,2);

        \draw[color=red, thick] (0,3) to (1,3) to[out=0, in=180] (2,4) to (5,4) to[out=0, in=180] (6,3);

         \draw[color=red, thick] (0,1) to (2,1) to[out=0, in=180] (3,2) to[out=0, in=180] (4,1) to (6,1);

         \draw[dashed, color=blue] (0, 0.5) to (0, 5.5);

         \draw (-3, 1) to[out=0, in=180] (-2, 2) to[out=0, in=180] (-1, 3) to[out=0, in=180] (0,4);
         \draw (-3, 2) to[out=0, in=180] (-2, 1) to[out=0, in=180] (0,1);
         \draw (-3, 3) to (-2, 3) to[out=0, in=180] (-1, 2) to (0,2);
         \draw (-3, 4) to[out=0, in=180] (-2, 5) to[out=0, in=180] (0, 5);
         \draw (-3, 5) to[out=0, in=180] (-2, 4) to (-1, 4) to[out=0, in=180] (0,3);

         \node at (0.7, 4.5) {\tiny{1}};
         \node at (1.7, 3.5) {\tiny{2}};
         \node at (2.7, 1.5) {\tiny{3}};
         \node at (3.7, 1.5) {\tiny{4}};
         \node at (4.7, 2.5) {\tiny{5}};
         \node at (5.7, 3.5) {\tiny{6}};

         \draw (-3, 0.8) to[out=270, in=120] (-1.5, 0.2);
         \draw (-1.5, 0.2) to[out=60, in=270] (0, 0.8);
         \node at (-1.5, 0) {$r_1$ crossings};
    \end{tikzpicture}
    \caption{The monomial $m_{r_1+6}$  equals $\left(\u1_{3}\u1_5\u1_1\right)^{-1}$.}
    \label{fig:computing-m}
\end{figure}
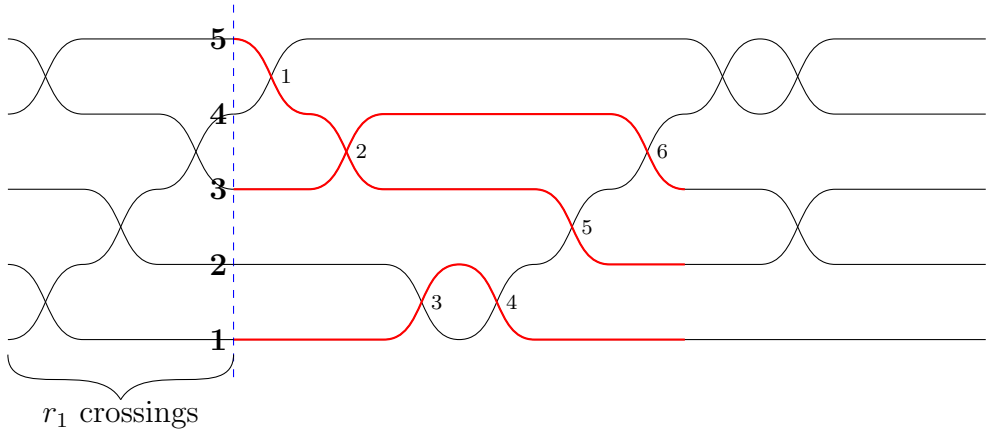

\begin{lemma}\label{lem:coincidence-up-to-monomial-frozens}
Under the isomorphism $\phipul: \C\left[\BS(\br^1) \times \BS(\br^2)\right] \to \C[\CU_{r_1}(\br)]$ we have:
\begin{enumerate}
    \item  $\phipul\left(\xun_j\right) = x_j$ for $j = 1, \dots, r_1$.
    \item $\phipul\left(\x2_\ell\right) = m_\ell x_\ell$ for $\ell = r_1 + 1, \dots, r$, where $m_\ell$ is given by \eqref{eq:monomial}.
\end{enumerate}
\end{lemma}
\begin{proof}
By \eqref{eq:pullback}, we have
\[
\phipul\left(\xun_j\right) = \phipul\left(\minor_{i_j}\!\left(B_{\br^1_j}(z_1, \dots, z_j)\right)\right) = \minor_{i_j}\!\left(B_{\br_j}(z_1, \dots, z_j)\right) = x_j
\]
where $j = 1, \dots, r_1$. On the other hand, if $\ell = r_1+1, \dots, r$ we have
\[
\phipul\left(\x2_\ell\right) = \phipul\left(\minor_{i_\ell}\!
\left(B_{\br^2_{\ell-r_1}}(z_{r_1+1}, \dots, z_\ell)\right)\right) = \minor_{i_\ell}\!\left(B_{\br^2_{\ell-r_1}}(z'_{r_1+1}, \dots, z'_\ell)\right).
\]
and our job is to compare this to $x_\ell = \minor_{i_\ell}\!\left(B_{\br_\ell}(z_1, \dots, z_\ell)\right)$. Note however that
\[
\begin{array}{rl}
x_\ell = \minor_{i_\ell}\!\left(B_{\br_\ell}(z_1, \dots, z_\ell)\right) = & \minor_{i_\ell}\!\left(B_{\br^1}(z_1, \dots, z_{r_1})B_{\br^2_{\ell-r_1}}\!(z_{r_1+1}, \dots, z_{\ell})\right) \\
= & \minor_{i_\ell}\!\left(L_1U_1B_{\br^2_{\ell-r_1}}(z_{r_1+1}, \dots, z_\ell)\right) \\
= & \minor_{i_\ell}\!\left(U_1B_{\br^2_{\ell-r_1}}(z_{r_1+1}, \dots, z_\ell)\right) \\
= & \minor_{i_\ell}\!\left(B_{\br^2_{\ell-r_1}}(z_{r_1+1}', \dots, z'_\ell)U_1'\right) \\
= & \minor_{i_\ell}\!\left(B_{\br^2_{\ell-r_1}}(z'_{r_1+1}, \dots, z'_{\ell})\right)\minor_{i_\ell}(U_1') \\
= & \phipul\!\left(\x2_\ell\right)\minor_{i_\ell}(U_1').
\end{array}
\]
Since $U'_1$ is upper triangular, its $i_\ell$-th principal minor is just the product of the upper left $i_\ell$-entries of the diagonal of $U'_1$. By Lemma \ref{lem: push right}, these coincide with the corresponding entries of $U_1$ up to permutation by $\pi(\beta^2_{\ell-r_1})=s_{i_{r_1+1}}\cdots s_{i_{\ell}}$. So
$$
\minor_{i_\ell}(U_1')=m_\ell^{-1},\quad x_{\ell}=\phipul\left(\x2_\ell\right)m_{\ell}^{-1}
$$
and the result follows. 
\end{proof}


We are now ready to prove the following result. 

\begin{theorem}\label{thm:quasi-cluster-dbs}
    The product and the open cluster structures on $\CU_{r_1}(\br)$ are quasi-cluster equivalent. 
\end{theorem}
\begin{proof}
    The fact that the cluster variables in both structures differ by monomials in frozens follows directly from Lemma \ref{lem:coincidence-up-to-monomial-frozens}. Now we need to verify that the exchange ratios in both cluster structures agree. The mutable cluster variables in both cluster structures are indexed by $\mut := \{1, \dots, r\}\setminus\{\last^1(s), \last(s) \mid s = 1, \dots, k-1\}$. 

Let us take $j \in \mut$ such that $j \leq r_1$. We claim that if $x_s$ is adjacent to $x_j$ then $s \leq r_1$. We have two cases. If $s$ and $j$ are the same color (i.e., $i_s = i_j$) then either $s \leq j \leq r_1$, or $j < s \leq \last^1(i_j) \leq r_1$, where the last equation follows since there cannot be $j < t < s$ with $i_j = i_t$ and, $x_j$ being mutable, $j$ cannot be the last appearance of $i_j$ in $\br^1$. Now assume $s$ and $j$ are not of the same color. If there is an arrow $j \to s$ then $s \leq j \leq r_1$, so we assume there is an arrow $s \to j$. So $s$ must be strictly between $j$ and $\last^1(j) \leq r_1$. Thus, $s \leq r_1$ and we have proven our claim. By Lemma \ref{lem:coincidence-up-to-monomial-frozens}(1), the exchange ratio associated to the cluster variable $x_j$ in the open structure coincides with the exchange ratio associated to $\phipul\left(\xun_j\right)$ in the product structure. 

Now we take $\ell \in \mut$ with $r_1 < \ell$. Our goal is to show that the exchange ratio associated to $x_\ell$ in the open structure coincides with the exchange ratio of $\phipul\left(\x2_\ell\right)$ in the product structure. Since $x_\ell$ is not frozen, there exists $\ell < \ell' \leq r$ such that $\colour(\ell) = \colour(\ell')$. In particular, we have an arrow $\ell \to \ell'$ both in $\widehat{Q}_{\br}$ and in $Q_{\br^1}\sqcup Q_{\br^2}$. We have several cases. 

\emph{Case 1. The quiver $Q^{\circ}_{\br}$ locally around $\ell$ looks as follows}. 


\begin{center}
    \begin{tikzpicture}
        \draw (0,0) to[out=0,in=180] (2,2);
        \draw (0,2) to[out=0,in=180] (2,0);
        \node at (2, 1) {$\ell''$};
        \draw[dotted] (2,2) to (4,2);
        \draw[dotted] (2,0) to (4,0);
        \draw (4,2) to[out=0,in=180] (6,0);
        \draw (4,0) to[out=0,in=180] (6,2);
        \node at (6,1) {$\ell$};
        \draw[dotted] (6,0) to (8,0);
        \draw[dotted] (6,2) to (8,2);
        \draw (8,2) to[out=0,in=180] (10,0);
        \draw (8,0) to[out=0,in=180] (10,2);
        \node at (10,1) {$\ell'$};

        \draw[thick, ->] (2.5, 1) to (5.5,1);
        \draw[thick, ->] (6.5, 1) to (9.5,1);
    \end{tikzpicture}
\end{center}

\noindent Since $\ell''$ is not frozen in $Q^{\circ}_{\br}$, we have $r_1 < \ell''$, so that both arrows also belong to $\QP$. So we have to show that  $x_{\ell'}/x_{\ell''} = \phipul\left(\x2_{\ell'}\right)/\phipul\left(\x2_{\ell''}\right)$. By Lemma \ref{lem:coincidence-up-to-monomial-frozens}, this is equivalent to showing that $m_{\ell'} = m_{\ell''}$. But this follows immediately from  \eqref{eq:monomial}\\

\emph{Case 2. The quiver $\QO$ locally around $\ell$ looks as follows.}

\begin{center}
    \begin{tikzpicture}
        \draw (4,2) to[out=0,in=180] (6,0);
        \draw (4,0) to[out=0,in=180] (6,2);
        \node at (6,1) {$\ell$};
        \draw[dotted] (6,0) to (8,0);
        \draw[dotted] (6,2) to (8,2);
        \draw (8,2) to[out=0,in=180] (10,0);
        \draw (8,0) to[out=0,in=180] (10,2);
        \node at (10,1) {$\ell'$};

        \draw[thick, ->] (6.5, 1) to (9.5,1);
    \end{tikzpicture}
\end{center}

So that the quiver $Q^{\circ}_{\br}$ locally around $\ell$ simply looks as $\ell'' \to \ell \to \ell'$, and $\ell'' = \last^1(\colour(\ell))$, so that $\ell''$ is frozen. The exchange ratio in the product cluster structure is then simply $\phipul\left(\x2_{\ell'}\right) = m_{\ell'}x_{\ell'}$. But $m_{\ell'} = \left(u_1\cdots u_{\last^1(\colour(\ell))}\right)^{-1} = x_{\last^1(\colour(\ell))}^{-1} = x_{\ell''}^{-1}$ and the result follows.

If the quiver $Q^{\circ}_{\br}$ looks $\ell \to \ell'$ locally around $\ell$, then $$x_{\last^1(\colour(\ell))} = 1 = \left(u_1\cdots u_{\last^1(\colour(\ell))}\right)^{-1} = m_{\ell'},$$ so this case is similar.  \\

\emph{Case 3. The quiver $Q^{\circ}_{\br}$ locally around $\ell$ looks as follows.}
\begin{center}
    \begin{tikzpicture}
        \draw (0,0) to[out=0,in=180] (1,1);
        \draw (0,1) to[out=0,in=180] (1,0);
        \draw[dotted] (1,0) to (2,0);
        \draw[dotted] (1,1) to (2,1);
        \draw[dotted] (0,2) to (2,2);
        \draw (2,2) to[out=0,in=180] (3,1);
        \draw (2,1) to[out=0,in=180] (3,2);
        \draw[dotted] (2,0) to (3,0);
        \draw (3,0) to (4,0);
        \draw (3,1) to (4,1);
        \draw[dotted] (3,2) to (4,2);
        \draw (4,0) to[out=0,in=180] (5,1);
        \draw (4,1) to[out=0,in=180] (5,0);
        \draw[dotted] (4,2) to (5,2);
        \draw[dotted] (5,0) to (6,0);
        \draw[dotted] (5,1) to (6,1);
        \draw[dotted] (5,2) to (6,2);
        \draw (6,2) to[out=0,in=180] (7,1);
        \draw (6,1) to[out=0,in=180] (7,2);
        \draw[dotted] (6,0) to (7,0);
        \draw (7,0) to (8,0);
        \draw (7,1) to (8,1);
        \draw[dotted] (7,2) to (8,2);
        \draw (8,0) to[out=0,in=180] (9,1);
        \draw (8,1) to[out=0,in=180] (9,0);
        \draw[dotted] (8,2) to (9,2);

        \node at (1, 0.5) {$\ell''$};
        \node at (5, 0.5) {$\ell$};
        \node at (9, 0.5) {$\ell'$};
        \node at (3, 1.5) {$j$};
        \node at (7, 1.5) {$j'$};

        \draw[thick, ->] (1.5, 0.5) to (4.7, 0.5);
         \draw[thick, ->] (5.5, 0.5) to (8.7, 0.5);
         \draw[thick, ->] (6.8, 1.5) to (5.2, 0.7);
         \draw[thick, ->] (4.8, 0.7) to (3.2, 1.5);
    \end{tikzpicture}
\end{center}

So we must show that $m_{\ell'}m_{j} = m_{\ell''}m_{j'}$. This follows easily from the pictorial interpretation of \eqref{eq:monomial}. Let us color the strands contributing to a monomial $m_{i}$ with the same color as $i$ in the picture below: 

\begin{center}
    \begin{tikzpicture}
        \draw (0,0) to[out=0,in=180] (1,1);
        \draw[dashed, thick, color=red] (0,1) to[out=0,in=180] (0.9,-0.1);
        \draw[dotted] (1,1) to (2,1);
        \draw[dashed, thick, color=blue] (0,2) to (2,2) to[out=0,in=180] (2.9,0.9);
        \draw (2,1) to[out=0,in=180] (3,2);
        \draw[dashed, thick, color=blue] (0,1.1) to (0.1, 1.1) to[out=0,in=180] (1,0) to (3,0);
        \draw (3,0) to (4,0);
        \draw[dashed, thick, color=orange] (3,1) to (4,1);
        \draw[dashed, thick, color=orange] (2.1, 2.1) to[out=0,in=180] (3,1) to (4,1) to[out=0,in=180] (5,0) to (7,0);
        \draw[dotted] (3,2) to (4,2);
        \draw (4,0) to[out=0,in=180] (5,1);
        \draw[dotted] (4,2) to (5,2);
        \draw[dotted] (5,1) to (6,1);
        \draw[dotted] (5,2) to (6,2);
        \draw[dashed, thick, color=orange] (6,2) to[out=0,in=180] (6.9,0.9);
        \draw (6,1) to[out=0,in=180] (7,2);
        \draw (7,0) to (8,0);
        \draw[dashed, thick, color=cyan] (6,2.1) to (6.1, 2.1) to[out=0, in=180] (7, 1) to (8,1) to[out=0,in=180] (9,0);
        \draw[dotted] (7,2) to (8,2);
        \draw (8,0) to[out=0,in=180] (9,1);
        \draw[dotted] (8,2) to (9,2);

        \node at (1, 0.5) {{\color{red}$\ell''$}};
        \node at (5, 0.5) {$\ell$};
        \node at (9, 0.5) {{\color{cyan}$\ell'$}};
        \node at (3, 1.5) {{\color{blue}$j$}};
        \node at (7, 1.5) {{\color{orange}$j'$}};

        \draw[thick, ->] (1.5, 0.5) to (4.7, 0.5);
         \draw[thick, ->] (5.5, 0.5) to (8.7, 0.5);
         \draw[thick, ->] (6.8, 1.5) to (5.2, 0.7);
         \draw[thick, ->] (4.8, 0.7) to (3.2, 1.5);
    \end{tikzpicture}
\end{center}


Note that we have used that there is no crossing of color $\colour(j')$ 
between $j'$ and $\ell'$, no crossing of color $\colour(\ell) - 1$ and no crossing of color $\colour(j)$ between $j$ and $\ell$. Note that every $u$ factor appearing in $m_{\ell'}m_{j}$ cancels with one appearing in $m_{\ell''}m_{j'}$ and vice versa, so the result follows. 

The case when $\ell'' < r_1$ (so that the quiver $Q_{\br^2}$ does not have the vertex $\ell''$ above, and $\ell'$ becomes frozen in $Q^{\circ}_{\br}$) is similar, after noticing that in this case $\last^1(\colour(\ell)) = \ell''$. 

We note that when the quiver locally looks like a horizontal reflection, we obtain a similar result.\\ 

\emph{Case 4. The quiver $Q^{\circ}_{\br}$ locally around $\ell$ looks as follows.}
\begin{center}
    \begin{tikzpicture}
        \draw (0,0) to[out=0,in=180] (1,1);
        \draw (0,1) to[out=0,in=180] (1,0);
        \draw[dotted] (1,0) to (2,0);
        \draw[dotted] (1,1) to (2,1);
        \draw[dotted] (0,2) to (2,2);
        \draw (2,2) to[out=0,in=180] (3,1);
        \draw (2,1) to[out=0,in=180] (3,2);
        \draw (3,0) to (4,0);
        \draw (3,1) to (4,1);
        \draw[dotted] (3,2) to (4,2);
        \draw (4,0) to[out=0,in=180] (5,1);
        \draw (4,1) to[out=0,in=180] (5,0);
        \draw[dotted] (4,2) to (5,2);
        \draw[dotted] (5,0) to (6,0);
        \draw[dotted] (5,1) to (6,1);
        \draw[dotted] (5,2) to (6,2);
        \draw (6,2) to[out=0,in=180] (7,1);
        \draw (6,1) to[out=0,in=180] (7,2);
        \draw (7,0) to (8,0);
        \draw (7,1) to (8,1);
        \draw[dotted] (7,2) to (8,2);
        \draw (8,0) to[out=0,in=180] (9,1);
        \draw (8,1) to[out=0,in=180] (9,0);
        \draw[dotted] (8,2) to (9,2);

        \draw[dotted] (0, -1) to (2, -1);
        \draw (2,-1) to[out=0, in=180] (3,0);
        \draw (2,0) to[out=0,in=180] (3,-1);
        \draw[dotted] (3,-1) to (6,-1);
        \draw (6,-1) to[out=0,in=180] (7, 0);
        \draw (6,0) to[out=0,in=180] (7, -1);
        \draw[dotted] (7,-1) to (9,-1);

        \node at (1, 0.5) {$\ell''$};
        \node at (5, 0.5) {$\ell$};
        \node at (9, 0.5) {$\ell'$};
        \node at (3, 1.5) {$j$};
        \node at (7, 1.5) {$j'$};
        \node at (3, -0.5) {$t$};
        \node at (7, -0.5) {$t'$};

        \draw[thick, ->] (1.5, 0.5) to (4.7, 0.5);
         \draw[thick, ->] (5.5, 0.5) to (8.7, 0.5);
         \draw[thick, ->] (6.8, 1.5) to (5.2, 0.7);
         \draw[thick, ->] (4.8, 0.7) to (3.2, 1.5);
         \draw[thick, ->] (6.8, -0.5) to (5.2, 0.2);
         \draw[thick, ->] (4.9, 0.2) to (3.2, -0.5);

         \draw[thick, dashed] (7, -1.5) to (7, 2.5);
    \end{tikzpicture}
\end{center}

In this case, we need to prove that
\[
m_{\ell''}m_{j'}m_{t'} = m_{\ell'}m_{j}m_{t},
\]
or, equivalently, that
\[
\frac{m_{j'}m_{t'}}{m_{\ell'}} = \frac{m_{j}m_{t}}{m_{\ell''}}
\]

Let $i = \colour(\ell)$ and consider the truncation of the braid $\br^{2}$ on the dashed vertical line as depicted above. Let $\tau$ be the resulting permutation. Then
\begin{equation}\label{eq:exchange-ratio-last-case}
\frac{m_{j'}m_{t'}}{m_{\ell'}} = \prod_{s \leq i}\left(\u1_{\tau^{-1}(s)}\right)^{-1}
\end{equation}

Now consider the truncation of the braid $\br^{2}$ on the red dashed vertical line below:

\begin{center}
    \begin{tikzpicture}
        \draw (0,0) to[out=0,in=180] (1,1);
        \draw (0,1) to[out=0,in=180] (1,0);
        \draw[dotted] (1,0) to (2,0);
        \draw[dotted] (1,1) to (2,1);
        \draw[dotted] (0,2) to (2,2);
        \draw (2,2) to[out=0,in=180] (3,1);
        \draw (2,1) to[out=0,in=180] (3,2);
        \draw (3,0) to (4,0);
        \draw (3,1) to (4,1);
        \draw[dotted] (3,2) to (4,2);
        \draw (4,0) to[out=0,in=180] (5,1);
        \draw (4,1) to[out=0,in=180] (5,0);
        \draw[dotted] (4,2) to (5,2);
        \draw[dotted] (5,0) to (6,0);
        \draw[dotted] (5,1) to (6,1);
        \draw[dotted] (5,2) to (6,2);
        \draw (6,2) to[out=0,in=180] (7,1);
        \draw (6,1) to[out=0,in=180] (7,2);
        \draw (7,0) to (8,0);
        \draw (7,1) to (8,1);
        \draw[dotted] (7,2) to (8,2);
        \draw (8,0) to[out=0,in=180] (9,1);
        \draw (8,1) to[out=0,in=180] (9,0);
        \draw[dotted] (8,2) to (9,2);

        \draw[dotted] (0, -1) to (2, -1);
        \draw (2,-1) to[out=0, in=180] (3,0);
        \draw (2,0) to[out=0,in=180] (3,-1);
        \draw[dotted] (3,-1) to (6,-1);
        \draw (6,-1) to[out=0,in=180] (7, 0);
        \draw (6,0) to[out=0,in=180] (7, -1);
        \draw[dotted] (7,-1) to (9,-1);

        \node at (1, 0.5) {$\ell''$};
        \node at (5, 0.5) {$\ell$};
        \node at (9, 0.5) {$\ell'$};
        \node at (3, 1.5) {$j$};
        \node at (7, 1.5) {$j'$};
        \node at (3, -0.5) {$t$};
        \node at (7, -0.5) {$t'$};

        \draw[thick, ->] (1.5, 0.5) to (4.7, 0.5);
         \draw[thick, ->] (5.5, 0.5) to (8.7, 0.5);
         \draw[thick, ->] (6.8, 1.5) to (5.2, 0.7);
         \draw[thick, ->] (4.8, 0.7) to (3.2, 1.5);
         \draw[thick, ->] (6.8, -0.5) to (5.2, 0.2);
         \draw[thick, ->] (4.9, 0.2) to (3.2, -0.5);

         \draw[color=red, thick, dashed] (4, -1.5) to (4, 2.5);
    \end{tikzpicture}
\end{center}
and let $\tau_1$ be the resulting permutation. Since there are no crossings of color $i = \colour(\ell)$ between $\ell$ and $\ell'$, we have that \eqref{eq:exchange-ratio-last-case} can  be rewritten as

\[
\frac{m_{j'}m_{t'}}{m_{\ell'}} = \left(\u1_{\tau_1^{-1}(i+1)}\right)^{-1}\prod_{s < i}\left(\u1_{\tau_1^{-1}(s)}\right)^{-1}.
\]

Which can easily seen to be $\frac{m_jm_{t}}{m_{\ell''}}$ and the result follows. Variations of this case (where, for example, the $t$-crossing belongs to $\br^{1}$) are proved similarly. 
\end{proof}


We now show that Example \ref{ex:intro} given in the introduction is a quasi-cluster isomorphism. 

\begin{example}
\label{ex:intro splice}
Recall Example \ref{ex:intro}. We verify the preservation of exchange ratios at each one of the vertices $10$--$13$. Note that the preservation of exchange ratios at the mutable vertices $1$--$8$ is immediate by construction. To lighten the notation, we denote $x_k' = \phipul\left(x_k^{(2)}\right)$.
\begin{itemize}
    \item Vertex $10$. In the open cluster structure, the exchange ratio is $\displaystyle{\frac{x_6x_{13}}{x_{7}x_{14}}}$. In the product cluster structure, this exchange ratio is $\displaystyle{\frac{x'_{13}}{x'_{14}}}$. According to Lemma \ref{lem:coincidence-up-to-monomial-frozens}, we have $x'_{14} = \left(\u1_{2}\u1_{3}\u1_{4}\right)^{-1}x_{14}$ and $x'_{13} = \left(\u1_{2}\u1_{4}\right)^{-1}x_{13}$. Now we note that $\displaystyle{\frac{x_6}{x_7}} = \frac{\u1_3\u1_2\u1_1}{\u1_2\u1_1} = \u1_3$, which shows the desired equality. 
    \item Vertex $11$. In the open cluster structure, the exchange ratio is $\displaystyle{\frac{x_7x_{12}}{x_9x_{13}}}$ while in the product cluster structure the exchange ratio is simply $\displaystyle{\frac{x'_{12}}{x'_{13}}}$. Now, according to Lemma \ref{lem:coincidence-up-to-monomial-frozens},
\[
x'_{12} = \left(\u1_{4}\right)^{-1}x_{12}, \qquad x'_{13} = \left(\u1_{4}\u1_{2}\right)^{-1}x_{13}, \qquad \text{so that} \qquad \frac{x'_{12}}{x'_{13}} = \u1_2\frac{x_{12}}{x_{13}}
\]
and it remains to notice that $\displaystyle{\frac{x_{7}}{x_{9}} = \frac{\u1_1\u1_2}{\u1_1} = \u1_2}$. 

\item Vertex $12$. In the open cluster structure, the exchange ratio is $\displaystyle{\frac{x_9x_{15}}{x_{11}x_{16}}}$ and in the product cluster structure it is $\displaystyle{\frac{x'_{15}}{x'_{11}x'_{16}}}$. Now,
\[
\frac{x'_{15}}{x'_{11}x'_{16}} = \frac{\left(\u1_3\u1_4\right)^{-1}x_{15}}{\left(\u1_4\u1_1\right)^{-1}x_{11}\left(\u1_3\right)^{-1}x_{16}} = \frac{\u1_1 x_{15}}{x_{11}x_{16}} = \frac{x_9x_{15}}{x_{11}x_{16}}. 
\]
\item Vertex $13$. We must verify that $\displaystyle{\frac{x_{11}x_{14}}{x_{10}x_{15}}}$ coincides with $\displaystyle{\frac{x'_{11}x'_{14}}{x'_{10}x'_{15}}}$. 
\[
\frac{x'_{11}x'_{14}}{x'_{10}x'_{15}} = \frac{\left(\u1_4\u1_1\right)^{-1}x_{11}\left(\u1_{2}\u1_{3}\u1_{4}\right)^{-1}x_{14}}{\left(\u1_{4}\u1_{2}\u1_{1}\right)^{-1}x_{10}{\left(\u1_3\u1_4\right)^{-1}x_{15}}} = \frac{x_{11}x_{14}}{x_{10}x_{15}}. 
\]
\end{itemize} 

\end{example}

\bibliographystyle{plain}
\bibliography{bibliography.bib}

\end{document}